\documentclass[a4paper,12pt]{amsart}
\usepackage{amsmath}
\usepackage{amsfonts}
\usepackage{amssymb}
\PassOptionsToPackage{colorlinks=true, linkcolor=oblue, urlcolor=magenta}{hyperref}

\usepackage{amsthm}
\usepackage{mathtools,hyperref}
\usepackage{tikz,tikz-cd,tikz}
\usetikzlibrary{shapes,arrows}
\usepackage{cleveref}
\usepackage{float}
\usepackage{xcolor}
\usepackage[edges]{forest}
\usepackage{colortbl} 
\usetikzlibrary{fit,positioning}
\usepackage{makecell}

\theoremstyle{plain}
\newtheorem{theorem}{Theorem}[section]

\newtheorem{lemma}[theorem]{Lemma}
\newtheorem*{claim*}{Claim}

\newcommand\Z{\mathbb{Z}}

\newcommand\N{\mathbb{N}}

\theoremstyle{definition}
\newtheorem{remark}[theorem]{Remark}
\newtheorem{fact}[theorem]{Fact}
\newtheorem{definition}[theorem]{Definition}
\newtheorem{example}[theorem]{Example}
\newtheorem{question}[theorem]{Question}    
\newtheorem{proposition}[theorem]{Proposition}
\newtheorem{corollary}[theorem]{Corollary}
\newtheorem{observation}[theorem]{Observation}

\usetikzlibrary{math} 
\usepackage{xfp}
\newcommand{\boxtwo}[2]
{
\begin{tikzpicture}
  \draw[draw=black] (0.5,-0.5) rectangle ++ (2,1);
  \node[below] at (1.5,-0.5) {#1};
  \foreach \i in {1,2}{
        \shade[ball color = gray!40, opacity = 0.4] (\i,0) circle (0.25 cm);
        \draw (\i,0) circle (0.25cm);
        \tikzmath{\t={\i+#2-1};}
        \node at (\i,0)  {\fpeval{\t}}; }
\end{tikzpicture}}
\newcommand{\boxone}[2]
{
\begin{tikzpicture}
  \draw[draw=black] (0.5,-0.5) rectangle ++ (2,1);
  \node[below] at (1.5,-0.5) {#1};
        \shade[ball color = gray!40, opacity = 0.4] (1.5,0) circle (0.25 cm);
        \draw (1.5,0) circle (0.25cm);
        \tikzmath{\t={1+#2-1};}
        \node at (1.5,0)  {\fpeval{\t}}; 
\end{tikzpicture}}
\definecolor{oblue}{rgb}{0.0, 0.37, 0.75}
\usepackage{enumerate}
\usepackage{geometry}
\newtheorem{conjecture}[theorem]{Conjecture}
\begin{document}
\title[Box Progressions, Abelian Power-free Morphisms and A Sieve Technique]{Box Progressions, Abelian Power-free Morphisms and A Sieve Technique for the Template Method}
\author[S. Ey\.{I}do\u{g}an]{Sad{\i}k Ey\.{I}do\u{g}an}
\address{
Department of Mathematics, Faculty of Science and Literature, \c{C}ukurova University, 
01330 Adana, Turkey}
\email{seyidogan@cu.edu.tr}	

\author[H. G\"{o}ral]{Haydar G\"{o}ral}
\address{
	Department of Mathematics, Izmir Institute of Technology, 35430 Urla, Izmir,
Turkey}
\email{haydargoral@iyte.edu.tr}

\author[N. Tanısalı]{N\.{I}han Tanısalı}
\address{
	Centre de recherche INRIA Saclay }
\address{
        \'Ecole Polytechnique, Institut Polytechnique de 
        Paris, 1 rue Honor\'e d’Estienne d’Orves, 91120 Palaiseau Cedex
	France}
\email{nihan.tanisali@inria.fr}

\date{\today}

\thanks{2020 $Mathematics\ subject\ classification.$
 68R15, 68R05, 11B25.}  

\begin{abstract}   
Given balls and boxes both enumerated by the positive integers, we consider a sequential allocation of the balls into the boxes. We fix $\ell \ge 2 $. Proceeding in increasing order of box labels, assign to each box the next $r$ smallest balls for some $ 1\leq r\leq \ell$.  
Given an integer $k \ge 3$, is there a natural number $N$ such that in any placement of $N$ balls into boxes, there exist 
$k$ balls whose labels and box labels each form a $k$-term arithmetic progression?
We address this question by identifying abelian power-free fixed points of morphisms over a binary alphabet. We present sufficient conditions under which a morphism is abelian $k$-power-free. Our conditions extend  Dekking's result over a binary alphabet and offer a weaker, yet more effective alternative to Carpi’s. Combining Dekking's result with the template method of Currie and Rampersad, we develop a sieve technique that significantly reduces the number of parents that must be examined to establish abelian power-freeness. We then identify a binary morphism that is abelian 16-power free (but not abelian $15$-power free) with an abelian 14-power free fixed point, demonstrating the strength of our technique in verifying abelian power-freeness.
Furthermore, we give a binary morphism which is not abelian power-free, yet has an abelian $5$-power free fixed point.
These results offer novel examples of morphisms whose fixed points exhibit stronger abelian power-freeness than the corresponding morphisms.
\end{abstract}

\maketitle
\pagestyle{headings}
\tableofcontents
\section{Introduction}
An arithmetic progression of length $k$, or a $k$-term arithmetic progression, or a $k$-AP for short, is a sequence of natural numbers $a_1, a_2, \dots, a_k$ 
where the difference between any two consecutive terms is constant. If the common difference is $d$, then we have
$a_i = a_1 + (i-1)d$ for all $i = 1, \dots, k$. If $d=0$, then the arithmetic progression is called trivial.  We say that a subset $A$ of the positive integers contains arbitrarily long arithmetic progressions if it contains a non-trivial $k$-AP for any $k \ge 3.$ The question of whether arithmetic progressions are present in large sets has led to some of the most famous results in mathematics. In 1927, van der Waerden \cite{Van} made a fundamental contribution by proving that for any partition of the positive integers into a finite number of classes, there are arbitrarily long arithmetic progressions in at least one class.\\ 

Following this, extensive efforts were made to understand the arithmetic structure in large sets. In 1936, Erd\H{o}s and Tur\'an \cite{Erd} conjectured that any subset of the positive integers with positive upper density contains arbitrarily long arithmetic progressions. Formally, the upper density of a set $A \subseteq \mathbb{Z}_{>0}$ is defined by
$$ \overline{\delta}(A) = \limsup_{N \to \infty} \frac{|A \cap \{1, \dots, N\}|}{N}. $$
This conjecture, known as the Erd\H{o}s--Tur\'an conjecture, implies that positive upper density is sufficient to ensure the presence of arithmetic progressions, and it connects the size of a set with its internal structure. 
The first breakthrough on this conjecture was made by Roth \cite{Rth} in 1953, who proved that subsets of the positive integers with positive upper density contain infinitely many non-trivial 3-term arithmetic progressions. This was a significant step forward, but the more general case remained unresolved until 1975. Szemerédi \cite{Szemere} provided a complete proof for all lengths. More precisely, he obtained that any subset of the positive integers with positive upper density contains arbitrarily long arithmetic progressions. Szemer\'edi's theorem stands as a monumental achievement, not only confirming the Erd\H{o}s-Tur\'an conjecture, but also expanding our understanding of the intrinsic order within the positive integers.
Later on, Szemer\'edi's theorem was reproved by Furstenberg \cite{HF} in 1977 using ergodic theory and by Gowers \cite{Gow} in 2001 using Fourier analysis. The ergodic theoretic proof gave rise to a generalization of Szemer\'edi's theorem, namely the multidimensional Szemer\'edi theorem, see \cite{FK}.
To state this, we need the corresponding notion of upper density. For a subset $A \subseteq (\mathbb{Z}_{>0})^d$, its upper density is defined by
$$ \overline{\delta_d}(A) = \limsup_{N \to \infty} \frac{|A \cap \{1, \dots, N\}^d|}{N^d}.$$
The multidimensional Szemer\'edi theorem asserts that if $A$ is a subset of $(\mathbb{Z}_{>0})^d$ with positive upper density, then for any finite subset $F$ of $(\mathbb{Z}_{>0})^d$ there exist a positive integer $a$ and an element $b \in (\mathbb{Z}_{>0})^d$
such that $aF+b$ is a subset of $A$ (an affine copy of $F$ lies in $A$).\\

We are now ready to present the setting of the problem of interest in this work. Fix a (possibly infinite) collection of balls and boxes, both enumerated by the positive integers. We distribute the balls over the boxes according to the following  rules:

 \begin{enumerate}
     \item There are no empty boxes.
     \item There exists a positive integer $\ell \in \mathbb{Z}_{>0} $ such that each box has at most $\ell$ balls.
     \item For any $i \in \Z_{>0}$, each ball in the $i$-th box is smaller than the balls in the $(i+1)$-th box. 
 \end{enumerate}
For a fixed $\ell$, a distribution satisfying (1)–(3) is called an $\ell$-ball–box distribution. A ball-box distribution is an $\ell$-ball-box distribution for some $\ell \ge 2.$
More mathematically, we can see an $\ell$-ball-box distribution as a surjection 
$$ \mathcal{D} : \mathbb{Z}_{>0} \to \mathbb{Z}_{>0}$$
which maps the positive integer $i$ to $\mathcal{D}(i)$ where $\mathcal{D}(i)$ is the label of the box that contains the ball $i$. The conditions for a surjection corresponding to an $\ell$-ball-box distribution are: 
\begin{enumerate}
\item $ \mathcal{D}^{-1} (i)$ is non-empty and has size at most $\ell$ for all $i\in\mathbb{Z}_{>0}$,
\item $\mathcal{D}(i) \leq \mathcal{D} (i+1)$ for all $i \in \mathbb{Z}_{>0}$.
\end{enumerate}
One example of such a distribution where $ \ell=2$ starts as follows: 
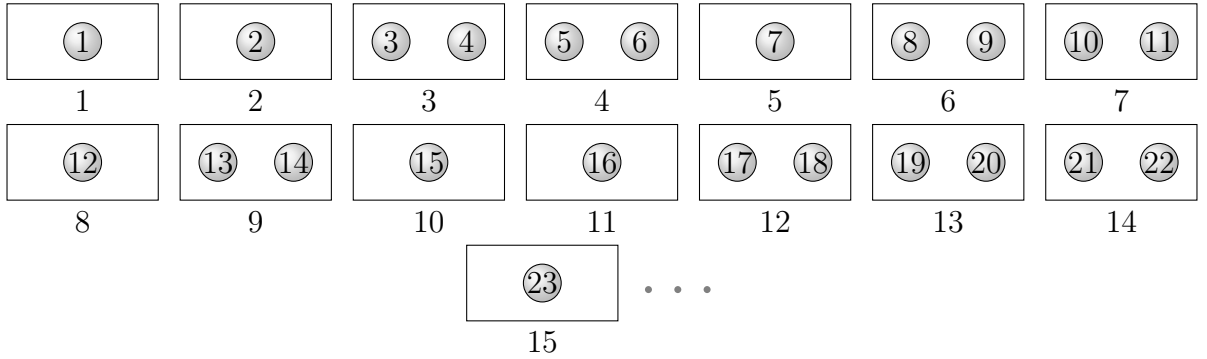
\begin{figure}[H]\label{Figure example ball-box}
\begin{center}
\boxone{1}{1}
\boxone{2}{2}
\boxtwo{3}{3}
\boxtwo{4}{5}
\boxone{5}{7}
\boxtwo{6}{8}
\boxtwo{7}{10}
\boxone{8}{12}
\boxtwo{9}{13}
\boxone{10}{15}
\boxone{11}{16}
\boxtwo{12}{17}
\boxtwo{13}{19}
\boxtwo{14}{21}
\boxone{15}{23}
\begin{tikzpicture}
  \draw[draw=white] (0.5,-0.5) rectangle ++ (1,1);
  \node[below] at (1,-0.5) {};
  \node[gray,below] at (1,0.5) {\huge{$\cdots$}};
\end{tikzpicture}
\end{center}
\caption{An example of a 2-ball-box distribution}
\end{figure}
Both
the balls $3,8,13,18,23$ and their box numbers $ 3,6,9,12,15$ are 5-term arithmetic progressions, respectively. To refer to these type of progressions, we give the following definition.
\begin{definition}
Given an $\ell$-ball-box distribution,
    a sequence of positive integers $ h_1, \ldots, h_k$ is called a $k$-term box progression ($k$-BP in short) if
    \begin{enumerate}
        \item $ h_1,\ldots,h_k$ is a non-trivial $k$-term arithmetic progression,
        \item $ b_1,\ldots, b_k$ is a $k$-term arithmetic progression (possibly a trivial progression) where $ h_i$ belongs to the $b_i$-th box for all $1\leq i\leq k$.
    \end{enumerate}
\end{definition}
Now, we state the main question of this paper below and then we link it to combinatorics on words.
\begin{question}\label{main question}
    For any $\ell$-ball-box distribution $ \mathcal{D} $, are there arbitrarily long BP's for $ \mathcal{D} $ with respect to the above rules?
\end{question}
Note that the set 
\begin{align*}
    \{( n, \mathcal{D} (n)): n\in \Z_{>0}  \} \subseteq \Z_{>0} \times \Z_{>0}
\end{align*}
has zero upper density. Therefore, the multidimensional Szemer\'edi theorem \cite{FK} is not applicable to answer this question. This question is non-trivial starting from $\ell  \geq 2$. 
Our approach is to use combinatorics on words as this question is closely related to the repetition-free word problem, see Proposition \ref{bp-word}. 
These types of problems are known as avoidability problems in the literature. A classical avoidability problem is the following:
does there exist an infinite word over a finite alphabet which does not contain consecutive
repetitions of non-empty factors (squares)? At the beginning of the last
century, Thue \cite{Thue06, Thue12} proved  that there exist sequences on 3 symbols which
contain no 2 identically equal consecutive factors (square-free), and there exist sequences on 2 symbols which
contain no 3 identically equal consecutive factors (cube-free). In 1961, Erd\H os \cite{erd61} asked whether there exist infinite abelian square-free words over a given finite alphabet of size at least 4. An abelian square is a finite non-empty word $ uv $, where $ u $ and
$ v $ are permutations (anagrams) of each other. Evdokimov \cite{Evdomi} and Pleasants \cite{Pleasants} gave solutions for alphabets of
sizes 25 and 5, respectively. Dekking \cite{dekking} proved that over a binary alphabet there exist infinite words that are
abelian 4-power free. The problem of whether abelian squares can be avoided over a four-letter alphabet had been open for a long time. In 1992, using a combination of computer checking and mathematical techniques,  Ker\"{a}nen \cite{Keranen} proved the Erd\H os conjecture on the avoidance of abelian squares for some infinite words over a four-letter alphabet. 
In 2007 and 2008, Currie and Visentin \cite{C-V-2007, C-V-2008} constructed the first example of a binary pattern that is abelian 2-avoidable despite containing no abelian 4-power. Additionally, they established that any binary pattern of length greater than 118 is abelian 2-avoidable. \\

The results on avoidability problems show that the intriguing case of the problem mentioned above is when $ \ell = 2 $. While the first impression might suggest a positive answer to our problem, a closer examination reveals that the answer is actually negative. To put it more clearly, we show that there is a $2$-ball-box distribution for which all box progressions are of length at most 6. 
Our problem is actually equivalent to whether certain types of infinite words on binary alphabets have the property of being abelian power-free. 
Let $\Sigma=\{a, b\}$ be a binary alphabet.
Our goal is to avoid abelian powers while also avoiding $bb$ and some power of $a$. For this purpose, we prove that there exists an abelian 6-power free morphism $h$ on $\Sigma$ with an infinite fixed point $ \Omega_1 $ such that (see Theorem \ref{4 consecutive a 6AP})
\begin{itemize}
	\item $ aaaaa $ is not a factor of $ \Omega_1 $  (there are at most $4$ consecutive $a$'s in $ \Omega_1 $),
		\item $ bb $ is not a factor of $ \Omega_1 $  (there are no two consecutive $b$'s in $ \Omega_1 $),
		\item $ \Omega_1 $ is abelian $ 6 $-power free (as $h$ is abelian 6-power free).	
\end{itemize}
\vspace{2mm}

In the study of infinite words over a finite alphabet, the property of being abelian power-free is well understood, with three primary approaches which decide whether a morphism or its fixed points possess this property. These include Dekking's result \cite{dekking}, Carpi's conditions \cite{Carpi93}, and template methods \cite{C-R2012, Rao-R}. In this manuscript, we present two new methods that advance the field by extending and refining these approaches.\\

Our first contribution is a new set of sufficient conditions for determining whether a morphism is abelian $k$-power free, see Theorem \ref{our result}.
While our new set of conditions extends Dekking's result over a binary alphabet, it remains less robust but more effective than Carpi's method, striking a balance between simplicity and effectiveness. We also consider Carpi's conjecture and its relation to abelian power-freeness of fixed points.

\begin{conjecture}[Carpi]
    The conditions in Theorem \ref{carpi} are necessary and sufficient for a morphism to be abelian power-free over alphabets of size less than six.
\end{conjecture}

Our second contribution is the analysis of abelian power-freeness in the fixed points of morphisms. Drawing inspiration from Dekking’s approach and the template method of Currie and Rampersad, we develop a technique to verify whether an infinite fixed point of a morphism avoids abelian powers, even when the morphism itself does not. 
We introduce a sieve technique that immensely reduces the number of parents that need checking. To illustrate this, our technique reduces the number of parents and ancestors of the morphism in Proposition~\ref{SievingT8} from 1953162 and 1953289 to 8 and 224, respectively.
Using this sieve technique, we identify a binary morphism that is abelian 16-power free (but not abelian $15$-power free), while its fixed point is abelian 14-power free. 
Moreover, we find a binary morphism which is not abelian power-free, but has an abelian $5$-power free fixed point. In fact, by improving upon the aforementioned infinite word $ \Omega_1 $,
we prove that there exists an infinite word $ \Omega$ on $\Sigma=\{a, b\}$ such that
\begin{itemize}
		\item $ bb $ is not a factor of $ \Omega $  (there are no two consecutive $b$'s in $ \Omega $),
		\item $ \Omega $ is abelian $ 5 $-power free.	
\end{itemize}
These results provide novel examples of morphisms whose fixed points exhibit stronger abelian power-freeness than the corresponding morphisms.\\

Table \ref{label d1} presents our results regarding the avoidability of abelian powers, as well as of the factors $bb$ and some powers of $a$ in the fixed points of certain morphisms on $\Sigma=\{a, b\}$. 
\begingroup
\tiny{
	\begin{table}[H]
		\caption{Avoided abelian powers for the binary morphisms and their fixed points from our main results\\}
		\centering 
            \renewcommand{\arraystretch}{1.5} 
		\begin{tabular}{|c |c| c|c| }
			\hline 
			\makecell[c]{Results} & $h$ & $ h^{\omega}(a) $ & { \makecell[c]{Avoided\\ Abelian \\ Power\\ $(i)$ Morphism\\ $(ii)$ Fixed Point} } \\[0.5ex] 
			\hline 
                \hline
			Theorem \ref{4 consecutive a 6AP}& $ \begin{array}{lcl} 
			a &\mapsto& aaaba \\  b  &\mapsto& bab
			\end{array} $ & $ a a a b a a a a \ \underline{b a a} \ \underline{a a b } \  \underline{a b a} \ \underline{b a a} \ \underline{a b a} \ a a a b  \cdots$ & $\begin{array}{c} (i)\ 6 \\ (ii) \ 6 \end{array}$ \\ 
			\hline 	\makecell[c]{Theorem \ref{3 consecutive a 9AP} \\and Proposition \ref{SievingT8} } & $ \begin{array}{lcl} 
			a &\mapsto& abaaaba\\  b  &\mapsto& babab  
			\end{array} $ & $ \begin{array}{lcl} 
			a b a a a b a b a b a b a b a a a b a a b a a a  \ \underline{b a a b a}   \ \underline{ a a b a b} \\  \underline{a b a b a} \ \underline{ b a a a b} \  \underline{a b a b a } \ \underline{ b a b a a}  \ \underline{a b a b a} \  b a b a b a \cdots			 
			\end{array} $ & $\begin{array}{c}(i) \ 9 \\ (ii)\ 9 \end{array}$ \\
			\hline 	\makecell[c]{Theorem \ref{2 consecutive a 16AP}\\and Theorem \ref{T14-template}} & $ \begin{array}{lcl} 
			a &\mapsto& abaabaababa \\  b  &\mapsto& babababab  
			\end{array} $ & $ \begin{array}{lcl} 
			a b a a b a a b a b a b a b a b a b a b a b a a b a a b a b a a b a  a b a a b a b a b a b a b a \\ b  a b a b a a b a a \
			\underline{b a b a a b a a b} \ 
			\underline{a a b a b a b a b} \
			\underline{a b a b a b a b a} \
			\underline{a b a a b a b a b} \\
			\underline{a b a b a b a b a} \
			\underline{b a a b a a b a b} \
			\underline{a b a b a b a b a} \
			\underline{b a b a a b a a b} \
			\underline{a b a b a b a b a} \\
			\underline{b a b a b a a b a} \
			\underline{a b a b a b a b a} \
			\underline{b a b a b a b a a} \
			\underline{b a a b a b a b a}\
			b a b a b a 
			\cdots
			\end{array} $ &  $\begin{array}{c} (i) \ 16 \\ (ii) \ 14 \end{array}$ \\
			\hline 
			 Theorem \ref{5 consecutive a 6AP} & $ \begin{array}{lcl} 
			a &\mapsto& ababa \\  b  &\mapsto& aaaaba  
			\end{array} $ & $ abab \ \underline{a} \ \underline{a} \ \underline{a} \ \underline{a} \ \underline{a} \ baaba  \cdots$ & $\begin{array}{c} (i) \ 6 \\ ii) \ 6 \end{array}$\\ 
			\hline 
			Theorem \ref{4 consecutive a 8AP}& $ \begin{array}{lcl} 
			a &\mapsto& aaabaaaba \\  b  &\mapsto& babaaaba  
			\end{array} $ & $ \begin{array}{lcl} 
			&aaabaaabaaaabaaabaaaabaaababa \ \underline{baaa} \ \underline{baaa} \ \underline{abaa} \\ & \underline{abaa} \ \underline{aaba} \ \underline{aaba} \  \underline{aaab} \ \underline{aaab} \ ababaaa  \cdots
			\end{array} $ & $\begin{array}{c} (i) \ 9 \\ (ii) \ 9 \end{array}$ \\
			\hline
				Theorem \ref{3 consecutive a 11AP} & $ \begin{array}{lcl} 
			a &\mapsto& abaaabaaaba \\  b  &\mapsto& ababababa  
			\end{array} $ & $ \begin{array}{lcl} 
			a  b  a  a  a  b  a  a  a  b  a  a  b  a  b  a  b  a  b  a  a  b  a  a  a  b  a  a  a  b  a  a  b  a  a  a  b  a  a  a  b  a  a \\  b  a  a  a  b  a  a  a  b  a  a  b  a  b   a  b  a  b  a  a  b  a  a  a  b  a  a  a  b  a  a  b  a  a  a  b  a  a  a  b  a  a  b  \\ a  a  a  b  a  a  a  b  a  a  b  a  b  a  b  a  b  a  a  b  a  a  a  b  a  a  a  b  a  a  b  a  a  a  b  a  a  a  b  a  a  b  a  \\b  a   b  a  b  a  a  b  a  a  a  b  a  a  a  b  a  a  b  a  b  a  b  a  b  a  a  b  a  a  a  b  a a  a  b  a  a  b  a  b  a  b  a    \\     \underline{b  a  a  b  a  a  a  b  a  a  a  b  a  a  b  a  b  a  b  a  b  a  a  b  a  a  a  b  a  a  a  b  a  a  b  a  a  a }    \\     
			\underline{b  a  a  a  b  a  a  b  a  b  a  b  a  b  a  a  b  a  a  a  b  a  a  a  b  a  a  b  a  a  a  b  a  a  a  b  a  a  }    \\       
			\underline{b  a  a  a  b  a  a  a  b  a  a  b  a  b  a  b  a  b  a  a  b  a  a  a  b  a  a  a  b  a   a  b  a  a  a  b  a  a  }    \\               
			\underline{a  b  a  a  b  a  a  a  b  a  a  a  b  a   a  b  a  b  a  b  a  b  a  a  b  a  a  a  b  a  a  a  b  a  a  b  a  a  }    \\               
			\underline{a  b  a  a  a  b  a  a  b  a  b  a  b  a  b  a  a  b  a  a  a  b  a  a  a  b  a  a  b  a  a  a  b  a  a  a  b  a  } \\
			\underline{ a  b  a  a  a  b  a  a  a  b  a  a  b  a  b  a  b  a  b  a  a  b  a  a  a  b  a  a  a  b  a  a  b  a  a  a  b  a }     \\             
			\underline{a  a  b  a  a  b  a  a  a  b  a  a  a  b  a   a  b  a  b  a  b  a  b  a  a  b  a  a  a  b  a  a  a  b  a  a  b  a }      \\             
			\underline{a  a  b  a  a  a  b  a  a  b  a  b  a  b  a  b  a  a  b  a  a  a  b  a  a  a  b  a  a  b  a  a  a  b  a  a  a  b  }    \\              
			\underline{a  a  b  a  a  a  b  a  a  a  b  a  a  b  a  b  a  b  a  b  a   a  b  a  a  a  b  a  a  a  b  a  a  b  a  a  a  b }            \\
			\underline{a  a  a  b  a  a  b  a  a  a  b  a  a  a  b  a  a  b  a  b  a  b  a  b  a  a  b  a  a  a  b  a  a  a  b  a   a  b  }   \ 
			a  a  a  b  a  a  \cdots \vspace{0.7cm}		 
			\end{array} $ & $\begin{array}{c}(i) \ 11\\ (ii) \ 11 \end{array}$ \\
			\hline 
            \hspace{0.4cm} Theorem \ref{5free}\hspace{0.4cm} & $ \begin{array}{lcl} 
			a &\mapsto& aaaab \\  b  &\mapsto& ababab
			\end{array} $ &  $  \underline{a}\ \underline{a}\ \underline{a}\  \underline{a} \ b\ aaaab \ aaaab \ aaaab \ ababab \ aaaab   \cdots$ & $\begin{array}{c}  (i) \ \infty\\ (ii) \ 5 \end{array}$ \\ 
			\hline 
		\end{tabular}
		\label{label d1}
    \end{table}
}
\endgroup

The results of Section \ref{Uniform Section}, summarized in Table \ref{section 4 table}, establish  uniform bounds for binary words that avoid $bb$, small powers of $a$, and small abelian powers. Together with Table \ref{label d1}, this provides a complete resolution of the $2$-ball-box distribution problem.
\begingroup
\small
\begin{table}[H] 
\centering
\caption{ Summary of Section \ref{Uniform Section}: On the Uniform Bounds of Lengths of Words Containing Abelian Powers: Computer-Assisted Results  }
\label{table:imp}
\begin{tabular}{|c|c|c|c|}
\hline
\hspace{0.6cm} Results \hspace{0.6cm}&Avoided factors & Avoided abelian $k$-power &  \hspace{0.25cm}Length of the longest word  \hspace{0.25cm} \\
\hline\hline
Figure \ref{fig: word 9} & $bb$ & $ 3$& $8$  \\
Proposition\ref{p41}  & $bb$ & $4$& $17$\\
Proposition \ref{aaa,bb,25,5power} &$aaa$, $bb$ & $ 5$& $24$ \\
Proposition \ref{aaaa,bb,78,5power} &$aaaa$, $bb$ & $ 5$& $77$  \\
\hline
\end{tabular}
\label{section 4 table}
\end{table}
\endgroup

\textbf{A Short Outline of the Paper:} In Section \ref{Sec2}, we will present the necessary definitions related to combinatorics on words. Then, we will see the relation between box progressions with respect to a ball-box distribution and the abelian power factors of an infinite word. We also give our first main theorem (Theorem \ref{our result}) which yields a new set of conditions for a morphism to preserve abelian $ k $-power free words, and these conditions are effective. Using this theorem, we will construct three morphisms over the binary alphabet $\{a, b\}$ that have fixed points avoiding the factor $bb$ while being abelian 6, 9, and 16-power free, respectively. In Section \ref{Sieve Section}, we will develop a sieve technique for the template method to verify whether a fixed point of a morphism avoids abelian powers. Moreover,
we will give an improved version of the inverse parent lemma (Lemma \ref{new search bound}) which is essential due to the high powers we aim to analyze (abelian 8-powers in Proposition \ref{SievingT8} and abelian 14-powers in Theorem \ref{T14-template}).
Then, using this sieve technique, we will identify a binary morphism that is abelian 16-power free but not abelian $15$-power free, while its fixed point is abelian 14-power free. 
In Section \ref{Uniform Section}, we will focus on constructing the longest possible words over a binary alphabet that avoid certain factors while remaining abelian power-free.
In Section \ref{not_abelian_power_free}, we will obtain a binary morphism which is not abelian power-free, yet has an abelian $5$-power free fixed point.
In Section \ref{some_applications}, we will explore three different applications of the ball-box distribution problem by highlighting its relevance to combinatorics on words.  Finally, in Section \ref{open problems}, we will give some open problems emerging from our results. 
\section{From Ball-Box Distributions to Words} \label{Sec2}

In what follows, we refer the reader to Lothaire's and Shallit's books \cite{Lot, Shallit}. Let $\Sigma$ be a finite set (a finite alphabet). The elements of $\Sigma$ are called letters. We let $\Sigma^*$ and $\mathbb Z^\Sigma$ denote to the free monoid and the free $\mathbb Z $-module generated by $\Sigma$, respectively. An element of $\Sigma^*$ is called a word. The empty word is denoted by $\epsilon$. An infinite word over the alphabet $\Sigma$ is a non-ending sequence of elements of $\Sigma$, that is, a map from $\mathbb N$ to $\Sigma$. We represent the set of infinite words by $\Sigma^{\mathbb N}$. For the following definitions, we list the elements in $\Sigma$ as $ \{ \sigma_1,\ldots ,\sigma_m\}$.\\

We denote the concatenation of two words $ u \in \Sigma^* $ and $ v \in \Sigma^* $ by $ uv. $ For example, given two words $ u = aba $ and $ v=bbba $ over the binary alphabet $\Sigma=\left\lbrace a,b\right\rbrace$, the concatenation of $u $ with $v$ is $ uv=ababbba. $\\

A word $u$ is called a factor of another word $v$ if there exist words $ s, t\in \Sigma^*$ such that $v= sut$. The factor $u$ is called a prefix if $ s= \epsilon$, and a suffix if $ t=\epsilon$. For a word $u\in \Sigma^* $, we set $ \operatorname{Pref} ( u) = \{ v \in \Sigma^* : v \text{ is a prefix of }u \}$ and $ \operatorname{Suff} ( u) = \{ v \in \Sigma^* : v \text{ is a suffix of }u \}$. \\

Let $ u \in \Sigma^*$ be a word over the alphabet $\Sigma$. The Parikh vector of $u$ is denoted by $$\psi (u)=(u_1, \ldots, u_m) \in \N^{m},$$ and it is defined as the vector of the number of appearances of $ \sigma_i $'s in $u$, that is to say, $u_i$ is the number of appearances of $ \sigma_i $ in $u$. The length of $u$ is denoted by $|u|$, and we have $|u|=\displaystyle\sum_{i=1}^{m}u_i$. 
Recall that two finite words $u$ and $v$ from $\Sigma^*$ are said to be abelian equivalent and we write $u \sim v $, if $u $ is obtained from $v$ by permuting the letters (they are anagrams of each other). 
In other words, $\psi(u)= \psi(v)$.\\

Let $\Sigma_1$ and $\Sigma_2$ be two alphabets with $m_1$ and $m_2$ elements, respectively. A morphism from  $\Sigma_1^*$  to $\Sigma_2^*$ is a map $h$ such that for every word $u=vw \in \Sigma_1^*,$ one has
$h(u)=h(v)h(w)$.
The frequency matrix of $ h$ is an $m_1\times m_2$ matrix whose $ i$-th row is given by the vector $\psi (h(\sigma_i)) $. For a morphism $ h: \Sigma_1^* \to \Sigma_2^* $, we define \begin{align*}
\operatorname{Pref}{(h)} &:= \bigcup_{ \sigma\in \Sigma_1} \operatorname{Pref}(h(\sigma)), \\
\operatorname{Suff}{(h)} &:= \bigcup_{ \sigma\in \Sigma_1} \operatorname{Suff}(h(\sigma)). 
\end{align*}
The set of Parikh vectors of $h$ is denoted by $\psi(\operatorname{Pref}{(h)}).$

\begin{definition} Let $k$ be a positive integer.
    An abelian $k$-power free word, or a word avoiding abelian $k$-powers,  is a word $u$ (finite or infinite) containing no $k$ consecutive factors that are permutations of one another.
     In other words, if $u$ is written as $u= xv_1\cdots v_k y$ where $v_1, \ldots,v_k$ are words with  $\psi(v_1) = \cdots =\psi(v_k)$, then $ v_1= \cdots = v_k=\epsilon$.
 \end{definition}

\begin{definition}
Let $\Sigma_1$ and $\Sigma_2$ be two alphabets and $n$ be a positive integer. A morphism $h$ from  $\Sigma_1^*$  to $\Sigma_2^*$ is said to be abelian $n$-power free if for any abelian $n$-power free word $u$, the word $h(u)$ is also abelian $n$-power free.
\end{definition}

\begin{definition}
Let $G$ be an abelian group. An arithmetic progression of length $k$ in $G$ is a sequence of the form
$$ a, a+d, \ldots ,a+(k-1)d $$
for some $a,d\in G$. If $d=0$, then the arithmetic progression is called trivial. 
Let $A$ be a subset of $G$. 
Define $ \operatorname{a-rk}(A)$ to be the length of a longest non-trivial arithmetic progression in $A$. If this length is unbounded, we let  $ \operatorname{a-rk} (A)=\infty$.
\end{definition}
For example, consider $G =\Z/6\Z $ and $A= \{\overline{0},\overline{3}\}$. Then $\operatorname{a-rk} (A) =\infty $ as the sequence $ \overline{0},\overline{3},\overline{0},\overline{3} ,\ldots$ is an arithmetic progression of infinite length. If $p$ is a prime and $ A $ is a proper subset of $G=\Z/p\Z$, then $\operatorname{a-rk } (A) \leq p-1$. \\

Now, we will give the relation between ball-box distributions and infinite words.
Any ball-box distribution $\mathcal{D}$ is uniquely identified with an infinite word $w_\mathcal{D}$ over the binary alphabet $\Sigma=\{a,b\}$ using the following rules: 
\begin{enumerate}
    \item The $ i $-th letter is $a$ if $ i$ and $i+1$ are in the different boxes.
    \item The $ i $-th letter is $b$ if $ i$ and $i+1$ are in the same box.
\end{enumerate}
For example, the ball-box distribution in Figure \ref{Figure example ball-box} corresponds to the infinite word:
\begin{align*} 
    aababaababaabaaabababa \ldots
\end{align*}
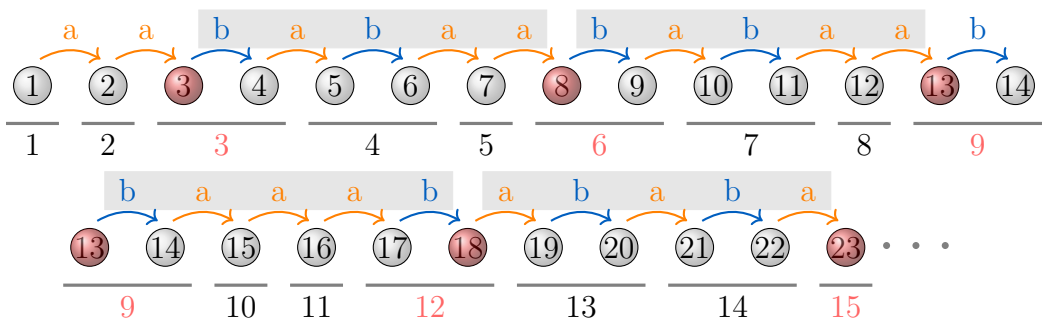
\begin{figure}[H]\label{ball box distribution figure}
\begin{tikzpicture}
\definecolor{oblue}{rgb}{0.0, 0.37, 0.75}
  \node[below] at (1,-0.5) {1};
  \node[below] at (2,-0.5) {2};
  \node[below,red!60] at (3.5,-0.5) {3};
  \node[below] at (5.5,-0.5) {4};
  \node[below] at (7,-0.5) {5};
  \node[below,red!60] at (8.5,-0.5) {6};
  \node[below] at (10.5,-0.5) {7};
  \node[below] at (12,-0.5) {8};
  \node[below,red!60] at (13.5,-0.5) {9};
  \foreach \i in {1,...,14}{
        \shade[ball color = gray!40, opacity = 0.4] (\i,0) circle (0.25 cm);
        \draw (\i,0) circle (0.25cm);
        \node at (\i,0)  {\fpeval{\i}};
        }

        \shade[ball color = red, opacity = 0.4] (3,0) circle (0.25 cm);       
        \shade[ball color = red, opacity = 0.4] (8,0) circle (0.25 cm);   
        \shade[ball color = red, opacity = 0.4] (13,0) circle (0.25 cm);   
\foreach \i in {1,2,7,12}{
  \draw [very thick,gray] (-0.35+\i,-0.5) to (0.35+\i,-0.5);}
  \foreach \i in {3,5,8,10,13}{
  \draw [very thick,gray] (-0.35+\i,-0.5) to (0.35+\i+1,-0.5);}
  
  \draw [fill=gray,gray!20] (3.2,0.5) rectangle (7.8,1); 
  \draw [fill=gray,gray!20] (8.2,0.5) rectangle (12.8,1);

  \foreach \i in {1,2,4,6,7,9,11,12}{
  \draw [->,thick,orange] (\i+0.1,0.35) to [out=30,in=150] node[above] {a}(\i+0.9,0.35);}
  \foreach \i in {3,5,8,10,13}{
  \draw [->,thick,oblue] (\i+0.1,0.35) to  [out=30,in=150] node[above] {b} (\i+0.9,0.35);}
\end{tikzpicture}
\quad
\quad
\begin{tikzpicture}
\definecolor{oblue}{rgb}{0.0, 0.37, 0.75}
\node[below,red!60] at (13.5,-0.5) {9};
\node[below] at (15,-0.5) {10};
\node[below] at (16,-0.5) {11};
\node[below,red!60] at (17.5,-0.5) {12};
\node[below] at (19.5,-0.5) {13};
\node[below] at (21.5,-0.5) {14};
\node[below,red!60] at (23,-0.5) {15};
    \draw [fill=gray,gray!20] (13.2,0.5) rectangle (17.8,1); 
    \draw [fill=gray,gray!20] (18.2,0.5) rectangle (22.8,1);
    \foreach \i in {13,...,23}{
        \shade[ball color = gray!40, opacity = 0.4] (\i,0) circle (0.25 cm);
        \draw (\i,0) circle (0.25cm);
        \node at (\i,0)  {\fpeval{\i}};
        }
        \shade[ball color = red, opacity = 0.4] (13,0) circle (0.25 cm);   
        \shade[ball color = red, opacity = 0.4] (18,0) circle (0.25 cm);   
        \shade[ball color = red, opacity = 0.4] (23,0) circle (0.25 cm);   
      
    \foreach \i in {13,17,19,21}{
    \draw [very thick,gray] (-0.35+\i,-0.5) to (0.35+\i+1,-0.5);}
    \foreach \i in {15,16,23}{
    \draw [very thick,gray] (-0.35+\i,-0.5) to (0.35+\i,-0.5);}
    \foreach \i in {14,15,16,18,20,22}{
    \draw [->,thick,orange] (\i+0.1,0.35) to [out=30,in=150] node[above] {a}(\i+0.9,0.35);}
    \node[gray] at (24,0)  {\huge$\cdots$}; 
    \foreach \i in {13,17,19,21}{
    \draw [->,thick,oblue] (\i+0.1,0.35) to  [out=30,in=150] node[above] {b} (\i+0.9,0.35);}
\end{tikzpicture}
\caption{An example of the correspondence between an abelian 4-power and a 5-BP}
\end{figure}
Figure \ref{ball box distribution figure} illustrates the correspondence between the abelian $4$-power
$$\underline{babaa} \ \underline{babaa} \ \underline{baaab} \ \underline{ababa} $$ in the word and the $5$-BP of the balls $3,8,13,18,23$ which lie in the boxes with numbers $3,6,9,12,15$, respectively. In the following proposition, we prove that there is a correspondance between abelian $(k-1)$-powers in the infinite word and $k$-BP's in the corresponding ball-box distribution. 

\begin{proposition}\label{bp-word}
Let $k \ge 2$ be a positive integer.
    For a ball-box distribution $\mathcal{D}$, the corresponding word $w_\mathcal{D}$ via the given set of rules is abelian $(k-1)$-power free if and only if there is no $k$-BP with respect to $\mathcal{D}$.
\end{proposition}
\begin{proof}
    Assume that there are $k-1$ consecutive blocks in $w_ \mathcal{D}$ that are permutations of each other. Let $d$ be the block size and $i$ be the coordinate of the first letter in the first aforementioned block. Let $A$ and $B$ be the number of $a$'s and the number of $b$'s in a block, respectively. Thus, $d=A+B$.
    Consider the following balls given by $\mathcal{D}$: $$i, i+d, \ldots, i+(k-1)d.$$
    These balls form a $k$-AP. Let $ i$ be in the $ j$-th box. Then, the $(i+ d)$-th ball is in the $(j + A)$-th box. Similarly, the $(i+ md)$-th ball is in the $ (j+ mA)$-th box for each $1\leq m \leq k-1$. This yields a $k$-BP.
    
Conversely, assume that there is a $k$-BP given by $\mathcal{D}$ containing the balls $$i, i +d,\ldots , i +(k-1)d .$$
Let these balls lie in the boxes $ j, j+A , \ldots, j+ (k-1)A$, respectively. Put $B=d-A$. In $ w_\mathcal{D} $, consider the $k-1$ consecutive blocks of size $d=A+B$ starting from the $i$-th index. Since $i+md $ is in $j+ mA$ for $1\leq m \leq k-1$, there are $A$ many $a$'s and $B$ many $b$'s in the block whose coordinates lie between $i$ and $i+d-1$. This is the case for all aforementioned $k-1$ blocks.
The above discussion indicates that $k$-BP's given by $\mathcal{D}$ and abelian $(k-1)$-powers in the infinite word $w_\mathcal{D}$ are in one-to-one correspondence.  
\end{proof}

\subsection{\texorpdfstring{Deciding Whether a Morphism is Abelian $n$-Power Free}{Deciding Whether a Morphism is Abelian n-Power Free}}
In the literature, abelian power-free words are commonly produced by iterating abelian $n$-power free morphisms for some $n \geq 2$.  Thus, an effective procedure to decide whether a given morphism is abelian $ n $-power free 
is very useful to produce words avoiding abelian powers.\\

Dekking \cite{dekking} gave sufficient conditions for a morphism to be abelian $ n $-power free. One of the weak points of this result is that it does not work efficiently for abelian square-free morphisms and the condition $(D4)$ below is a tight one. Now, we state Dekking's elegant result.

\begin{theorem}\cite[Lemma]{dekking} \label{dekkinglemma}
	Let $n>1$, $h:\Sigma^*\to \Sigma^* $ be a morphism. Let $G$ be a finite abelian group and $f:\Sigma^*\to G$ be a monoid morphism such that:
	
	\begin{enumerate}[\indent ({D}1)]
		\item The frequency matrix of $h $ is non-singular,
		\item For each $\sigma_i \in \Sigma $, we have $f(h (\sigma_i) )=0$,
		\item The set $ A= f (\operatorname{Pref} (h) )= \{g\in G: g=f(v), v\text{ is a prefix of $h(\sigma_i)$ for some $\sigma_i \in \Sigma$} \} $ does not contain a non-trivial arithmetic progression of length $n+1$, in other words $\operatorname{a-rk(A)} \leq n$,
		\item The morphism $f$ is $h$-injective. That is to say, for any positive integer $r$, if $$ f( v_1) =\cdots =f(v_r)$$ with $v_i \in \operatorname{Pref}(h)\setminus h(\Sigma)$ for $i \in \{1,\ldots,r\}$, then either $ v_1 =\cdots=v_r$ or $v_1^\prime=\cdots= v_r^\prime$ where $ v_i v_i^\prime = h(\sigma_{j_{i}})$ for some $\sigma_{j_{i}} \in \Sigma$.
	\end{enumerate}
Then, the morphism $h$ is an abelian $n$-power free morphism.
\end{theorem}

In the same paper \cite[Theorem 1]{dekking}, Dekking also proved that
the morphism $h:\Sigma^*\to \Sigma^* $ is abelian 4-power free, where $\Sigma=\{a,b\}$ is a binary alphabet and $h$ is given by
\begin{align} \label{dekm}
	h(a) = abb \ \  \text{and} \ \ h(b) = aaab.
\end{align}
From this, one immediately concludes that the infinite word generated by $a$ under $h$ is  abelian 4-power free. \\

In 1993, Carpi \cite{Carpi93} constituted an effective and sufficient set of conditions for abelian $ n $-power freeness of a morphism (see also \cite{Rao15} for a generalization). This set of conditions is weaker than Dekking's result (Theorem \ref{dekkinglemma}), in other words, if a morphism satisfies $(D1)-(D4)$ from Theorem \ref{dekkinglemma}, then the morphism also satisfies $(C1)-(C3)$ from Theorem \ref{C93} below. 
Hence, Carpi's result extends Dekking's result.
Carpi's theorem has two strong aspects. First, it can decide more effectively whether a given morphism is abelian square-free, and which is not the case for Dekking's result. For another limitation of Dekking's result, see Proposition \ref{nonexample for 2 consecutive a}. Carpi's result even allows us to verify that the morphism of Ker\"{a}nen \cite{Keranen} is abelian square-free.
Second, this set of conditions turns out to be a complete characterization of abelian $ n $-power free morphisms defined on alphabets with six or more letters. It is left as a conjecture that Carpi's set of abelian power-free morphism conditions \cite{Carpi93} is necessary and sufficient for morphisms to be abelian power-free for alphabets of size less than six.  Now, we state Carpi's remarkable result.

\begin{theorem} \label{C93}\cite[Proposition 1]{Carpi93} \label{carpi}
	Given an integer $n \geq 2$, two alphabets $\Sigma_1$ and $\Sigma_2$ and a morphism $h: \Sigma^*_1 \rightarrow \Sigma^*_2$, the subgroup of $\mathbb{Z}^{\Sigma_2}$ generated by $\psi(h(\Sigma_1))$ is denoted by $G_h$. Suppose the following conditions are satisfied:
\begin{enumerate}[\indent ({C}1)]
	
		\item  For any abelian $n$-power free word $z$ of length 2,  the word 
 $h(z)$ is abelian $n$-power free.
		
		\item $h$ is commutatively bijective, that is, for any $u,v$ from $\Sigma^*_1 $, if $\psi(h(u))=\psi(h(v))$, then
  $\psi(u)=\psi(v).$
		\item For every $(n+1)$-tuple 
  $ \left( x_0,x_1,\dots,x_n\right) $ in $$ \prod_{j=0}^{n}\bigg( \operatorname{Pref}\left(h\left(a_j\right)\right) \setminus \{h(a_j)\}\bigg),$$
  where $ a_j \in \Sigma_1 $ and
		$$
		\begin{gathered}
		\psi\left(x_{j+1}\right)-2 \psi\left(x_j\right)+\psi\left(x_{j-1}\right) \in G_h \\
		\ \text{for} \ j=1,2, \ldots, n-1,
		\end{gathered}
		$$
		there exists an $(n+1)$-tuple $ \left( \delta_0,\delta_1,\dots,\delta_n\right)  $ in $ \left\lbrace 0,1\right\rbrace^{n+1} $ such that
		$$
		\begin{aligned}
		\psi\left(x_{j+1}\right)-2 \psi\left(x_j\right)+\psi\left(x_{j-1}\right) & =\delta_{j+1} \psi\left(h\left(a_{j+1}\right)\right)-2 \delta_j \psi\left(h\left(a_j\right)\right)+\delta_{j-1} \psi\left(h\left(a_{j-1}\right)\right) \\
		\ \text{for} \ j & =1,2, \ldots, n-1 .
		\end{aligned}
		$$
	\end{enumerate}
	Then, the morphism $h$ is an abelian $n$-power free morphism.
\end{theorem}

\subsection{\texorpdfstring{Deciding Whether a Fixed Point is Abelian $k$-Power Free}{Deciding Whether a Fixed Point is Abelian k-Power Free}}
Up to this point, we focused on the results on the abelian power-freeness of morphisms. In this subsection, our focus shifts to deciding abelian power-freeness of the fixed points of a morphism. Given a morphism $h: \Sigma^* \rightarrow \Sigma^*$ with an infinite fixed point $ w $ and an integer $ k \ge 2, $ decide if $ w $ is abelian $ k $-power free. In 2012, Currie and Rampersad \cite{C-R2012} showed that for a morphism $ h $ satisfying certain  general conditions, this problem is decidable.  This novel approach is known as the template method. For a generalization, we direct the reader to \cite{Rao-R}. The template method leverages the regular structure of morphisms, providing a more theoretically efficient and comprehensive framework for analysis compared to existing methods. This method requires the use of computer programs for its implementation. However, when $k$ is large, its computational complexity might be a drawback in deciding whether a fixed point is abelian $k$-power free.

\begin{definition}
Let $ A \in \mathbb{R}^{m \times m} $ be a matrix. The operator norm (or induced norm) of $ A $ is defined as	$$ \|A\| = \sup_{\substack{x \in \mathbb{R}^m \\ x \neq 0}} \frac{\|Ax\|_2}{\|x\|_2}, \  \ \text{where}  \ \|\cdot\|_2 \  \text{denotes the Euclidean norm on }\mathbb{R}^m.$$ 
\end{definition}
Now, we state the template method.
\begin{theorem}\cite[Theorem 6]{C-R2012} \label{templatem}
	Let $h$ be a morphism on $\Sigma=\left\lbrace \sigma_1,\sigma_2,\dots,\sigma_m\right\rbrace $, let $ M $ be the frequency matrix of $ h $ and let $k \ge 2$ be an integer. Suppose that
	
	\begin{enumerate}[\indent ({T}1)]
		\item $ h(\sigma_1)=\sigma_1x, $ for some non-empty word $ x $ in $ \Sigma^*, $
		\item $|h(\sigma_i)| > 1$ for each $ i \in \left\lbrace 1,\dots,m\right\rbrace,  $
		\item $M$ is non-singular and
  the operator norm of the inverse of  $M$ is less than 1, that is to say, $\|M^{-1}\| < 1$.
	\end{enumerate}
Let $w$ be the infinite word generated by $\sigma_1$ under $h$, namely $w=h^{\omega}(\sigma_1)$. Then, it is decidable whether $ w $ is abelian $ k $-power free.
\end{theorem}
\subsection{\texorpdfstring{A New Set of Conditions to Decide Whether a Morphism is Abelian $n$-Power Free}{A New Set of Conditions to Decide Whether a Morphism is Abelian n-Power Free}}
Next, we provide a set of effective and sufficient conditions for a morphism to preserve abelian $ k $-power free words. Building on foundational approaches, our set of conditions not only provides a more efficient enhancement to Dekking's result over a binary alphabet, but also addresses certain limitations.
Although Carpi's result is more general, checking the condition $(C3)$ may take longer. Similarly, the template method provides a super-powerful tool for fixed points but it often requires extensive computational resources. 
Our set of conditions stands out for being both theoretically well-founded and practical, offering an accessible approach for deciding abelian power-free morphisms.

\begin{theorem}\label{our result}
	Given an integer $n \geq 4$, an alphabet $\Sigma$ of size at least 2 and a morphism $h: \Sigma^* \rightarrow \Sigma^*$, suppose the subgroup of $\mathbb{Z}^{\Sigma}$ generated by $\psi(h(\Sigma))$ is denoted by $G_h$.
  For any element $e \in \mathbb{Z}^{\Sigma}$, the element $\overline{e} \in \mathbb{Z}^{\Sigma} / G_h$
 represents the equivalence class of $e$ with respect to $G_h$, namely $\overline{e}=\{d \in \mathbb{Z}^{\Sigma}: e-d \in G_h \}$.
 Let  $\mathcal{P} = \psi (\operatorname{Pref}{(h)})$
 and 	$ \overline{\mathcal{P}} = \{\overline{v}: v\in \mathcal{P} \}$.  Suppose that the following conditions are satisfied:
	\begin{enumerate}[\indent ({O}1)]
		\item The frequency matrix of $ h $ is non-singular,
		\item $\operatorname{a-rk} (\overline{\mathcal{P}}) \leq n$,
		\item For $v\in \mathcal{P} \setminus (\psi(h(\Sigma)) \cup \{0\})$, the
        equivalence class $[v]_\mathcal{P}= \{ w \in \mathcal{P}: w-v \in G_h
        \}$ of $v$ in $\mathcal{P}$ contains at most $2$ elements,
        and the Parikh vector of a non-empty proper prefix of $h$ is not in
        the equivalence class of $0$, where $0$ is the zero vector in
        $\mathbb{Z}^{\Sigma}$. Moreover, if $[v]_\mathcal{P}=\{v,w\}$
        with $v \neq w$, then
        $$ v-w=\psi(h(p))-\psi(h(q)) $$
        for unique $p,q \in \Sigma$ where $v \in \psi(\operatorname{Pref}(h(p)))$ and
        $w \in \psi(\operatorname{Pref}(h(q)))$.
	\end{enumerate}
	Then, $h$ is an abelian $n$-power free morphism.
	\end{theorem}

\begin{proof}
	We will make use of Carpi's theorem (Theorem \ref{C93}). For this purpose, first let $z$ be a word of length $2$. Clearly, $z$ is an abelian $n$-power free word as $n\geq 4$. We will show that $h(z)$ is abelian $n$-power free as well. Write $z=ab$ for some $a,b \in \Sigma$.  Suppose that 
 $$ h(z) =h(a) h(b) =x v_1 \cdots v_n y$$ where  $\psi(v_i)= \psi(v_j)$ for $i,j \in \{1,\ldots, n\}$. Then, without loss of generality, we may write 
 $$h(a)= xv_1\cdots v_{i-1} v_i ^\prime \ \ \text{and}\ \  h(b)= v_i ^{\prime \prime}v_{i+1} \cdots v_ny$$ for some $i$ where $v_i =v_i ^{ \prime}v_i ^{\prime \prime} $.
We set $\alpha =\overline{\psi(x)}$ and $d=\overline{\psi(v_1)} =\cdots =\overline{\psi(v_n)} $. Note that $\alpha, \alpha +d,\ldots ,\alpha +(i-1)d \in \overline{\mathcal{P}}$. Let $d_1= \overline{\psi{(v_i ^\prime )}}$. Then $ d_2 =\overline{\psi(v_i ^{\prime\prime})}=d-d_1$. Observe that 
	$$\overline{0} = \alpha + (i-1)d+d_1=\alpha+i d-d_2 , $$
	in other words $d_2=\alpha+id $. Notice also that $$ \alpha +i d=d_2,d_2+d,\ldots, d_2+(n-i)d \in \overline{\mathcal{P}}.$$ This yields that $\alpha, \ldots, \alpha+id, \ldots, \alpha +nd $ are all in $\overline{\mathcal{P}}$. As $\operatorname{a-rk }(\overline{\mathcal{P}}) \leq n$ by $(O2)$, we deduce that $d=\overline{0}$. Therefore, $ \psi(x),\psi(xv_1),\ldots, \psi (xv_1 \ldots v_{i-1}) $ are all in $\mathcal{P}$ and in the same equivalence class given by $\Z^\Sigma/ G_h$. 
 Suppose that $i \ge 3$ and the words $x, xv_1, xv_1v_2$ are all distinct. By $(O3)$, we have that $\psi(x), \psi(xv_1), \psi(xv_1v_2)$ are in $\psi(h(\Sigma)) \cup \{0\}$.
If $x=\epsilon$, then $\psi(x)=0$
and $\psi(xv_1v_2)=2\psi(xv_1)$.
This contradicts the non-singularity of the frequency matrix of $h$ by $(O1)$ as $\psi(xv_1), \psi(xv_1v_2)$ are in $\psi(h(\Sigma))$. 
If $x \neq \epsilon$, set
$\psi(xv_1)=\psi(h(p_1))=\psi(x)+\psi(v_1)$ and
$\psi(xv_1v_2)=\psi(h(p_2))=\psi(x)+2\psi(v_1)$ for some $p_1, p_2$ from $\Sigma$. This yields that the vectors 
$\psi(x), \psi(xv_1), \psi(xv_1v_2)$ from $\psi(h(\Sigma))$ are linearly dependent. This again contradicts the non-singularity of the frequency matrix of $h$ by $(O1)$. Thus $i \le 2$ or $v_1= \cdots= v_n =\epsilon .$ If the second situation holds, then we are done.
Therefore, suppose that $i\leq 2$. By a similar argument to the above,
$$ \psi(v_i^{\prime\prime}), \psi(v_i^{\prime\prime} v_{i+1}) ,\ldots , \psi(v_i^{\prime\prime}v_{i+1} \ldots v_n ) $$ are all in $\mathcal{P}$ and are in the same equivalence class given by $\Z^\Sigma/ G_h$. 
If $n-i+1 \ge 3$ and the words 
$v_i^{\prime\prime}, v_i^{\prime\prime} v_{i+1}, v_i^{\prime\prime} v_{i+2}$ are all distinct, then as above, this case will contradict the non-singularity of the frequency matrix of $h$ by $(O1)$. As a result, we have two cases. The first case is $n-i+1 \le 2$ and the second case is 
$v_1= \cdots= v_n =\epsilon .$ 
The first case is impossible as $n \ge 4$ and $i \le 2$.
Therefore, we deduce that $$v_1= \cdots= v_n =\epsilon .$$ Hence, $ h(z)$ is abelian $n $-power free and we have $(C1)$ from Theorem \ref{C93}.\\
	
Secondly, $h$ is commutatively bijective as its frequency matrix is non-singular by $(O1)$. Thus, we obtain $(C2)$ from Theorem \ref{C93}.\\

	Lastly, we will show that $(C3)$ from Theorem \ref{C93} is satisfied to finish the proof. Let $$(x_0 ,\ldots, x_n ) \in \prod_{j=0}^{n}\bigg( \operatorname{Pref}\left(h\left(a_j\right)\right) \setminus \{h(a_j)\}\bigg),$$ 
 where $ a_j \in \Sigma.$
 Suppose that
 $$\psi( x_{j+1}) -2\psi(x_j) + \psi(x_{j-1} ) \in G_h $$
	for $j=1,\ldots ,n-1$. One infers that $$\overline{\psi (x_{j-1}) }, \ \overline{\psi (x_{j}) }, \ \overline{\psi (x_{j+1})}  $$
 is a $3$-AP in $\overline{ \mathcal{P} }$.
 Therefore, $$\overline{\psi (x_{0}) } ,\ldots,\overline{\psi (x_{n}) }$$
 is an arithmetic progression of length $n+1$ in $\overline{ \mathcal{P} }$.
 As $\operatorname{a-rk} (\overline{\mathcal{P}})\leq n$ by $(O2)$, we conclude $$\overline{\psi (x_{0}) } =\cdots=\overline{\psi (x_{n}) }.$$ This means that $\psi(x_0),\ldots,\psi(x_n) \in \mathcal{P}$ are in the same equivalence class. 
 Put $K=\{\psi(x_0),\ldots, \psi(x_n) \}.$ If the size of $K$ is at least 3, then we see that 
$$x_0= \cdots= x_n =\epsilon,$$
and we can choose $\delta_j=0 $ in $(C3)$ for $j\in \{1,\ldots,n-1\}$.
So, we may assume that $x_i \neq \epsilon$ for any $i$ from $\{0,\ldots,n\}$.
 Therefore, the size of $K=\{\psi(x_0),\ldots, \psi(x_n) \} $ is at most $2$. 
	If $|K|=1$, then one may take $\delta_j=0 $ in $(C3)$ for $j\in \{1,\ldots,n-1\}$. Suppose that $|K| =2$  and write $ K= \{v,w\}$ where $v\neq w$. By the assumption, $$ v-w =\psi(h(p))-\psi (h(q))$$ for $p,q\in \Sigma$ where $v\in \psi(\operatorname{Pref} (h(p)))  \ \text{and} \ w\in \psi(\operatorname{Pref}(h(q)))$. Now for any $j\in\{1,\ldots,n-1\}$, notice that
	$$ \psi(x_{j+1} )-2 \psi(x_{j} )+\psi(x_{j-1} )\in \{0,v-w,w-v, 2v-2w,2w-2v\}. $$
	In any case, \begin{align*}
	\psi(x_{j+1} )-2 \psi(x_{j} )+\psi(x_{j-1} )= \psi(h(a_{j+1})) -2\psi(h(a_j)) + \psi(h(a_{j-1}))
	\end{align*}
	for suitable $a_{j+1},a_{j},a_{j-1} \in\{p,q\}$ and $\delta_j=1 $ for $j\in \{1,\ldots,n-1\}$. We conclude that Carpi's last condition $(C3)$ is satisfied, and the proof is now complete.
\end{proof}

Next, we prove that our previous theorem extends Dekking's result (Theorem \ref{dekkinglemma}) over a binary alphabet.

\begin{theorem} \label{extdek}
Let $n \ge 4$, $\Sigma$ be a binary alphabet,
$G$ be a finite abelian group and $f:\Sigma^*\to G$ be a monoid morphism. Assume that there is a morphism $h:\Sigma^*\to \Sigma^*$ satisfying $(D1)-(D4)$ from Theorem \ref{dekkinglemma}. Then, the morphism also satisfies $(O1)-(O3)$ from Theorem \ref{our result}. 
\end{theorem}

\begin{proof}
To give the essence of the proof, we consider the following diagram:
	
	$$
	\begin{tikzpicture}[auto]
	\node (G) {$G$};
	\node (F) [left of= G ,xshift=-4 cm ] {$\Sigma^*$};
	\node (P) [below of= F, yshift =-1.5cm] {\small{$\N^2 = \psi (\Sigma^*) \leq \Z^2$}};
	\node (Z) [below of =P , yshift=-1.5cm] {$\Z^{2 } /G_h $} ;
	\draw[->] (F)  -- node[midway,above] {$f$} (G);
	\draw[->] (P) -- node[midway,above] {$\tilde f$} (G);
	\draw[->] (F) -- node[midway,left] {$\psi$} (P);
	\draw[->] (P) -- node[midway,left] {$\pi$} (Z);
	\draw[->] (Z) -- node[midway,above] {$\overline f$} (G);
	\end{tikzpicture}
	$$

	We define $ \tilde f :\psi( \Sigma^* ) \to G $ such that $ \Tilde{f} \circ \psi = f$. This is possible since $ G$ is abelian and $ f$ is a morphism. Also, we extend $ \tilde f$ to  
 $\Z^2$ so that it becomes a group homomorphism from $\Z^2$ to $G$.
 We define the map $\pi (x_1,x_2) = \overline{(x_1,x_2)} $ where $\overline{(x_1,x_2)}$ denotes the class of the tuple $(x_1,x_2)$ in $ \Z^2 /G_h$.  We define $\overline{f} $ such that it makes the above diagram commute, which is possible as $$ G_h\subseteq \ker \tilde f .$$
	Now, we will show that $(O1)-(O3)$ from Theorem \ref{our result} are satisfied. Firstly, $(D1)$ is $(O1)$. 
	To see that $(O2)$ holds, let 
 $$ \overline{ \psi (x_0)},\ldots, \overline{\psi(x_n)} $$ be an arithmetic progression of length $n+1$ in $\overline{\mathcal{P}}$ where
 
 $$(x_0 ,\ldots, x_n ) \in \prod_{j=0}^{n}\operatorname{Pref}(h\left(a_j\right)),$$
 and $a_j$ from $\Sigma.$ We will prove that $\overline{\psi(x_i)}$'s form a trivial arithmetic progression.
 Then, we have that
	$\overline{f} ( \overline{\psi (x_0)}) , \ldots , \overline{f} ( \overline{\psi (x_n)}) $ is an arithmetic progression in $G$. As $ \overline{f} ( \overline{\psi (x_i)}) =f(x_i )$ for $i \in\{0,\ldots,n \}$, we have an $(n+1)$-AP given by $f(x_i)$'s. By $(D3)$, one obtains that $f(x_0) = \cdots = f(x_n)$. Now by $(D4)$, one of the following cases hold:
	\begin{enumerate}
		\item $ x_0 =\cdots =x _n$ or
		\item $ x_0 ^{\prime} =\cdots = x_n ^\prime$ or
		\item $x_i\in h(\Sigma) $ for some $0\le i\le n$.
	\end{enumerate}  In the first case, $ \overline{ \psi (x_0)}=\cdots= \overline{\psi(x_n)} $. In the second case,
 $$  \psi(x_i )- \psi(x_j)= \left(\psi(x_i )+ \psi (x_i^\prime)\right) - ( \psi(x_j )+ \psi (x_j^\prime)) \in G_h, $$
 and we obtain the desired equality. In the third case, we have $\overline{\psi(x_i)}=\overline{0}$ for all $0 \le i\le n$, hence the arithmetic progression is trivial.  \\
 
Finally, we will show that $(O3)$ holds. For this purpose, suppose that $$[v]_\mathcal{P}= \{ w \in \mathcal{P}: w-v \in G_h \}$$   contains $\{u,v,w\}$ where
$v\in \mathcal{P} \setminus (\psi(h(\Sigma)) \cup \{0\})$ and $\mathcal{P} = \psi (\operatorname{Pref}{(h)})$. 
There is $a\in \Sigma=\{a,b\}$ and there are two of $ u,v,w$ which come from the prefixes of $h(a) $ as $|\Sigma| =2$. Without loss of generality, say $u $ and $v$. Assume that $ ||u||_2 \geq ||v||_2$. There exist $u_1,v_1 \in \operatorname{Pref}(h (a)) \setminus \{h(a)\}$ such that $ \psi (u_1)=u$ and $\psi(v_1)=v$. This means that $h(a) = v_1 xy$ where $ u_1= v_1x$ and $ \psi(x) \in G_h$. We deduce that $f(u_1) =f(v_1) $. By $(D4)$, the previous equality implies either $u_1=v_1 $ or $ xy=u_1^\prime =v_1 ^\prime=y$. In either case $x=\epsilon$. 
Hence, $u=v$ and we conclude that $[v]_{\mathcal{P}}$ has at most two elements. To see the last part of $(O3)$, suppose that 
$[v]_{\mathcal{P}}=\{v,w\}$ for some $v \in \mathcal{P} \setminus (\psi(h(\Sigma)) \cup \{0\})$ and $v$ and $w$ are distinct. As before, there are $v_1 \in  \operatorname{Pref}(h (a)) \setminus \{h(a)\}$ and $w_1 \in  \operatorname{Pref}(h (b)) \setminus \{h(b)\}$ such that $\psi(v_1)=v$ and 
$\psi(w_1)=w$. As $v,w$ are in the same equivalence class, we conclude that $f(v_1)=f(w_1)$. By $(D4)$, either $v_1=w_1$ or $v_1^\prime=w_1^\prime$ where $h(a)=v_1 v_1^\prime$ and $h(b)=w_1 w_1^\prime.$ The first case is impossible as $v$ and $w$ are distinct. Thus, we get that $v_1^\prime=w_1^\prime$. Hence, 
$$
\psi(h(a))-v=\psi(v_1^\prime)=\psi(w_1^\prime)=
\psi(h(b))-w,
$$
which in turn yields that $v-w=\psi(h(a))-\psi(h(b)).$
 \end{proof}

\begin{remark}

An important question in the literature is whether a $k$-power free morphism is necessarily $(k + 1)$-power free. 
 To clarify, a word is said to be $k$-power free if it does not contain any factor of the form $u^k$, where $u$ is a non-empty word, and $u^k$ denotes the $k$ consecutive repetitions of $u$. Besides, a morphism $f$ is said to be $k$-power free if for any word $u$, $f(u)$ is
$k$-power free whenever the word $u$ is so. 
 The above question was addressed as Conjecture 1 in \cite{Ric}. In the case of a binary alphabet, a positive answer was provided by Leconte \cite[Theorem 12.3.1]{leconte}, who showed that for any $k \geq 3$, if a binary morphism is $k$-power free, then it is $(k + n)$-power free for all integers $n\ge1$. However, for alphabets containing at least three letters, the answer is not always positive. For example, in \cite{bean}, it was shown that there exists a morphism 
	$$ f =   \begin{dcases}
	& a \mapsto abacbab \\
	& b \mapsto cdabcabd \\
	& c \mapsto cdacabcbd \\
	& d \mapsto cdacbcacbd \\
	\end{dcases}$$
	that is 2-power free but not 3-power free, as demonstrated by the fact that $f(aa)$ contains a 3-power.
	This question becomes significant in the context of abelian power-freeness. As seen from the criteria provided in our main result, if a morphism is decided to be abelian $k$-power free using Theorem \ref{our result}, then it is necessarily abelian $(k + n)$-power free for all positive integers $n$. This result highlights  the effectiveness of Theorem \ref{our result} (also of Theorem \ref{carpi}) in deciding the abelian power-freeness of morphisms. 
	\end{remark}

\noindent \textbf{Question:} Before applying our results, 
an intriguing question is the following. Over a binary alphabet, if a morphism satisfies  $(C1)-(C3)$ from Theorem \ref{C93}, does it also satisfy  $(O1)-(O3)$ from Theorem \ref{our result}? Under Carpi's conjecture, does the set of conditions in Theorem \ref{our result} give a complete characterization of abelian $ n $-power free morphisms defined on binary alphabets?

\subsection{Abelian Power-Free Morphisms and Their Fixed Points}
In this subsection, we exhibit abelian power-free morphisms whose fixed point $\Omega$ avoids the factor $bb$ and some power of $a$. This yields an answer to the 2-ball-box distribution problem by Proposition \ref{bp-word} (see Section \ref{not_abelian_power_free}).\\

Recall that the question whether an $\ell$-ball-box distribution $\mathcal{D}$ yields a $k $-BP becomes meaningful when $ \ell \geq 2$. In the realm of words, $\ell=2$ means that in the infinite word $w_\mathcal{D} $ there are no two consecutive $b$'s. Therefore, from now on, we will consider this type of words. Before giving our results, we need the following fact, the so-called Smith decomposition.

\begin{fact} (Smith decomposition \cite[Chapter 3]{Jac})
For any matrix $M \in M_{n}(\mathbb Z)$, there exist $U$, $D$ and $V$ in $M_{n}(\mathbb Z)$ such that
\begin{itemize}
    \item $D=(d_{ij})_{ij}$ is diagonal, that is $d_{ij}=0$ if $i \neq j$,
    \item The determinants of $U$ and $V$ are 1,
    \item $M=UDV.$
\end{itemize}
\end{fact}
Using the Smith decomposition, one can obtain the  following well-known result.

\begin{corollary} \label{rankh}
 Let $L$ be a free abelian subgroup of $\mathbb Z^{n}$ of rank $n$. Suppose that $\{v_1,\ldots,v_n\}$ is a basis of $L$. Then
$$
| \mathbb Z^{n}/L | =|\det (v_1,\ldots,v_n)|.
$$  
\end{corollary}

\begin{remark} \label{r214}
  Let $L$ be a free abelian subgroup of $\mathbb Z^{n}$ of rank $n$ and $\{v_1,\ldots,v_n\}$ be a basis of $L$. We also call $L$ a lattice in $\mathbb Z^{n}$.  Given  an element $x=(x_1,\ldots,x_n) \in \mathbb Z^{n}$, one can effectively check whether $x \in L$. Suppose that $z=x^T$ is the column vector of $x$. Let $A$ be the $n \times n$ matrix whose $i$-th column is given by $v_i$ as a column vector for $i \in \{1,\ldots,n\}$. Put $d=| \mathbb Z^{n}/L |=|\det (v_1,\ldots,v_n)|$ and let $B=\operatorname{adj}(A)$ be the adjoint matrix of $A$ whose row vectors are given by $R_i$ for $i \in \{1,\ldots,n\}$. We know that
$A^{-1}=\frac{1}{\det(A)}B$.
  Now, the following are equivalent:
  \begin{enumerate}
\item  $x \in L$,
\item There is a column vector $y$ with coordinates in $\mathbb Z$ such that $Ay=z$,
\item There is a column vector $y$ with coordinates in $\mathbb Z$ such that $y=\frac{1}{\det(A)}B z$,
\item There exist $y_i \in \mathbb Z$ for $i \in \{1,\ldots,n\}$ such that
$y_i=\frac{R_i \cdot x}{\det(A)}$,
\item For $i \in \{1,\ldots,n\}$, 
$$
R_i \cdot x \equiv 0 \ (\text{mod} \ d).
$$
\end{enumerate}
Hence, one sees that $x \in L$ can be effectively checked by the fifth equivalence above. We will also apply this equivalence in the following three results.
\end{remark}

Now, we are ready to give our avoidability results on words using Theorem \ref{our result}.
Note that Dekking's result (Theorem \ref{dekkinglemma}) is not applicable for the morphisms in Theorem \ref{4 consecutive a 6AP}, Theorem \ref{3 consecutive a 9AP} and Theorem \ref{2 consecutive a 16AP}.

\begin{theorem}\label{4 consecutive a 6AP}
There exists an abelian $6$-power free morphism $h$ over the binary alphabet $\Sigma= \left\lbrace a,b\right\rbrace  $ with a fixed point $\Omega_1$ such that
	\begin{itemize}
		\item $ aaaaa $ is not a factor of $ \Omega_1 $  (there are at most $4$ consecutive $a$'s in $ \Omega_1 $),
		\item $ bb $ is not a factor of $ \Omega_1 $  (there are no two consecutive $b$'s in $ \Omega_1 $),
		\item $ \Omega_1 $ is abelian $ 6 $-power free.		
	\end{itemize}
\end{theorem}
\begin{proof}
	Consider the morphism $h: \Sigma^* \to \Sigma^ * $ given by
	\begin{align*}
	&h(a) = aaaba\\
	&h(b) = bab.
	\end{align*}
	The frequency matrix of $h $ is $$ \begin{pmatrix}
	4 & 1\\1& 2
	\end{pmatrix}$$ which is non-singular with determinant $7$. 
 Let $G_h$ be the subgroup of $\mathbb{Z}^{2}$ generated by $\psi(h(\Sigma))=\{(4,1), (1,2)\}$.
 Hence, $ \mathbb{Z}^{2}/G_h \cong \Z/7\Z $ by Corollary \ref{rankh}. 
 By Remark \ref{r214}, or one can directly check that an element $(x,y)$ from $ \mathbb{Z}^{2}$ is in $G_h$ if and only if
\begin{equation} \label{mod7}
    2x-y \equiv 0 \ (\text{mod} \ 7) \ \ \text{and} \ -x+4y \equiv 0 \ (\text{mod} \ 7).
\end{equation}
 Notice that the set of Parikh vectors of $h$ is 
 $$\mathcal{P} = \{ (0,0), (1,0),(2,0),(3,0),(3,1), (4,1), (0,1),(1,1), (1,2)\}.$$ There are at most two elements in the same equivalence class given by $G_h$ for the elements of $\mathcal{P} \setminus \{(0,0), (4,1), (1,2)\}$. In fact, $(3,0)$ and $(0,1)$ are in the same equivalence class and
 $$(3,0)-(0,1)=(4,1)-(1,2).$$
 All the other elements of $\mathcal{P} \setminus \{(0,0), (4,1), (1,2)\}$ have trivial equivalence classes in $\mathcal{P}$ and this can be obtained by \eqref{mod7}. Therefore, $$\overline{\mathcal{P}}= \{ \overline{(0,0)}, \overline{(1,0)},\overline{(2,0)},\overline{(3,0)}, \overline{(1,1)},\overline{(3,1)} \}.$$ 
The map 
\begin{align*}
	f: \mathbb{Z}^{2}/G_h &\to \Z/7\Z \\
	\overline{(1,0)} &\mapsto \overline{1} 
	\end{align*}  yields an isomorphism with 
  \begin{align*}
    f(\overline{(0,0)}) =\overline{0}, \ f(\overline{(1,0)})=\overline
	1 ,\ f(\overline{(2,0)}) =\overline 2, \ f(\overline{( 3,0) })= \overline{3}, \ f(\overline{(1,1)}) = \overline{4} , \ f(\overline{(3,1)}) =\overline{6}.   
 \end{align*}
This indicates that $\operatorname{a-rk } ( \overline{\mathcal{P}}) = 6$
as $\operatorname{a-rk }(\{\overline{0}, \overline{1},\overline{2},\overline{3},\overline{4}, \overline{6}\})=6$ in $\Z/7\Z$.
By Theorem \ref{our result}, we obtain that $h$ is an abelian 6-power free morphism.
Let $\Omega_1=h^{\omega}(a)$
be the infinite fixed point of $h$ which is obtained by repeatedly applying $h$ to $a$. Therefore, $\Omega_1$ is abelian 6-power free. 
To check the first and the second condition, just note that in $h(a)$ and $h(b)$ there are no $aaaaa$ and $bb$ as factors, and $h(aa)$, $h(ab) $, $h(ba)$ do not produce the factors $aaaaa$ and $bb$ as well. 
Hence, the result follows.	
\end{proof}

\begin{theorem}\label{3 consecutive a 9AP}
    There exists an abelian $9$-power free morphism $h$ over the binary alphabet $\Sigma= \left\lbrace a,b\right\rbrace  $ with a fixed point $\Omega_2$ such that
\begin{itemize}
	\item $ aaaa $ is not a factor of $ \Omega_2 $  (there are at most $3$ consecutive $a$'s in $ \Omega_2 $),
	\item $ bb $ is not a factor of $ \Omega_2 $  (there are no two consecutive $b$'s in $ \Omega_2 $),
	\item $ \Omega_2 $ is abelian $ 9 $-power free.		
\end{itemize}
\end{theorem}
\begin{proof}
	Consider the morphism $h: \Sigma ^* \to \Sigma^ * $ given by
	\begin{align*}
	&h(a) = abaaaba\\
	&h(b) = babab.
	\end{align*}
	The frequency matrix of $h $ is $$\begin{pmatrix}
	5 & 2\\2& 3
	\end{pmatrix}$$ and it is non-singular with determinant $11$.
  Let $G_h$ be the subgroup of $\mathbb{Z}^{2}$ generated by $\psi(h(\Sigma))=\{(5,2), (2,3)\}$.
	Hence, $ \mathbb{Z}^{2}/G_h \cong \Z/11\Z $ by Corollary \ref{rankh}.
  One can infer that an element $(x,y)$ from $ \mathbb{Z}^{2}$ is in $G_h$ if and only if
\begin{equation} \label{mod11}
    3x-2y \equiv 0 \ (\text{mod} \ 11) \ \ \text{and} \ -2x+5y \equiv 0 \ (\text{mod} \ 11).
\end{equation}
 Note that the set of Parikh vectors of $h$ is $$\mathcal{P} =  \{ {(0,0)}, {(1,0)},{(1,1)},{(2,1)}, {(3,1)},{(4,1)} ,{(4,2)}, (5,2), {(0,1)},{(1,2)} ,{(2,2)}, (2,3)\}.$$ 
	There are at most two elements in the same equivalence class given by $G_h$ for the elements of $\mathcal{P} \setminus \{(0,0), (5,2), (2,3)\}.$ More precisely, $(4,1)$ and $(1,2)$ are in the same equivalence class with $$(4,1)-(1,2)=(5,2)-(2,3).$$
 All the other elements of $\mathcal{P} \setminus \{(0,0), (5,2), (2,3)\}$ have trivial equivalence classes in $\mathcal{P}$ and one can see this by \eqref{mod11}. Therefore, $$\overline{\mathcal{P}}= \{ \overline{(0,0)}, \overline{(1,0)},\overline{(1,1)},\overline{(2,1)}, \overline{(3,1)},\overline{(4,1)} ,\overline{(4,2)},\overline{(0,1)},\overline{(2,2)}\}.$$ 
	The map 
	\begin{align*}
	f:  \mathbb{Z}^{2}/G_h &\to \Z/11\Z \\
	\overline{(1,0)} &\mapsto \overline{1 } 
	\end{align*}  
	gives an isomorphism with 
	
 \begin{align*}
     f(\overline{(0,0)}) &=\overline{0}, \ f(\overline{(1,0)})=\overline
	1 ,\ f(\overline {(0,1)}) =\overline 3, \ f(\overline{( 1,1) })= \overline{4}, \ f(\overline{(2,1)}) = \overline{5},\\
 &f(\overline{(3,1)}) = \overline{6} , \ f(\overline{(4,1)}) =\overline{7}, \ f(\overline{(2,2)}) =\overline{8},\ f(\overline{(4,2)}) =\overline{10}.
 \end{align*}
 This signifies that $\operatorname{a-rk } ( \overline{\mathcal{P}}) = 9$ as $\operatorname{a-rk }(\{\overline{0}, \overline{1},\overline{3},\overline{4},\overline{5}, \overline{6}, \overline{7}, \overline{8}, \overline{10}\})=9$ in $\Z/11\Z$.
By Theorem \ref{our result}, we obtain that $h$ is an abelian 9-power free morphism.
Let $\Omega_2=h^{\omega}(a)$
be the infinite fixed point of $h$ which is obtained by repeatedly applying $h$ to $a$. Thus, $\Omega_2$ is abelian 9-power free. 
The first and the second condition follow similarly as in Theorem \ref{4 consecutive a 6AP}.
	\end{proof}

\begin{theorem}\label{2 consecutive a 16AP}
    There exists an abelian $16$-power free morphism $h$ over the binary alphabet $\Sigma= \left\lbrace a,b\right\rbrace  $ with a fixed point $\Omega_3$ such that
\begin{itemize}
	\item $ aaa $ is not a factor of $ \Omega_3 $  (there are at most $2$ consecutive $a$'s in $ \Omega_3 $),
	\item $ bb $ is not a factor of $ \Omega_3 $  (there are no two consecutive $b$'s in $ \Omega_3 $),
	\item $ \Omega_3 $ is abelian $ 16 $-power free.		
\end{itemize}
	\end{theorem}
\begin{proof}
	Consider the morphism $h: \Sigma ^* \to \Sigma^ * $ given by
	\begin{align*}
	&h(a) = abaabaababa\\
	&h(b) = babababab.
	\end{align*}
	The frequency matrix of $h $ is $$\begin{pmatrix}
	7 & 4\\4& 5
	\end{pmatrix},$$ and it is non-singular with determinant $19$.
 Let $G_h$ denote the subgroup of $\mathbb{Z}^{2}$ generated by $\psi(h(\Sigma))=\{(7,4), (4,5)\}$. 
	Hence, $ \mathbb{Z}^{2}/G_h \cong \Z/19\Z $ by Corollary \ref{rankh}. 
  One can deduce that an element $(x,y)$ from $ \mathbb{Z}^{2}$ is in $G_h$ if and only if
\begin{equation} \label{mod19}
    5x-4y \equiv 0 \ (\text{mod} \ 19) \ \ \text{and} \ -4x+7y \equiv 0 \ (\text{mod} \ 19).
\end{equation}
 Observe that the set of Parikh vectors of $h$ is \begin{align*}\mathcal{P}=  \{ ( 0 , 0 ),
	( 1 , 0 ),
	( 1 , 1 ),
	( 2 , 1 ),
	( 3 , 1 ),
	( 3 , 2 ),
	( 4 , 2 ),
	( 5 , 2 ),
	( 5 , 3 ),
	( 6 , 3 ),
    ( 6, 4  ), 
    ( 7, 4  ), \\
	( 0 , 1 ),
	( 1 , 2 ),
	( 2 , 2 ),
	( 2 , 3 ),
	( 3 , 3 ),
	( 3 , 4 ),
	( 4 , 4 ),
	( 4 , 5 )\}.\end{align*}
	There are at most two elements in the same equivalence class given by $G_h$ for the elements of $\mathcal{P} \setminus \{(0,0), (7,4), (4,5)\}$. In fact, $(5,2)$ and $(2,3)$ are in the same equivalence class with $$(5,2)-(2,3)=(7,4)-(4,5),$$ and $ (6,3)$ and $(3,4)$
 are in the same equivalence class  with $$(6,3)-(3,4)=(7,4)-(4,5).$$ Apart from them, all the other elements of $\mathcal{P} \setminus \{(0,0), (7,4), (4,5)\}$ have trivial equivalence classes in $\mathcal{P}$ by \eqref{mod19}.  
 Therefore, \begin{align*} \overline{\mathcal{P}}= \{ \overline{( 0 , 0 )},
	\overline{( 1 , 0 )},
	\overline{( 0 , 1 )},
	\overline{( 1 , 1 )},
	\overline{( 2 , 1 )},
	\overline{( 3 , 1 )},
	\overline{( 1 , 2 )},
	\overline{( 2 , 2 )},
	\overline{( 3 , 2 )},\\
	\overline{( 4 , 2 )},
	\overline{( 5 , 2 )},
	\overline{( 3 , 3 )},
	\overline{( 5 , 3 )},
	\overline{( 6 , 3 )},
	\overline{( 4 , 4 )},
	\overline{( 6 , 4 )}\}.\end{align*}
	The map 
	\begin{align*}
	f: \Z^2 /G_h &\to \Z/19\Z \\
	\overline{(1,0)} &\mapsto \overline{1 } 
	\end{align*}  
	yields an isomorphism with     
 \begin{align*}
        f (\overline {( 0 , 0 )})= \overline{   0   },
	\ f (\overline {( 1 , 0 )})= \overline{   1   },
	\ f (\overline {( 0 , 1 )})= \overline{   3   },&
	\ f (\overline {( 1 , 1 )})= \overline{   4   },
	\ f (\overline {( 2 , 1 )})= \overline{   5   },
	\ f (\overline {( 3 , 1 )})= \overline{   6   },\\
	f (\overline {( 1 , 2 )})= \overline{   7   },
	\ f (\overline {( 2 , 2 )})= \overline{   8   },
	\ f (\overline {( 3 , 2 )})&= \overline{   9   },
	\ f (\overline {( 4 , 2 )})= \overline{   10   },
	\ f \overline {( 5 , 2 )}= \overline{   11   },\\
	f (\overline {( 3 , 3 )})= \overline{   12   },
	\ f (\overline {( 5 , 3 )})= \overline{   14   },
	\ f (\overline {( 6 , 3 )})&= \overline{   15   },
	\ f (\overline {( 4 , 4 )})= \overline{   16   },
	\ f (\overline {( 6 , 4 )})= \overline{   18   }.
 \end{align*}
 This indicates that $\operatorname{a-rk } ( \overline{\mathcal{P}}) = 16$
 as 
 $$\operatorname{a-rk }(\{\overline{0}, \overline{1},\overline{3},\overline{4},\overline{5}, \overline{6}, \overline{7}, \overline{8}, \overline{9}, \overline{10}, \overline{11}, \overline{12}, \overline{14}, \overline{15}, \overline{16}, \overline{18}\})=16$$ 
 in $\Z/19\Z$. The previous equality follows from the fact that $\overline{2}, \overline{17}, \overline{13} $ is a $3$-AP in $\Z/19\Z$.
By Theorem \ref{our result}, we obtain that $h$ is an abelian 16-power free morphism.
Let $\Omega_3=h^{\omega}(a)$
be the infinite fixed point of $h$ which is obtained by repeatedly applying $h$ to $a$. Hence, $\Omega_3$ is abelian 16-power free. 
The first and the second condition follow similarly as in Theorem \ref{4 consecutive a 6AP}.
\end{proof}

\subsection{\texorpdfstring{Some Methodological Constraints under $bb$-Avoidance}{Some Methodological Constraints under bb-Avoidance}}

When examining the ball-box distribution problem through the distinct frameworks of Dekking’s theorem, Carpi’s conditions, and the template method, different methodological constraints naturally arise.
The objective of this part is to provide a quantitative comparison of the limitations of the various approaches. In particular, using Dekking's theorem (Theorem \ref{dekkinglemma}),
we show that we can give a partial answer to the  2-ball-box distribution problem. The abelian powers of the infinite words that we produce using Dekking's result will be higher, and there are cases where Dekking's result is not applicable as in the case of avoiding squares over an alphabet of size 4. 
When Dekking's result is applicable, the natural subsequent goal is to improve the bound on the avoidance of abelian powers. We pose this as an open problem in the last section. In this subsection, we present the best avoidance results that we achieved through the application of Dekking's theorem.
The proposition below establishes that Theorem \ref{dekkinglemma} fails to provide a
result in the case where fixed points contain no three consecutive $a$’s
and no two consecutive $b$’s (see Theorem \ref{2 consecutive a 16AP}). 
Thus, it can be seen that this part of our problem is closed to Dekking's result. The next proposition also sheds some light on why the morphism $h$ from Theorem \ref{2 consecutive a 16AP} has the property that $h(a)$ starts with $ab$ and $h(b)$ starts with $ba$.

\begin{proposition}\label{nonexample for 2 consecutive a}
Let $\Sigma=\{a,b\}$
be a binary alphabet.
Let $G$ be a finite abelian group, $h:\Sigma^*\to \Sigma^*$ and $f:\Sigma^*\to G$ be monoid morphisms such that
$(D1)-(D4)$ from Theorem \ref{dekkinglemma} are satisfied.
Assume that none of $aaa$ and $bb$ is a factor of $h(a)$ and $h(b)$, and $|h(a)|, |h(b)| >1.$ 
Then, at least one of $aaa$ and $bb$ is a factor of $h^3(a),$ and  at least one of $aaa$ and $bb$ is a factor of $h^3(b)$.
\end{proposition}
\begin{proof}
Suppose not and we first
assume that $h$ takes the same values on $a$ and $b$ up to their first $k$ letters for some positive integer $k$. Without loss of generality, we write
	\begin{align*}
	h(a) &= \alpha_1\cdots \alpha_k a u \\
	h(b) &= \alpha_1\cdots \alpha_k b v
	\end{align*}
 (or the other way around when it is crucial for the proof) for some words $u$ and $v$, where $\alpha_i$ is in $\Sigma$ for $i \in \{1,\ldots,k\}$. Notice that $\alpha_k=a$ as $bb$ is not allowed as a factor.
Before going into the proof, we make an observation. Let $x \in \operatorname{Pref}(h)$ and $x \notin h(\Sigma)\cup \{\epsilon \}$. Then, one sees that $f(x)\neq 0$ by applying (D4) to the words $\epsilon$ and $x$.
Suppose that $|u| \ge 2$ and $|v| \ge 1$. As $aaa$ is not a factor of $h(a)$, this gives $u=bu_1$ for some word $u_1$ which is not $\epsilon$. For a similar reason, we have $v=av_1$ for some word $v_1.$ 
Consider the prefixes with length $ k+2$ of $h(a)$ and $h(b)$. They are
\begin{align*}
    V= \alpha_1\cdots \alpha_k ab \\  
    W=\alpha_1\cdots \alpha_k ba
\end{align*} 
respectively. We now rule out the possibility that $W\in h(\Sigma)$. This would imply $f(V)=f(\epsilon)$, contradicting (D4). Therefore, we have $W\not\in h(\Sigma)$ and continue the proof. As the number of $a$'s and $b$'s are the same in $V$ and $W$, one infers that $f(V)=f(W)$, which in turn implies that $V^{\prime}=W^{\prime}$ by $(D4)$, where $VV'=h(a)$ and $WW'=h(b)$. However, the rows of the frequency matrix of $h$ will be identical, making the matrix singular, and this contradicts $(D1)$.
Similarly, the case $|u| \ge 1$ and $|v| \ge 2$ is impossible. If $|u|=|v|=1$ ($u=b$ and $v=a$), then this contradicts $(D1)$ again. This leads us to three cases.\\

\noindent \textit{Case 1.1}: Both $u$ and $v$ are $\epsilon$.
If $k=1$, then $h(a)=aa$ and $h(b)=ab$. Then, we see that $aaa$ is a factor of both $h^3(a)=aaaaaaaa$ and $h^3(b)=aaaaaaab$.
Similarly, if $h(a)=ab$ and $h(b)=aa$, then $aaa$ is a factor of both $h^3(a)=abaaabab$ and $h^3(b)=abaaabaa$.
Now, suppose that $k \ge 2$.
By $(D2)$, we know that $f(h(a))=f(h(b))=0$, which in turn yields that $f(a)=f(b).$ First, we will show that $f(a)=f(b)$ is not zero. Suppose that $f(a)=f(b)=0$.
Then, $f(\alpha_1)=f(\alpha_1\cdots \alpha_k)=0$. 
However, neither $\alpha_1$ and $\alpha_1\cdots \alpha_k$ nor $\alpha_2\cdots \alpha_ka$ and $b$ are equal. This contradicts $(D4)$. Thus, $f(a)$ cannot be zero. In this case $(k+1)f(a)=0$, and hence 
the set $$ f (\operatorname{Pref} (h) )= \{g\in G: g=f(v), v\text{ is a prefix of $h(\sigma_i)$ for some $\sigma_i \in \Sigma$} \} $$ contains the following arithmetic progression  $\{mf(a): m \ge 1\}$ of infinite length. This contradicts $(D3)$.\\

\noindent \textit{Case 1.2}: Suppose that $u$ is not $\epsilon$ but $v=\epsilon.$ As $\alpha_k=a$, one obtains that $h(a)=\alpha_1\cdots \alpha_k ab u_1$ and $h(b)=\alpha_1\cdots \alpha_k b$ for some word $u_1$ and $u=bu_1.$ Assume that $\alpha_1=b$. Then, both $h(a)$ and $h(b)$ start with $ba$. This yields that $bb$ is a factor of both $h^3(a)$ and $h^3(b)$. Thus, we may suppose that $\alpha_1=a$. Since $f(h(b))=0$, one has that 
$f(\alpha_1\cdots \alpha_k ab)=f(a)=f(\alpha_1)$.
However, neither $\alpha_1\cdots \alpha_k ab$ and $a$ nor $u_1$ and $\alpha_2\cdots \alpha_k abu_1$ are equal. This contradicts $(D4)$ when $u_1\neq \epsilon$. If $u_1=\epsilon$, then $f(a)=0$, contradicting the observation.
\\

\noindent \textit{Case 1.3}: Suppose that $v$ is not $\epsilon$ but $u=\epsilon.$ Then, $h(a)=\alpha_1\cdots \alpha_k a$ and $h(b)=\alpha_1\cdots \alpha_k bav_1$ for some word $v_1$. If $\alpha_1=b$, then $f(\alpha_1)=f(b)=f(\alpha_1\cdots \alpha_k ba)$ as $f(h(a))=0$. This contradicts $(D4)$ as in the previous case. Thus, $\alpha_1$ must be $a$. Suppose that $k=1$. Then, $h(a)=aa$ and $h(b)=abav_1$. This yields that $aaa$ is a factor of both $h^3(a)$ and $h^3(b)$. Similarly, if $h(a)=abav_1$ and $h(b)=aa$, then $aaa$ is a factor of both $h^3(a)$ and $h^3(b)$ again. Now, suppose that $k \ge 2$. Then, $h(a)=axaa$ and $h(b)=axabav_1$ for some word $x$, or the other way around. In all cases, we deduce that $aaa$ is a factor of both $h^3(a)$ and $h^3(b)$.\\

Hence, the first letters of $h(a)$ and $h(b)$ must be different.
We are only left with the case
\begin{align*}
	h(x)&=ay_1\\
	h(y)&=bax_1
\end{align*}
	where $\{x,y\}=\{a,b\}$ and $x_1, y_1$ are some words.
 As before, by the observation at the beginning of the proof, we rule out the case $y_1=\epsilon$. Note that if the first entry of $y_1$ is $b$, then we have once again the singularity problem of the matrix by $(D4)$. These constraints lead to   
\begin{align*}
	h(x)&=aau\\
	h(y)&=bav.
	\end{align*}
for some words $u$ and $v$.
Again, we split the proof into several cases.\\

\noindent \textit{Case 2.1}: Suppose that $u$ and $v$ are both the empty word $\epsilon$.
Then, either $h^3(a)=aaaaaaaa$ or $h^3(a)=babaaaba$. In either case, $aaa$ is a factor of $h^3(a)$. Similarly, either 
$h^3(b)=baaaaaaa$ or $h^3(b)=aabaaaba$. In either case, $aaa$ is a factor of $h^3(b)$. \\

\noindent \textit{Case 2.2}: Assume that $u$ is not $\epsilon$ but $v=\epsilon$.
Then, this yields that $u=bu_1$ for some word $u_1.$ Therefore, one sees that both $ab$ and $ba$ are factors of $h^2(a)$ for $h(a) \in \{aabu_1, ba\}$. This in turn yields that $aaa$ is a factor of $h^3(a).$ Similarly, both $ab$ and $ba$ are factors of $h^2(b)$ for $h(b) \in \{aabu_1, ba\}$. Hence, $aaa$ is a factor of $h^3(b).$ \\

\noindent \textit{Case 2.3}: Suppose that $u=\epsilon$ and $v$ is not $\epsilon$. Then, either $h^3(a)=aaaaaaaa$ or $h^3(a)=bavbavaabavh(v)h^2(v)$.
In the first case, $aaa$ is a factor of $h^3(a)$. Now, assume that $h(a)=bav$ and so $h^3(a)=bavbavaabavh(v)h^2(v)$.
Write $v=v_1\cdots v_n$ where $n \ge 1$ and $v_i \in \{a,b\}$ for $i \in \{1,\ldots,n\}$. If $v_n=b$, then $bb$ is a factor of
$h^3(a)$, and if $v_n=a$, then $aaa$ is a factor of
$h^3(a)$. 
Similarly, either $h^3(b)=bavaah(v)aaaah^2(v)$ or
$h^3(b)=aabavh(v)aabavh(v)$.
In the first case, $aaa$ is a factor of $h^3(b)$. In the second case, we have $h(b)=aa$ and $h(v)aa$ is a factor of
$h^3(b)$, and so $h(v_n)aa$ is a factor of $h^3(b)$. If $v_n=a$, then $h(v_n)aa=bavaa=bav_1\cdots v_naa$ and $aaa$ is a factor of it. If $v_n=b$, then $h(v_n)aa=aaaa$ and thus $aaa$ is a factor of it.\\

\noindent \textit{Case 2.4}: Assume that both $u$ and $v$ are not $\epsilon.$ Then, $u=bu_1$ for some word $u_1$. Using $(D1)$ and $(D4)$ again, one infers that $v=bv_1$ for some word $v_1$. First, suppose that $h(a)=aabu_1$ and $h(b)=babv_1$.
Then, one arrives at 
$h^3(a)=aabu_1aabu_1babv_1v_2$
for some word $v_2$. Write $u_1=w_1\cdots w_n$. 
If $u_1=\epsilon$, or one of the first and the last letter of $u_1$ is $b$, then we obtain that $bb$ is a factor of $bu_1b$, so it is also a factor of $h^3(a)$. If not, then the last letter of $u_1$ is $a$, and this leads to the fact that $aaa$ is a factor of $u_1aa$ and we are done again.
Similarly, in this situation, one sees that
$h^3(b)=babv_1aabu_1babv_1h(v_1)aabu_1aabu_1v_3$ for some word $v_3$. Consider the factors $bu_1b$ and $u_1aa$ of $h^3(b)$ again. By the previous analysis, we conclude that $aaa$ or $bb$ is a factor of $h^3(b)$.
Next, assume that $h(a)=babv_1$ and $h(b)=aabu_1$.
This gives $h^3(a)=babv_1babv_1aav_4$ for some word $v_4$. Similar to the previous argument, as $bv_1b$ and $v_1aa$ are both factors of $h^3(a)$, we are done. Finally, note that
$$h^3(b)=aabu_1babv_1aabu_1h(v_1)aabu_1babv_1aabu_1h(v_1)babv_1babv_1aav_5$$
for some word $v_5$. Once again, as 
$bv_1b$ and $v_1aa$ are both factors of $h^3(b)$, the proof is now complete.
\end{proof}

Now, we will prove analogs of Theorem \ref{4 consecutive a 6AP} and Theorem \ref{3 consecutive a 9AP} by applying Dekking's theorem (Theorem \ref{dekkinglemma}), therefore one can also prove them by Theorem \ref{our result} via Theorem \ref{extdek}.
We will first prove an analog of Theorem \ref{4 consecutive a 6AP} using Dekking's theorem  with the same abelian power, but the word $aaaaa$ is a factor this time. Recall that in Theorem \ref{4 consecutive a 6AP}, the word $aaaaa$ is not a factor. 

\begin{theorem}\label{5 consecutive a 6AP}
    Dekking's result yields that there exists an abelian $6$-power free morphism $h$ over the binary alphabet $\Sigma= \left\lbrace a,b\right\rbrace  $ with a fixed point $\Omega_1'$ such that
	\begin{itemize}
		\item $ aaaaaa $ is not a factor of $ \Omega^{\prime}_1 $  (there are at most $5$ consecutive $a$'s in $ \Omega^{\prime}_1 $),
		\item $ bb $ is not a factor of $ \Omega^{\prime}_1 $  (there are no two consecutive $b$'s in $ \Omega^{\prime}_1 $),
		\item $ \Omega^{\prime}_1 $ is abelian $ 6 $-power free.		
	\end{itemize}
	\end{theorem}
 
\begin{proof}
	Define $h:\Sigma^*\to \Sigma^*$ by
	\begin{align*}
	h(a) &= ababa\\
	h(b) &= aaaaba.
	\end{align*}
	The frequency matrix of $h$ is $$ \begin{pmatrix}
	3 &2\\5& 1
	\end{pmatrix} $$
	and it is non-singular. Choosing $G= \Z /  7 \Z$, and defining $f:\Sigma^*\to G$ as
	$$f(a)=\overline{1}, \ f(b)=\overline{2} $$
	in Theorem \ref{dekkinglemma}, we get that $f(h(a))=f(h(b))=\overline{0}$.
	One sees that $f$ takes the following values when evaluated at the prefixes:	
	\begin{align*}
	\{f(\epsilon), f(a),f(ab), f(aba) ,f(abab), f(ababa)\} &=\{\overline{0}, \overline{1},\overline{3},\overline{4},\overline{6}\},\\
	\{f(\epsilon), f(a),f(aa), f(aaa) ,f(aaaa), f(aaaab) , f(aaaaba)\} &=\{\overline{0}, \overline{1},\overline{2},\overline{3},\overline{4},\overline{6}\}.
	\end{align*}
	This indicates that $f $ is $h$-injective. The set 
	$$
	A=\{g\in G: g=f(V), \ V \in \operatorname{Pref}(h) \}=\{\overline{0}, \overline{1},\overline{2},\overline{3},\overline{4},\overline{6}\} 
	$$
	does not contain a non-trivial 7-AP. By Theorem \ref{dekkinglemma}, the morphism $h$ is abelian 6-power free, and the infinite fixed point $\Omega'_1$ of $h$ obtained by repeatedly applying $h$ to $a$ satisfies the last condition of the theorem. 
	To check the first and the second condition, just note that in $h(a)$ and $h(b)$ there are no $aaaaaa$ and $bb$ as factors, and $h(aa)$, $h(ab) $, $h(ba)$ do not produce the factors $aaaaaa$ and $bb$ as well. 
\end{proof}

\begin{example} One might also consider the following morphism $h$ to prove the previous result using Dekking's theorem:
    \begin{align*}
		h(a) &= aaabaa\\
		h(b) &= babaa,
		\end{align*}
		with $G=\Z/7\Z$, and $ f(a)=\overline{1}$, $f(b) =\overline{2}$.  
\end{example}
Next, we will prove analogs of Theorem \ref{4 consecutive a 6AP} and Theorem \ref{3 consecutive a 9AP} (see Theorem \ref{4 consecutive a 8AP} and Theorem \ref{3 consecutive a 11AP}) using Dekking's theorem (Theorem \ref{dekkinglemma}) with a higher abelian power, respectively. Recall that in Theorem \ref{4 consecutive a 6AP}, the infinite word $\Omega_1$ is abelian 6-power free, while in Theorem \ref{3 consecutive a 9AP}, the infinite word $\Omega_2$ is abelian 9-power free. It is an intriguing question to show that the upper bounds below are best possible that can be obtained using Dekking's theorem with the given constraints, and we present these questions in the concluding section. 
\begin{theorem}\label{4 consecutive a 8AP}
    Dekking's result yields that there exists an abelian $9$-power free morphism $h$ over the binary alphabet $\Sigma= \left\lbrace a,b\right\rbrace  $ with a fixed point $\Omega_1''$ such that
	\begin{itemize}
		\item $ aaaaa $ is not a factor of $\Omega^{\prime \prime}_1$ (there are at most $4$ consecutive $a$'s in $\Omega^{\prime \prime}_1$),
		\item $ bb $ is not a factor of $\Omega^{\prime \prime}_1$ (there are no two consecutive $b$'s in $\Omega^{\prime \prime}_1$),
		\item $\Omega^{\prime \prime}_1$ is abelian $ 9 $-power free.		
	\end{itemize}
	\end{theorem}    
\begin{proof}
	Define $h:\Sigma^*\to \Sigma^*$ by
	\begin{align*}
	h(a) &= aaabaaaba\\
	h(b) &= babaaaba.
	\end{align*}
	The frequency matrix of $h$ is 
 $$\begin{pmatrix}
	7 & 2\\5& 3
	\end{pmatrix}$$ which is non-singular. Choosing $G= \Z /  11 \Z$, and defining $f:\Sigma^*\to G$ as
	$$f(a)=\overline{1}, \ f(b)=\overline{2} $$
	in Theorem \ref{dekkinglemma}, we get that $f(h(a))=f(h(b))=\overline{0}$.
	Notice that $f$ takes the following values when evaluated at the prefixes:
	\begin{align*}
	f (\{\epsilon, a, aa, aaa, aaab, aaaba,aaabaa,aaabaaa, aaabaaab, aaabaaaba \}) \\
	=\{\overline{0}, \overline{1},\overline{2},\overline{3},\overline{5},\overline{6},\overline{7},\overline{8},\overline{10} \},\\
	f (\{\epsilon, b,ba,bab, baba,babaa,babaaa,babaaab,babaaaba \}\\ =\{\overline{0}, \overline{2},\overline{3},\overline{5},\overline{6},\overline{7},\overline{8},\overline{10}\}).
	\end{align*}
	This yields that $f $ is $h$-injective. The set 
	\begin{align*}
	A=&\{g\in G: g=f(V),\  V \in \operatorname{Pref}(h) \}\\
	=&	\{\overline{0}, \overline{1},\overline{2},\overline{3},\overline{5},\overline{6},\overline{7},\overline{8},\overline{10} \}  
	\end{align*}
	has 9 elements, and one of the longest non-trivial arithmetic progressions contained in $A$ is the following one
 $$ \overline{3},\overline{8}, \overline{2}, \overline{7}, \overline{1}, \overline{6}, \overline{0}, \overline{5}, \overline{10}.$$
 Hence, $A$
	is arithmetic progression free of length $10$ in $G$. By Theorem \ref{dekkinglemma}, the morphism $h$ is abelian 9-power free, and the infinite fixed point $\Omega_1''$ of $h$ obtained by repeatedly applying $h$ to $a$ satisfies the last condition of the result. 
Other conditions can be checked as in Theorem \ref{5 consecutive a 6AP}.
\end{proof}

\begin{example}
	One may also consider the following morphism $h$ to prove the previous result via Dekking's theorem:
	\begin{align*}
	h(a) &= ababaaab\\
	h(b) &= aaaabaaab,
	\end{align*}
	with  $G=\Z/11\Z$, and $ f(a)=\overline{1}$, \ $f(b) =\overline{2}$. 
\end{example}

\begin{theorem}\label{3 consecutive a 11AP}
    Dekking's result yields that there exists an abelian $11$-power free morphism $h$ over the binary alphabet $\Sigma= \left\lbrace a,b\right\rbrace  $ with a fixed point $\Omega_2'$ such that
	\begin{itemize}
		\item $ aaaa $ is not a factor of $ \Omega^{\prime}_2 $  (there are at most $3$ consecutive $a$'s in $ \Omega^{\prime}_2 $),
		\item $ bb $ is not a factor of $ \Omega^{\prime}_2 $  (there are no two consecutive $b$'s in $ \Omega^{\prime}_2 $),
		\item $ \Omega^{\prime}_2 $ is abelian $ 11 $-power free.		
	\end{itemize}
	\end{theorem}
\begin{proof}
	Define $h:\Sigma^*\to \Sigma^*$ by
	\begin{align*}
	h(a) &= aba aabaa aba\\
	h(b) &= aba bab aba.
	\end{align*}
	The frequency matrix of $h$ is $$\begin{pmatrix}
	8 & 3\\5& 4
	\end{pmatrix}$$ which is non-singular. Choosing $G= \Z /  17 \Z$, and defining $f:\Sigma^*\to G$ as
	$$ f(a)=\overline{1}, \ f(b)=\overline{3} $$
	in Theorem \ref{dekkinglemma}, we get that $f(h(a))=f(h(b))=\overline{0} $.
 We observe that
	$f$ assumes the following values when evaluated at the prefixes:
	\begin{align*}
	f (\{V: V \text{ is a prefix of $h(a)$}\})&=\{\overline{0}, \overline{1}, \overline{4}, \overline{5}, \overline{6}, \overline{7}, \overline{10}, \overline{11}, \overline{12}, \overline{13}, \overline{16}\}, \\
	f (\{V: V \text{ is a prefix of $h(b)$}\})&=\{\overline{0}, \overline{1}, \overline{4}, \overline{5}, \overline{8}, \overline{9}, \overline{12}, \overline{13}, \overline{16}\}.
	\end{align*}
	This gives that $f $ is $h$-injective. The set 
	\begin{align*}
	A=&\{g\in G: g=f(V), \ V \in \operatorname{Pref}(h)\}\\
	=&\{\overline{0}, \overline{1}, \overline{4}, \overline{5}, \overline{6}, \overline{7}, \overline{8}, \overline{9}, \overline{10}, \overline{11}, \overline{12}, \overline{13}, \overline{16}\}  
	\end{align*}
	has 13 elements, and the longest non-trivial arithmetic progression in $A$ is of length $11$.
	By Theorem \ref{dekkinglemma}, the morphism $h$ is abelian 11-power free, and the infinite fixed point $\Omega_2'$ of $h$ obtained by repeatedly applying $h$ to $a$ satisfies the last condition of the result. Other conditions can be checked as in Theorem \ref{5 consecutive a 6AP}.
\end{proof}
\section{A Sieve Technique for the Template Method} \label{Sieve Section}
The template method provides a framework to decide whether a morphic word is abelian power-free. This method is grounded in the observation that a word is an abelian $k$-power if and only if it realizes the following $k$-template 
$$ T_k = \left[ \epsilon, \ldots, \epsilon, \vec{0}, \ldots, \vec{0} \right] .$$
By  analyzing ancestor templates and their realizations by certain factors, the template method ascertains the existence of abelian $k$-powers. When using the template method, the number of ancestors and the number of factors of a certain length can become prohibitively large, even for relatively small values of $k$. To make this approach practical, we apply a sieve technique that dramatically reduces the number of ancestors and candidate instances and it makes the method computationally tractable. This section begins with a detailed explanation of the template method and demonstrates how the sieve technique is applied to specific morphisms. At the end of this section, we provide an overview of morphisms and abelian $k$-power freeness.\\

Now, we give definitions of templates, realizations (referred to as instances in \cite{C-R2012}), parents, and ancestors, as well as their roles in the decision process. We refer the reader to \cite{C-R2012, Rao-R} for more details. Throughout this section, $\Sigma$ denotes a finite alphabet of size $m$. 

\begin{definition}
	For any integer $ k \ge 2 $, a k-template is a finite sequence of the form
	$$
	t = [a_1, a_2, \dots, a_{k+1}, d_1, d_2, \dots, d_{k-1}],
	$$
	where $a_i \in \Sigma \cup \{\epsilon\}$ and $d_j \in \mathbb{Z}^m$ for $i\in\{1,\ldots, k+1\} ,$ $j \in \{1,\ldots,k-1\}$. 
\end{definition}

A template provides a blueprint for a pattern in a fixed point. Each $a_i$ determines the letters at certain positions in the realization, while the vectors $d_i$'s describe the differences in Parikh vectors between consecutive segments of a word.

\begin{definition}
	A word $w$ is said to be a realization of the template $$t = [a_1, a_2, \dots, a_{k+1}, d_1, d_2, \dots, d_{k-1}]$$ if there exist factors $X_1, X_2, \dots, X_k \in \Sigma^*$ of $w$ such that:
	$$ w = a_1 X_1 a_2 X_2 \dots a_k X_k a_{k+1},$$
	and for each $i = 1, 2, \dots, k-1$,
	$$\psi(X_{i+1}) - \psi(X_i) = d_i.$$
\end{definition}

For example, consider the 2-template $t = [a, b, \epsilon, (1, -2)]$ on the alphabet $ \left\lbrace a,b\right\rbrace  $. A word realizing $t$ must have the form $a X_1 b X_2$, where the difference between the Parikh vectors of $X_2$ and $X_1$ is $(1, -2)$. The word $w = abbba$ realizes this template because the Parikh vectors of $a$ and $bb$ differ by $(1, -2)$.

\begin{definition}
	A parent of the template $t = [a_1, a_2, \dots, a_{k+1}, d_1, d_2, \dots, d_{k-1}]$ under a morphism $h$ is another template $t_p = [A_1, A_2, \dots, A_{k+1}, D_1, D_2, \dots, D_{k-1}]$ such that:
	$$h(A_i) = p'_i a_i p''_i,$$
    for some words $p'_i, p''_i \in \Sigma^*$ with $i\in \{1,\ldots, k+1\}$, and the Parikh difference vectors satisfy
	$$	\psi(p''_{j+1} p'_{j+2}) - \psi(p''_j p'_{j+1}) + D_j M = d_j$$
for $j \in \{1,\ldots,k-1\}$,
where $M$ is the frequency matrix of the morphism $h$.
\end{definition}
The concept of a parent allows us to trace back how a template is formed under the action of the morphism. 
If a word $w$ realizes a parent template $t_p$, then $h(w)$ contains a factor that realizes the template $t$. Now, we quickly prove this (see \cite{C-R2012}).
	Let $w = A_1 Y_1 A_2 Y_2 \dots A_k Y_k A_{k+1}$ be a realization of the parent template $t_p$. By applying the morphism $h$ to $w$, the resulting word $h(w)$ has the form
	$$h(w) = p'_1 a_1 p''_1 h(Y_1) p'_2 a_2 p''_2 \dots p'_k a_k p''_k h(Y_k) p'_{k+1} a_{k+1} p''_{k+1},$$
	which contains the factor $a_1 X_1 a_2 X_2 \dots a_k X_k a_{k+1}$, where $X_i = p''_i h(Y_i) p'_{i+1}$. The Parikh vector conditions ensure that this factor realizes the template $t$.

\begin{definition}  
    Let $t_1, t_2, \dots, t_n$ be a sequence of templates where each $t_{i+1}$ is a parent of $t_i$ for $i\in\{1,\ldots,n-1\}$. Then, the template $t_n$ is called an ancestor of $t_1$. The set of all ancestors of $t$ is denoted by $\operatorname{Anc}(t)$.
\end{definition}

The set of ancestors of the template $T_k$ is finite when the operator norm of $M^{-1}$ is strictly less than 1 where $M$ is the frequency matrix of $h$. This is crucial in the implementation of the template method, as it guarantees that only finitely many templates are examined to decide abelian power-freeness. However, even if the operator norm of $M^{-1}$ is larger than 1, in which case $\operatorname{Anc}(t)$ may be infinite,  it is still possible to use a generalization of the template method introduced in  \cite{Rao-R}.\\

The \emph{Inverse Parent Lemma} in \cite{C-R2012} states that if a factor $u$ of a fixed point of a morphism $h$ has length greater than 
\begin{align*}
N + k - 1 + (k - 1)(N - 2 + mk\Delta),\end{align*} 
where $ N=\max\limits_{\sigma \in \Sigma }|h(\sigma)| $,   $ m=|\Sigma| $ and $ \Delta = \lfloor \max_i \|d_i\|_2 \rfloor $, and $u$ realizes a template $t_1$, then there exists a parent template $t_2$ of $t_1$ which has a realization of length less than $|u|$. Thus the analysis can be reduced to factors of bounded length: to determine whether a fixed point is abelian power-free,  it suffices to check that no templates $t\in\operatorname{Anc} (T_k)$ 
where 
$$ T_k = \left[ \epsilon, \ldots, \epsilon, \vec{0}, \ldots, \vec{0} \right] $$
are realized by any factor of length less than
\begin{align} 
    N + k - 1 + (k - 1)(N - 2 + mk\Delta), \label{search bound}
\end{align}
where $ \Delta = \lfloor \max_i \|D_i\|_2 \rfloor $. 
As an application of the template method, Currie and Rampersad \cite{C-R2012} computed the ancestors of the template $T_3$ with respect to the morphism $h$ defined on a ternary alphabet $\Sigma=\{a,b,c\}$ by $h(a)=aabc$, $h(b)=bbc$, $h(c)=acc$. Then, by checking the word length specified in equation \eqref{search bound}, Currie and Rampersad concluded that $h^{\omega}(a)$ is abelian 3-power free and hence they reproved Dekking's result from \cite{dekking}. {In the same spirit, we verified that the template $T_4$ has 3123 ancestors with respect to Dekking's morphism $h$ defined by \eqref{dekm}, and we re-concluded that $h^{\omega}(a)$ is abelian 4-power free.}\\

The bound in \eqref{search bound} for the inverse parent lemma can be large when one of $m$ or $k$ is large enough, and this may limit the applicability of the template method. Thus, we felt the need to give the inverse parent lemma in the following form to enhance its feasibility, and it is vital due to the high powers we aim to investigate (see Proposition \ref{SievingT8} and Theorem \ref{T14-template}). 

\begin{lemma} \label{new search bound}
Let $h:\Sigma^*\to \Sigma^*$ be a morphism with
 $ N=\max\limits_{\sigma \in \Sigma }|h(\sigma)| $ and  $ m=|\Sigma|$.
Suppose that
\begin{itemize}
    \item $ h(\sigma_1)=\sigma_1x, $ for some non-empty word $ x $ in $ \Sigma^*, $
	\item $|h(\sigma_i)| > 1$ for each $ i \in \left\lbrace 1,\dots,m\right\rbrace.$
\end{itemize} 
Assume that $\mathcal{I}$ is a factor of $h^{\omega}(\sigma_1)$ which is an instance of a $k$-template $$t = [a_1, a_2, \dots, a_{k+1}, d_1, d_2, \dots, d_{k-1}],$$
and 
$$
|\mathcal{I}| > N+k-1 + (k-1)\big(N-2 + \max_{i< j \le k}\big|\sum_{r=i}^{j-1} d_r \cdot (1, \ldots, 1)\big|\big).
$$
Then, for some parent $t_1$ of $t$, the infinite word $h^{\omega}(\sigma_1)$ contains a factor $\mathcal{J}$ which is an instance of $t_1$ and $|\mathcal{J}| < |\mathcal{I}|$.
\end{lemma}

The proof of Lemma \ref{new search bound} is similar to that of the inverse parent lemma \cite[Lemma 5]{C-R2012}, so we will be brief.
In the inverse parent lemma, we have the term $mk\Delta$ instead of the sum 
\begin{equation} \label{new sum}
  \max_{i< j \le k}\big|\sum_{r=i}^{j-1} d_r \cdot (1, \ldots, 1)\big|.  
\end{equation}

However, the sum \eqref{new sum} can be smaller than $mk\Delta$, as our vectors may have negative coordinates, so they can cancel each other. As a result, one can get a better search bound for the template method. In any case, the sum \eqref{new sum} is at most $\sum_{r=1}^{k-1}\|d_r\|_1$,
where $\|v\|_1$ denotes the 1-norm for a vector $v$ in $\mathbb R^m.$ Using the Cauchy--Schwarz inequality, we see that 
$$\sum_{r=1}^{k-1}\|d_r\|_1 \le \sqrt{m}\sum_{r=1}^{k-1}\|d_r\|_2 \le \sqrt{m} (k-1)(\Delta +1),$$
where 
$ \Delta = \lfloor \max_i \|d_i\|_2 \rfloor $. Hence, when the size of the alphabet $m$ is large enough, the sum \eqref{new sum} is smaller than $mk\Delta$ even if there is no cancellation.

\begin{proof}

Let $\mathcal{I}$ be an instance of the template $t$, and write 
$$\mathcal{I}=a_1X_1 \cdots a_kX_ka_{k+1},$$
where $\psi(X_{i+1})-\psi(X_i)=d_i$ for $i$ from $\{1, \ldots, k-1 \}$.
For $i < j \le k$, note that $$\psi(X_j)-\psi(X_i)=\sum_{r=i}^{j-1}d_r.$$
For any word $w$, we know that the length $|w|$ of $w$ is equal to the dot product $$\psi(w) \cdot  (1, \ldots, 1).$$ This yields that 
$$
|X_j|-|X_i|=\sum_{r=i}^{j-1} d_r \cdot (1, \ldots, 1).
$$
Therefore, we deduce that

$$\max_{i< j \le k} \big||X_j|-|X_i| \big|
=  \max_{i< j \le k}\big|\sum_{r=i}^{j-1} d_r \cdot (1, \ldots, 1)\big|. 
$$

The rest of the argument is completely identical to that of \cite[Lemma 5]{C-R2012}, so we are done.
\end{proof}

We consider the morphism $ h $ in Theorem~\ref{3 consecutive a 9AP} for which we have established that it is abelian 9-power free. However, it remains unclear whether $ h $ is abelian 8-power-free. In fact, $ h $ is not abelian 8-power-free, as demonstrated by the following example:
$$
h(abba)=abaa\, \underline{ab} \, \underline{a b}\, \underline{ab}\,\underline{ab} \,\underline{ba}\,\underline{ba}\,\underline{b a}\,\underline{ba}\,aaba,
$$
which contains an abelian 8-power. Nevertheless, this does not imply that its fixed point $ h^\omega(a) $ contains an abelian 8-power. Therefore, it is necessary to test whether $ h^\omega(a) $ avoids abelian 8-powers. By analyzing the first 11039 letters of $ h^\omega(a) $, we find no occurrence of an abelian 8-power. The frequency matrix $ M $ of $ h $ is 
$$ \begin{pmatrix}
5 &  2 \\
2 &  3
\end{pmatrix} $$ and it is non-singular as given before. The operator norm of the matrix $ M^{-1} $ is $$ \frac{1}{11}\left( 4+\sqrt{5}\right) \approx 0.56 < 1 $$ by the well-known conclusion in linear algebra that the operator norm of a symmetric matrix is the largest absolute value of its eigenvalues (or the largest singular value). Therefore, we can apply the template method (Theorem \ref{templatem}) to decide whether $ h^\omega (a)$ is abelian power-free.\\

When the template method is applied to this morphism, it yields approximately two million ancestors. Among these, we compute that the maximal Parikh vector deviation is $ \Delta  = \lfloor \max_i \|D_i\|_2 \rfloor = 2 $. To use the algorithm described by Currie and Rampersad, we computed the search bound in equation \eqref{search bound} as
$$N + k - 1 + (k - 1)(N - 2 + mk\Delta) = 7 + 8 - 1 + 7\cdot (7 - 2 + 2 \cdot 8 \cdot 2) = {273},$$
where $N = \max\{ |h(a)|, |h(b)| \} = 7 $, $k = 8$ and  $m = |\Sigma| = 2 $.
To decide whether $h^{\omega}(a)$ is abelian $8$-power free, we must examine all factors of $ h^\omega(a) $ of length at most $273$. Checking whether these factors realize any of the nearly two million ancestors is beyond our reach with the current version of the code. \\

To reduce the search space, we introduce a sieve technique that synthesizes elements of Dekking’s combinatorial method with the template framework of Currie and Rampersad. Our technique enables an effective structural analysis that drastically reduces the search space. Specifically, we demonstrate that among the 1953162 parents associated with the template $ T_8 $, only 8 of them can possibly be realized in the fixed point (see Figure~\ref{parent figure}). This observation allows us to avoid computing the vast majority of the templates. Thus the verification of whether an abelian 8-power exists becomes computationally feasible.

    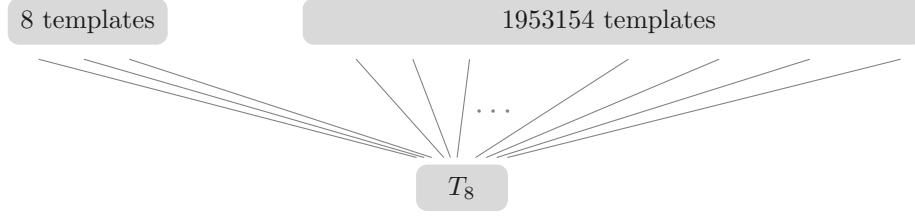
\begin{figure}[H]
    \centering
    \begin{tikzpicture}
    \draw[rounded corners,gray!30,fill=gray!30] (-0.4,0) rectangle (0.8,0.6) node[midway] {\textcolor{black!90}{\footnotesize $T_8$}};
    \foreach \i in {1,2,3}{\draw [-,gray] (-0.5+\i*0.1 ,0.7) to (-6+0.6*\i,   2);}
    \foreach \i in {8,9.25,10.5,14,16,18,20}{\draw [-,gray] (-0.6+\i*0.07,0.7) to (-6+0.6*\i,2);}
    \node[gray] (B) at (0.65, 1.3) {$ \cdots $};
    \draw[gray!30,fill=gray!30,rounded corners] (-5.8,2.2) rectangle (-3.7,2.8) node[midway] {\textcolor{black!90}{\footnotesize $8$ templates}};
    \draw[gray!30,fill=gray!30,rounded corners] (-1.9,2.2) rectangle (6.3,2.8) node[midway] {\footnotesize \textcolor{black!90}{$1953154$ templates }};
    \end{tikzpicture}
    \caption{Parents  of $T_8$ in Proposition \ref{SievingT8}}\label{parent figure}
    \end{figure}

\begin{proposition} \label{SievingT8}
 Let $ \Sigma = \{a, b\} $ be a binary alphabet. Then, the fixed point $h^{\omega}(a)$ of the morphism $ h: \Sigma^* \to \Sigma^* $ defined in Theorem~\ref{3 consecutive a 9AP} by
$$h(a) = abaaaba, \quad h(b) = babab$$
contains abelian 8-powers.
\end{proposition}

\begin{proof}

To find abelian 8-powers in our fixed point, we will find exactly $8$ parent templates of 
$$ T_8 = [\epsilon, \epsilon, \epsilon, \epsilon, \epsilon, \epsilon, \epsilon, \epsilon, \epsilon, (0,0), (0,0), (0,0), (0,0), (0,0), (0,0), (0,0)] $$
such that it is sufficient to analyze their realizability by certain factors.
We begin with some preliminary observations.
Let $G=\Z /  11 \Z$ and $f:\Sigma^*\to G$ be a monoid morphism such that
	$$f(a)=\overline{1}, \ f(b)=\overline{3}. $$
While $f$ vanishes on $h(\Sigma)$, that is to say $f(h(a))=f(h(b))=\overline{0}$,
it takes the following values in Table \ref{Pref 8} when evaluated on $\operatorname{Pref}(h)$.

\begin{table}[h]
    \centering
    \begin{tabular}{|c|c|c|}
        \hline
        score & $\operatorname{Pref}(h(a))$& $\operatorname{Pref}(h(b))$\\
        \hline\hline
        $\overline{1}$ & $a$ &-\\ \hline
        $\overline{2}$ & - &-\\ \hline
        $\overline{3}$ & - &$b$\\ \hline
        $\overline{4}$ & $ab$ &$ba$\\ \hline
        $\overline{5}$ & $aba$ &-\\ \hline
        \end{tabular}
        \hspace{0.5cm}
    \begin{tabular}{|c|c|c|}
        \hline
        score & $\operatorname{Pref}(h(a))$& $\operatorname{Pref}(h(b))$\\
        \hline\hline
        $\overline{6}$ & $abaa$ &-\\ \hline
        $\overline{7}$ & $abaaa$ &$bab$\\ \hline
        $\overline{1}$ & - & $baba$\\ \hline
        $\overline{9}$ & - &-\\ \hline
        $\overline{10}$ & $abaaab$ &-\\ 
        \hline
    \end{tabular}
    \vspace{0.5cm}
   \caption{Images of $\operatorname{Pref}( h(a)) $ and $\operatorname{Pref}(h(b))$ under $f$}
    \label{Pref 8}
\end{table}
Let
\begin{align*}
    P&: = f(\operatorname{Pref} (h) )=\{f(\epsilon), f(a),f(ab), f(aba) ,f(abaa), f(abaaa),f(abaaab),\\
    &f(b),f(ba),f(bab), f(baba)\}=\{\overline{0}, \overline{1},\overline{3},\overline{4},\overline{5}, \overline{6}, \overline{7}, \overline{8}, \overline{10} \}.
	\end{align*}
All the values in $P$ have a unique preimage in $\operatorname{Pref}(h)$ under $f$, except for $ \overline{4}$ and $\overline{7}$: 
\begin{align}
    f(ab)&=f(ba)=\overline4, \nonumber\\
    f(abaaa)&=f(bab)=\overline{7}.\label{8 pref of templates}
\end{align}
We  also have the corresponding evaluations for the suffixes:
\begin{align}
    f(aaaba)&=f(bab)=\overline{7}, \nonumber \\
    f(ba)&=f(ab)=\overline{4}. \label{8 pref of templates 2}
\end{align}
We are now ready to present the proof.
  We assume that $ h^\omega(a)$ contains an abelian 8-power. 
  Let $n$ be the minimum positive integer such that $h^n(a)$ contains an abelian 8-power, but $h^{n-1}(a)$ does not. 
Since $h^3(a)$ does not contain an abelian $8$-power, we see that $n \ge 4$.
Let $B_1, B_2, \dots, B_{8}$ occur consecutively in $ h^n(a)$, with the $B_i$'s are permutations of each other. As shown in Figure \ref{Fig:  abelian 8 power},
we set $v_i' \in \operatorname{Suff} (h) $ and  $v_{i+1}\in \operatorname{Pref} (h) $ to be the prefix and the suffix of the block $B_i$ where $v_iv_i' \in \{h(a), h(b) \}$ and $i \in \{1,\ldots, 8\}$. We write
$$B_i = v_i' \, \mu_i \, v_{i+1}$$
for some word $\mu_i$.
\begin{figure}[H]
    \centering
    \begin{tikzpicture}
        \draw (-1,0.7) -- (12,0.7);
        \draw (-1,0) -- (12,0);
    
        \foreach \x/\i in {0/1, 2/2, 4/3} {
            \draw (\x,0) rectangle (\x+0.5,  0.7) node[midway] {\footnotesize $v_{\i}$};
            \draw (\x+0.5,0) rectangle (\x+1,0.7) node[xshift=-0.23cm,yshift=-0.3cm]   {\footnotesize $v'_{\i}$};
        }

        \node at (6.5,0.35) {\dots};
        \node at (6.5,-1) {\dots};
        
        \foreach \x/\i in {8/8, 10/9} {
            \draw (\x,0) rectangle (\x+0.5,0.7) node[midway] {\footnotesize $v_{\i}$};
            \draw (\x+0.5,0) rectangle (\x+1,0.7) node[xshift=-0.23cm,yshift=-0.3cm]    {\footnotesize $v'_{\i}$};
        }

        \draw[fill=gray!30,rounded corners] (0.5,-0.7) rectangle (2.5,-1.2) node[midway] {\footnotesize $B_1$};
        \draw[fill=gray!30,rounded corners] (2.5,-0.7) rectangle (4.5,-1.2) node[midway] {\footnotesize $B_2$};
        \draw[fill=gray!30,rounded corners] (8.5,-0.7) rectangle (10.5,-1.2) node[midway] {\footnotesize $B_{8}$};
        
        \foreach \x in {0.5, 2.5, 4.5, 8.5, 10.5} {
            \draw[thick,dashed,gray] (\x,-0.7) -- (\x,0);}
     
     \foreach \x in {1, 2}{
            \node at (2*\x-0.5,0.3) {\footnotesize $\mu_{\x}$};
        }
    \node at (2*5-0.5,0.3) {\footnotesize $\mu_{8}$};
    \end{tikzpicture}
    \caption{An abelian 8-power and the corresponding blocks}
    \label{Fig:  abelian 8 power}
\end{figure}
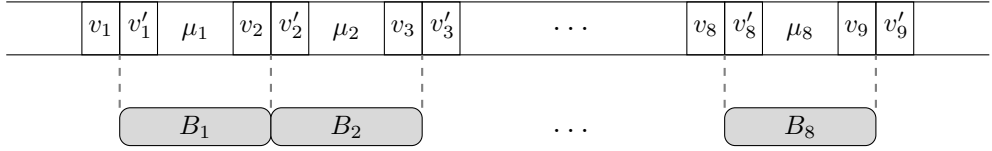

As the $B_i$'s are permutations of each other and $f(h(a))=f(h(b))=\overline{0}$, we deduce the following equalities:
\begin{equation}\label{T8 equ class}
    f(v_1') + f(v_2) = f(v_2') + f(v_3) = \cdots = f(v_{8}') + f(v_{9}).
\end{equation} 
The above equalities in \eqref{T8 equ class} yield that the sequence $ (f(v_{i})) $ for $ 1 \leq i \leq 9 $ is an arithmetic progression of length 9 in $ G $. 
Now, we define the sequence $(p_i)$ for $1 \leq i \leq 9$ by $p_i = f(v_i).$
Suppose that $(p_i)_{i=1}^9$ is a trivial arithmetic progression, in other words, it is a constant sequence. Among these constant sequences, two of them require special attention. Note that when $(p_i)_{i=1}^9 = (\overline 4)_{i=1}^9$, the prefixes taking $\overline 4$ under $f$ have the same Parikh vector:  $\psi(ab)=\psi(ba)$. When $(p_i)_{i=1}^9 = (\overline 7)_{i=1}^9 $,
the prefixes taking $\overline 7$ under $f$ do not have the same Parikh vector, $\psi(abaaa)\neq \psi(bab)$.  However, the corresponding suffixes $ab$ and $ba$ have the same Parikh vector, see Equation \eqref{8 pref of templates 2}. Now, assume that
$(p_i)_{i=1}^9 = (j)_{i=1}^9 \text{ for some $j\neq \overline{7}$.}$
Including the case $j=\overline{4} $, the prefixes $v_i$ have the same length, say $t_i$. 
We define $B_i'$ starting from $v_i$ to $ v_{i+1}$ for each $ 1\leq i \leq 8$ as in Figure \ref{Fig: B_i B_i prime}. 
Note that every block $B_i'$ has the same Parikh vector as $ B_i$. 
Moreover, each $B_i'$ is the image of a word $A_i$ for some word $A_i$ under $h$. The word $A_1\cdots A_8$ is an abelian $8$-power in $h^{n-1}(a)$. This contradicts the minimality of $n$. For the case $ j=\overline{7}$, we define $B_i'$ to be the block starting from at the end of $v_i'$ and ending at the end of $v_{i+1}'$ for all $1\leq i \leq 8$. We deduce the same result in a similar manner. Therefore, abelian $8$-powers cannot arise from the trivial arithmetic progressions. 

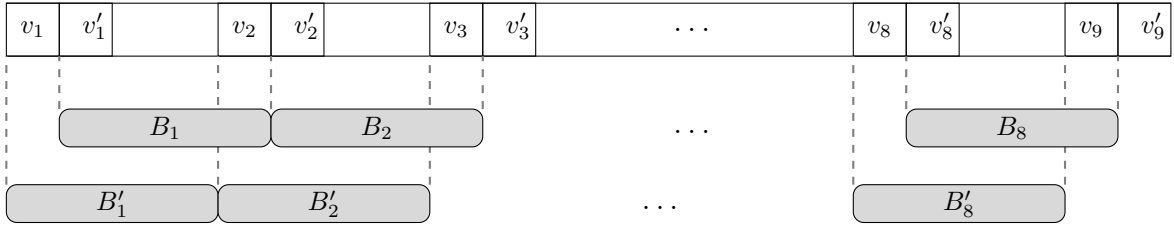
\begin{figure}[h]
    \centering
    \begin{tikzpicture}[xscale=1.4]
        \draw (0,0.7) -- (11,0.7);
        \draw (0,0) -- (11,0);
        \foreach \x/\i in {0/1, 2/2, 4/3} {
            \draw (\x,0) rectangle (\x+0.5,  0.7) node[midway] {\footnotesize $v_{\i}$};
            \draw (\x+0.5,0) rectangle (\x+1,0.7) node[xshift=-0.23cm,yshift=-0.3cm]   {\footnotesize $v'_{\i}$};
        }
        \node at (6.5,0.35) {\dots};
        \node at (6.5,-1) {\dots};
        \node at (6.2,-2) {\dots};
        
        \foreach \x/\i in {8/8, 10/9} {
            \draw (\x,0) rectangle (\x+0.5,0.7) node[midway] {\footnotesize $v_{\i}$};
            \draw (\x+0.5,0) rectangle (\x+1,0.7) node[xshift=-0.23cm,yshift=-0.3cm]    {\footnotesize $v'_{\i}$};
        }
        \draw[fill=gray!30,rounded corners] (0,-1.7) rectangle (2,-2.2) node[midway] {\footnotesize $B_1'$};
        \draw[fill=gray!30,rounded corners] (2,-1.7) rectangle (4,-2.2) node[midway] {\footnotesize $B_2'$};
        \draw[fill=gray!30,rounded corners] (8,-1.7) rectangle (10,-2.2) node[midway] {\footnotesize $B_{8}'$};

        \foreach \x in {0.5, 2.5, 4.5, 8.5, 10.5} {
            \draw[thick,dashed,gray] (\x,-0.7) -- (\x,0);
        }
        \foreach \x in {0, 2, 4, 8, 10} {
            \draw[thick,dashed,gray] (\x,-1.7) -- (\x,0);
        }

        \draw[fill=gray!30,rounded corners] (0.5,-0.7) rectangle (2.5,-1.2) node[midway] {\footnotesize $B_1$};
        \draw[fill=gray!30,rounded corners] (2.5,-0.7) rectangle (4.5,-1.2) node[midway] {\footnotesize $B_2$};
        \draw[fill=gray!30,rounded corners] (8.5,-0.7) rectangle (10.5,-1.2) node[midway] {\footnotesize $B_{8}$};
    \end{tikzpicture}
    \caption{The blocks $B_i'$ and $B_i$ when  $j\neq \overline{7}$}
    \label{Fig: B_i B_i prime}
\end{figure}
Now, we focus on the non-trivial arithmetic progressions of length $9$. 
There are exactly two such non-trivial arithmetic progressions. We denote them by $L_1$ and $L_2$ where
\begin{align*}
L_1:  &
(p_i)_{i=1}^9 = \{ \overline{5}, \overline{1}, \overline{8}, \overline{4}, \overline{0}, \overline{7}, \overline{3}, \overline{10}, \overline{6} \},\\
L_2: &
(p_i)_{i=1}^9 = \{ \overline{6}, \overline{10}, \overline{3}, \overline{7}, \overline{0}, \overline{4}, \overline{8}, \overline{1}, \overline{5} \}. 
\end{align*}
Equation \eqref{8 pref of templates} shows that the fourth and the sixth terms of $L_1$, namely $\overline{4}$ and $\overline{7}$, are images of two distinct prefixes. Therefore, we have $4$ prefix sequences that are mapped to $ L_1$ under $f$. Similarly, we have $4$ prefix sequences that are mapped to $L_2$ under $f$. Hence, we have $8$ cases to be analyzed in total, and these cases will give us $8$ parent templates of $T_8$ to understand the existence of 
abelian $8$-powers in our infinite word.\\

\noindent\underline{\textit{Case $L_1$}}: We aim to find $ v_1, \ldots, v_9 $ and $ v_1^\prime, \ldots, v_9^\prime $ corresponding to $L_1$. For each $i \in \{1,\dots,9\}$, the prefix $v_{i}$ and the suffix $v_i'$ are determined according to  $p_i=f(v_{i})$ scoring types. Recall that there are two different prefixes and suffixes in 2 different scoring types by \eqref{8 pref of templates}. 
Thus, for the abelian $8$-power  $B_1B_2\cdots B_{8}$,   there are $4$ choices for the prefixes $ v_1, \ldots, v_9 $ . The first subcase is shown in Table \ref{table-8power} in detail:
\begin{align*}
    v_4& = ab   \\
    v_6&=  abaaa .
\end{align*}
\noindent \underline{\textit{Case 1.1}}:
We give the full list of $ v_1,\ldots,v_9$ and $v_1',\ldots,v'_9$ below:
\begin{align*}
    &v_1=aba  , v_2=a, v_3= baba, v_4=ab, v_5=\epsilon, v_6=abaaa, v_7=b,    v_8=abaaab, v_9=abaa\\
    &v_1'=aaba, v_2'=baaaba, v_3'= b, v_4'=aaaba, v_5'=\epsilon, v_6'=ba, v_7'=abab,  v_8'=a, v_9'=aba.
\end{align*}

For a word $ w $, let $ |w|_a $ and $ |w|_b $ denote the number of occurrences of the letters $ a $ and $ b $ in $ w $, respectively, that is, the first and the second coordinates of its Parikh vector. We make use of the fact that each block $ B_i $ in the abelian 8-power has the same number of $ a $'s and the same number of $ b $'s:
\begin{align*}
   |B_1|_a &= |B_2|_a = \cdots = |B_8|_a, \\
|B_1|_b &= |B_2|_b = \cdots = |B_8|_b. 
\end{align*}
We denote the number of $h(a)$'s  and $h(b)$'s in the middle block $\mu_1$ of $B_1$ by $  \alpha $ and $  \beta $, respectively. Then, the number of $a$'s and $b$'s in the first block can be calculated as
    \begin{align*} |aaba|_a+|a|_a+ \alpha |h(a) |_a+\beta |h(b)|_a= 4+ 5 \alpha+2 \beta ,\\
    |aaba|_b+|a|_b+ \alpha |h(a) |_b+\beta |h(b)|_b= 1+ 2 \alpha+3 \beta. \end{align*}
    Let $\alpha_2$ and $\beta_2$ denote the number of  $h(a)$'s and $h(b)$'s
    appearing in the second middle block $\mu_2$, respectively. Similarly,  one calculates the number of letters in the second block as
    \begin{align*} |baaaba|_a+|baba|_a+ \alpha_2 |h(a) |_a+\beta_2 |h(b)|_a= 6+ 5 \alpha_2+2 \beta_2 ,\\
    |baaaba|_b+|baba|_b+ \alpha_2 |h(a) |_b + \beta_2 |h(b)|_b= 4+ 2 \alpha_2+3 \beta_2. \end{align*}
    Therefore,
    \begin{align*}
        5(\alpha-\alpha_2) +2(\beta-\beta_2)  = 2,\\
        2(\alpha-\alpha_2) +3(\beta-\beta_2)  = 3.
    \end{align*}
    We have a unique solution to this equation and in fact $\alpha_2=\alpha$ and $\beta_2 =\beta-1$. Continuing in this fashion, we can compute the number of $h(a)$'s and $h(b)$'s appearing in each middle block as shown in Table \ref{table-8power}.

\begin{table}[h]
\caption{The structure of the blocks $ B_1, B_2, \dots, B_8 $ where  $  \alpha $ and $  \beta $ denote the number of $h(a)$'s  and $h(b)$'s in the middle block $\mu_1$ \\}
\label{table-8power}
\renewcommand{\arraystretch}{1}
\begin{tabular}{|c|c|c|c|}
\hline
 \rule{0pt}{5ex} Block&\shortstack{ \footnotesize The prefix of the block \\  \footnotesize which is a suffix of $h(a)$ or $h(b)$} 
& \shortstack{\footnotesize The number of $h(a)$'s and $h(b)$'s \\ \footnotesize in $\mu_i$}
& \shortstack{ \footnotesize The suffix of the block \\ \footnotesize which is a prefix of $h(a)$ or $h(b)$}  \\
\hline\hline
\rowcolor{gray!30} $B_1$ &$aaba$     & $\alpha, \quad \beta $      & $a$  \\
 $B_2$ &$baaaba$       & $\alpha, \quad (\beta-1) $      & $baba$   \\
\rowcolor{gray!30} $B_3$ &$b$         & $(\alpha +1), \quad (\beta-1) $   & $ab$  \\
 $B_4$ &$aaaba$  & $\alpha, \quad \beta$ & $\epsilon$  \\
\rowcolor{gray!30} $B_5$ &$\epsilon$   & $\alpha, \quad \beta $ & $abaaa$  \\
 $B_6$ &$ba$   & $(\alpha +1), \quad (\beta -1) $ & $b$  \\
\rowcolor{gray!30} $B_7$ &$abab$   & $\alpha, \quad (\beta-1) $ & $abaaab$  \\
 $B_8$ &$a$& $\alpha, \quad \beta $ & $abaa$  \\
\hline
\end{tabular}
\end{table}

The prefix given for $B_1$, which is $aaba$, and  the suffix given for $B_{8}$, which is $abaa$, are both factors of $h(a)$. 
We observe that the larger factor 
$$ v_1 B_1\cdots B_{8} v_{9}'=aba B_1B_2 \cdots B_{8}aba$$ 
containing the abelian $8$-power $B_1 \cdots B_{8}$ is the image of $I$ under $h$ 
$$h(I)= aba B_1B_2 \cdots B_{8}aba$$
where
$$I:= a X_1 a X_2 b X_3 a X_4 \epsilon X_5 a X_6 b X_7 a X_8 a    ,$$
$$\psi(X_i)=( \text{the number of} \ h(a)\text{'s} \ \text{in block} \ \mu_i \ , \ \text{the number of} \ h(b)\text{'s} \ \text{in block} \ \mu_i )$$ for each $i \in \{1,\dots,8 \}.$ Now, we see that $I$ realizes the following template:
\begin{align*}
    T(1.1)=[a,a,b, a,\epsilon,a,  b,a,a,  (0,-1),(1,0),(-1,1),(0,0),(1,-1),(-1,0),(0,1) ].
\end{align*}
Observe that $T(1.1)$ is a parent of $T_8$, and it is the first one selected by our sieve technique. We find that $  T(1.1) $ has 2 parents and 5 ancestors. \\

Next, we state the fact which plays a critical role in the template method. The word $h^\omega(a)$ contains the abelian power $B_1B_2\cdots B_{8}$ of the form in Table \ref{table-8power} if and only if an instance of an ancestor of  $  T(1.1) $   is a factor of $h^\omega(a)$. 
We  calculated that
$$\max_{i< j \le 8}\big|\sum_{r=i}^{j-1} d_r \cdot (1, \ldots, 1)\big|=1$$
is the maximum value
among all ancestors of $T(1.1)$.
The search bound for the ancestors of $T(1.1)$  using Lemma \ref{new search bound} is
$$ N+k-1 + (k-1)(N-2 + 1)= 7+8-1 +(8-1) (7-2+1)=56, $$
where $ N=\max\left\lbrace |h(a)|,|h(b)|\right\rbrace=7$, $ k=8 $.
We tested the factors of $ h^{\omega}(a) $ of length at most 56.
None of the factors of length $\leq 56$ realizes a template in $\operatorname{Anc}(T(1.1))$. Hence, we see that $h^\omega(a)$ does not contain an abelian power of the form $B_1B_2\cdots B_{8}$ specified in Table \ref{table-8power}. This concludes that Case 1.1 can never be the form of an abelian $8$-power contained in $h^\omega (a)$.\\ 

Using the scoring types given in \eqref{8 pref of templates}, we compute the other templates arising from $L_1$ as follows (the templates arising from the Cases 1.2, 1.3, 1.4):
\begin{align*}
    T(1.2)&=[a,a,b, b,\epsilon,a,  b,a,a,  (0,-1),(1,0),(0,0),(-1,1), (1,-1),(-1,0),(0,1) ],\\
    T(1.3)&=[a,a,b, b,\epsilon,b,  b,a,a,  (0,-1),(1,0),(0,0),(0,0), (0,0),(-1,0),(0,1) ],\\
    T(1.4)&= [a,a,b, a,\epsilon,b,  b,a,a,  (0,-1),(1,0),(-1,1),(1,-1), (0,0),(-1,0).(0,1) ].
\end{align*}
Similarly, the templates arising from $L_2$ are given as 
\begin{align*}
    T(2.1)&=[a,a,b, a,\epsilon,a,  b,a,a , (1,0),(-1,0),(1,0),(0,0),(-1,0),(1,0),(-1,0) ],\\
    T(2.2)&=[a,a,b, b,\epsilon,a,  b,a,a , (1,0),(0,-1),(0,1),(0,0),(-1,0),(1,0),(-1,0) ],\\
    T(2.3)&=[a,a,b, b,\epsilon,b,  b,a,a , (1,0),(0,-1),(0,1),(0,0),(0,-1),(0,1),(-1,0) ], \\
    T(2.4)&=[a,a,b, a,\epsilon,b,  b,a,a , (1,0),(-1,0),(1,0),(0,0),(0,-1),(0,1),(-1,0) ].
\end{align*}
We set$$\tau =\{ T(1.1),T(1.2),T(1.3),T(1.4),T(2.1),T(2.2),T(2.3),T(2.4)  \}$$
as the set of $8$ parent templates selected by our sieve technique. Applying the template method for $\tau$, we found that the set $\tau$ has 82 parents and $224$ ancestors.
Now, we make our crucial observation. We have seen that the arithmetic progressions of length $9$ in $P$ determine the existence of abelian $8$-powers in $h^\omega(a)$. As we eliminated the trivial arithmetic progressions, and the non-trivial ones $L_1$ and $L_2$ gave rise to the $8$ selected templates $ \tau$, we conclude that $h^\omega (a)$ contains an abelian $8$-power if and only if one of the ancestors of $\tau$ is realized by a factor of $ h^{\omega} (a)$.
We  computed that
$$\max_{i< j \le 8}\big|\sum_{r=i}^{j-1} d_r \cdot (1, \ldots, 1)\big|=1$$
is the maximum value among all ancestors of $\tau$.
Thus, the search bound for the ancestors of $\tau$  using Lemma \ref{new search bound} is again 56.\\

We give a parent, a grandparent, and a grand grandparent of the template $T(2.1)$: 
\begin{align*}
 T_p (2.1)&=[ a, a, a, a, a, a, a, a, a, (-1, 0), (1, 0), (0, 0), (-1, 0), (1, 0), (-1, 0), (1, 0)],\\
T_{gp}(2.1)& = [a, a, b, a, a, b, a, b, b, (0, 0), (1, 0), (-1, 0), (0, 0), (1, 0), (-1, 0), (1, -1)  ],\\
T_{ggp}(2.1)&= [b, a, b, a, a, b, a, b, a, (-1, 0), (1, 0), (-1, 0), (0, 0), (1, 0), (-1, 0), (1, 0)].
\end{align*}
The template $T_{ggp}(2.1)$ is indeed realized by the factor
$baabaaabaabaa$
of $h^2(a)$, and
$T_{gp}(2.1)$ is realized by the factor
\begin{align*}
   w_{gp} = ababaaabaaba&aabababababaaabaabaaabaabaaababa\\             &bababaaabaabaaabababababaaabaab 
 \end{align*}
of $h^3(a)$. The templates $T_{p} (2.1)$ and $T(2.1)$ are realized by some factors of $h(w_{gp})$ and $h^2(w_{gp})$, respectively.  Finally the word $h^3(w_{gp})$, and hence $h^6(a)$ contains an abelian $8$-power. The accompanying code is available at 
\url{https://github.com/nihantanisali/Abelian-Power-Free-Morphisms/blob/main/Section3_1}. 
\end{proof}

By applying the sieve technique, we confirm that the fixed point $ h^\omega(a) $ in Theorem~\ref{3 consecutive a 9AP}  does in fact contain abelian 8-powers. This highlights that the sieve technique is highly effective in detecting abelian $k$-powers, even in cases involving millions of templates to be tested.

\begin{observation} \label{observation 14} 
In Theorem 2.17, we defined the morphism  
\begin{align*}  
h(a) = abaabaababa, \
h(b) = babababab.  
\end{align*}  
We have shown that this morphism is abelian 16-power free in Theorem \ref{2 consecutive a 16AP}, which guarantees that its fixed point $ h^\omega(a) $ is also abelian 16-power free. However, further analysis suggests that the fixed point might actually be abelian 14-power free. The first 200 terms of $ h^\omega(a) $ contain an abelian 13-power, whereas the 5-th iterate $ h^5(a) $, which consists of 117611 letters, does not contain an abelian 14-power. To prove that $ h^\omega(a) $ is abelian 14-power free, we attempted to apply the template method directly. The template method is applicable as the operator norm of $M^{-1}$ is approximately $ 0.53 < 1 $, where $M$ is the frequency matrix of $h$. However, this verification requires generating all parents and ancestors of the template
$ T_{14} := [ \epsilon,\epsilon,\epsilon,\epsilon,\epsilon,\epsilon,\epsilon,\epsilon,\epsilon,\epsilon,\epsilon,\epsilon,\epsilon,\epsilon,\epsilon,\vec{0},\vec{0},\vec{0},\vec{0},\vec{0},\vec{0},\vec{0},\vec{0},\vec{0},\vec{0},\vec{0},\vec{0},\vec{0}
].$
The process involves more than 600 million parents, making it infeasible with the current version of the code and our setup. 
However, we are able to show that $ h^\omega(a) $ is indeed abelian 14-power free by our sieve technique.
\end{observation}

\begin{theorem} \label{T14-template}
    There is a morphism over a binary alphabet that is abelian $16$-power free, but not abelian $15$-power free  with an abelian $14$-power free fixed point. 
\end{theorem}

\begin{proof}
Let $ \Sigma = \left\lbrace a,b\right\rbrace  $. Recall that the morphism $h: \Sigma ^* \to \Sigma^ * $ in Theorem \ref{2 consecutive a 16AP} is defined by
	\begin{align*}
	&h(a) = abaabaababa\\
	&h(b) = babababab.
	\end{align*}
We have already proved that $h$ is an abelian $ 16 $-power free morphism. 
The image of the word $``abba" $ under the morphism $h$
$$ h(abba)= abaa \, \underline{ba}\, \underline{ab} \, \underline{ab} \, \underline{ab} \, \underline{ab} \, \underline{ab} \, \underline{ab} \, \underline{ab} \, \underline{ba} \, \underline{ba} \, \underline{ba} \, \underline{ba} \, \underline{ba} \, \underline{ba} \, \underline{ab}\, aababa $$
contains an abelian 15-power, hence we conclude that the morphism $h$ is not abelian $15$-power free. We will now prove that the fixed point $h^{\omega}(a)$ of the morphism $h$ is abelian $14$-power free by following the sieve technique performed in Proposition \ref{SievingT8}.
To obtain our result, we shall find exactly $32$ parent templates of 
$$ T_{14} := [ \epsilon,\epsilon,\epsilon,\epsilon,\epsilon,\epsilon,\epsilon,\epsilon,\epsilon,\epsilon,\epsilon,\epsilon,\epsilon,\epsilon,\epsilon,\vec{0},\vec{0},\vec{0},\vec{0},\vec{0},\vec{0},\vec{0},\vec{0},\vec{0},\vec{0},\vec{0},\vec{0},\vec{0}
]$$
such that it is sufficient to analyze their realizability by certain factors.
For this purpose,
let $G=\Z /  19 \Z$ and $f:\Sigma^*\to G$ be a monoid morphism such that
	$f(a)=\overline{1}, \ f(b)=\overline{3}.$
While $f$ vanishes on $h(\Sigma)$, that is, $f(h(a))=f(h(b))=\overline{0}$,
it takes the following values in Table \ref{Pref 14} when evaluated on $\operatorname{Pref}(h)$.
\begin{table}[h]
    \centering
    \begin{tabular}{|c|c|c|}
        \hline
        score & $\operatorname{Pref}(h(a))$& $\operatorname{Pref}(h(b))$\\
        \hline\hline
        $\overline{1}$ & $a$ &-\\ \hline
        $\overline{2}$ & - &-\\ \hline
        $\overline{3}$ & - &$b$\\ \hline
        $\overline{4}$ & $ab$ &$ba$\\ \hline
        $\overline{5}$ & $aba$ &-\\ \hline
        $\overline{6}$ & $abaa$ &-\\ \hline
        $\overline{7}$ & - & $bab$\\ \hline
        $\overline{8}$ & - & $baba$\\ \hline
        $\overline{9}$ & $abaab$ & - \\ \hline
    \end{tabular}\hspace{0.5cm}
    \begin{tabular}{|c|c|c|}
        \hline
        score & $\operatorname{Pref}(h(a))$& $\operatorname{Pref}(h(b))$\\
        \hline\hline
        $\overline{10}$ & $abaaba$ & - \\ \hline
        $\overline{11}$ & $abaabaa$ & $babab$ \\ \hline
        $\overline{12}$ & - & $bababa$ \\ \hline
        $\overline{13}$ & - & - \\ \hline
        $\overline{14}$ & $abaabaab$ & - \\ \hline
        $\overline{15}$ & $abaabaaba$ & $bababab$ \\ \hline
        $\overline{16}$ & - & $babababa$ \\ \hline
        $\overline{17}$ & -& -\\ \hline
        $\overline{18}$ & $abaabaabab$ & -\\ \hline
    \end{tabular}
    \vspace{0.5cm}
    \caption{Images of $\operatorname{Pref}( h(a)) $ and $\operatorname{Pref}(h(b))$ under $f$}
    \label{Pref 14}
\end{table}

Let
\begin{align*}
	P&: = f(\operatorname{Pref} (h) )=\{f(\epsilon), f(a),f(ab), f(aba) ,f(abaa), f(abaab),f(abaaba),f(abaabaa),\\
    &f(abaabaab),
    f(abaabaaba),f(abaabaabab), f(b),f(ba),f(bab), f(baba), f(babab),f(bababa),\\
    & f(bababab),f(babababa)\} =\{\overline{0}, \overline{1},\overline{3},\overline{4},\overline{5}, \overline{6}, \overline{7}, \overline{8}, \overline{9}, \overline{10}, \overline{11}, \overline{12}, \overline{14}, \overline{15}, \overline{16}, \overline{18}\}.
\end{align*}   
All the values in $P$ have a unique preimage in $\operatorname{Pref}(h)$ under $f$, except for $ \overline{4}, \overline{11}$ and $\overline{15}$: 
\begin{align}
    f(ab)&=f(ba)=\overline4, \nonumber\\
    f(abaabaa)&=f(babab)=\overline{11},  \nonumber\\
    f(abaabaaba)&=f(bababab)=\overline{15}.\label{pref of templates}
\end{align}
We also have the corresponding evaluations for the following suffixes:
\begin{align}
    f(aabaababa )&=f( bababab)=\overline{15},\nonumber \\
    f(baba)&=f(abab)=\overline{8} ,\nonumber \\
    f(ba)&=f(ab)=\overline{4}.\label{suf of templates T-14}
\end{align}
Suppose that $ h^\omega(a)$ contains an abelian 14-power. 
  Let $n$ be the minimum positive integer such that $h^n(a)$ contains an abelian 14-power but $h^{n-1}(a)$ does not. 
We know that $n \ge 4$ by Observation \ref{observation 14}.
Let $B_1, B_2, \dots, B_{14}$ occur consecutively in $ h^n(a)$, where the $B_i$'s are anagrams of each other. As shown in Figure \ref{fig:abelian-14-power},
we set $v_i' \in \operatorname{Suff} (h) $ and  $v_{i+1}\in \operatorname{Pref} (h) $ to be the prefix and the suffix of the block $B_i$, where $v_iv_i' \in \{h(a), h(b) \}$ and $i \in \{1,\ldots, 14\}$. We write
$$B_i = v_i' \, \mu_i \, v_{i+1}$$
for some word $\mu_i$.

\begin{figure}[h]
    \centering
    \begin{tikzpicture}
        \draw (-1,0.7) -- (12,0.7);
        \draw (-1,0) -- (12,0);
        \foreach \x/\i in {0/1, 2/2, 4/3} {
            \draw (\x,0) rectangle (\x+0.5,  0.7) node[midway] {\footnotesize $v_{\i}$};
            \draw (\x+0.5,0) rectangle (\x+1,0.7) node[xshift=-0.23cm,yshift=-0.3cm]   {\footnotesize $v'_{\i}$};
        }
        \node at (6.5,0.35) {\dots};
        \node at (6.5,-1) {\dots};
        
        \foreach \x/\i in {8/14, 10/15} {
            \draw (\x,0) rectangle (\x+0.5,0.7) node[midway] {\footnotesize $v_{\i}$};
            \draw (\x+0.5,0) rectangle (\x+1,0.7) node[xshift=-0.23cm,yshift=-0.3cm]    {\footnotesize $v'_{\i}$};
        }

        \draw[fill=gray!30,rounded corners] (0.5,-0.7) rectangle (2.5,-1.2) node[midway] {\footnotesize $B_1$};
        \draw[fill=gray!30,rounded corners] (2.5,-0.7) rectangle (4.5,-1.2) node[midway] {\footnotesize $B_2$};
        \draw[fill=gray!30,rounded corners] (8.5,-0.7) rectangle (10.5,-1.2) node[midway] {\footnotesize $B_{14}$};
        \foreach \x in {0.5, 2.5, 4.5, 8.5, 10.5} {
            \draw[thick,dashed,gray] (\x,-0.7) -- (\x,0);
        }
        \foreach \x in {0.5, 2.5, 4.5, 8.5, 10.5} {
        \draw[thick,dashed,gray] (\x,-0.7) -- (\x,0);}
     
     \foreach \x in {1, 2}{
            \node at (2*\x-0.5,0.3) {\footnotesize $\mu_{\x}$};
        }
    \node at (2*5-0.5,0.3) {\footnotesize $\mu_{14}$};
    \end{tikzpicture}
    \caption{An abelian 14-power and the corresponding blocks}
    \label{fig:abelian-14-power}
\end{figure}
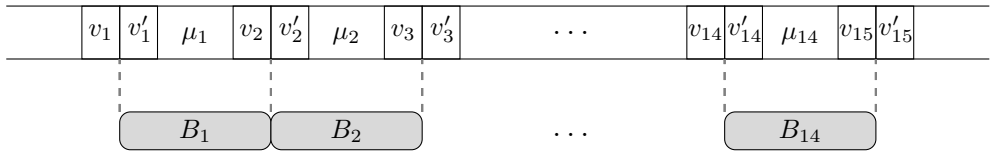
As the $B_i$'s are permutations of each other and $f(h(a))=f(h(b))=\overline{0}$, we infer the following equalities:
\begin{equation}\label{T14 equ class}
    f(v_1') + f(v_2) = f(v_2') + f(v_3) = \cdots = f(v_{14}') + f(v_{15}) .
\end{equation} 
The equalities in \eqref{T14 equ class} imply that the sequence $ (f(v_{i})) $ for $ 1 \leq i \leq 15 $ is an arithmetic progression of length $15$ in $ G $.
For $1 \leq i \leq 15$, we set
\begin{align*}
    p_i &= f(v_{i}).
\end{align*}
Assume that $ (p_i)_{i=1}^{15} $ is a trivial arithmetic progression. Among these trivial arithmetic progressions, three of them must be examined carefully. 
Observe that when
$$(p_i)_{i=1}^{15} = (\overline{4})_{i=1}^{15},$$
the prefixes taking the value $\overline 4$ under $f$ have the same Parikh vector by  \eqref{pref of templates}. When
$$(p_i)_{i=1}^{15} = (\overline{11})_{i=1}^{15} \text{ or } (\overline{15})_{i=1}^{15},$$
the corresponding suffixes have the same Parikh vector: $\psi(baba)=\psi(abab)$ and $\psi(ab)=\psi(ba)$ by \eqref{suf of templates T-14}. Now, assume that
$$(p_i)_{i=1}^{15} = (j)_{i=1}^{15} \text{ for some $j\neq \overline{11} ,\overline{15}$.}$$
Including the case $j=\overline{4} $, the prefixes $v_i$ have the same length, say $t_i$. 
We define $B_i'$ starting from $v_i$ to $ v_{i+1}$ for each $ 1\leq i \leq 14$.
Note that every block $B_i'$ has the same Parikh vector as $ B_i$. 
Moreover, each $B_i'$ is the image of a word $A_i$ for some word $A_i$ under $h$. The word $A_1\cdots A_{14}$ is an abelian 14-power in $h^{n-1}(a)$. This contradicts the minimality of $n$.
For the cases $ j=\overline{11},\overline{15}$, we use the suffixes to define the corresponding block $B_i'$ as in Proposition \ref{SievingT8}.  Therefore, abelian $14$-powers cannot emerge from the trivial arithmetic progressions. \\

Now, we concentrate on the non-trivial arithmetic progressions of length $15$. Notice that the set $P$ itself is an arithmetic progression of length 16, and this can be seen from the fact that $G \setminus P = \{\overline{13}, \overline{17}, \overline{2}\}$ is an arithmetic progression of length 3. Therefore, there are exactly 4 non-trivial arithmetic progressions of length $15$ in $P$:
\begin{align*}
L_1: (p_i) &= \{ \overline{10}, \overline{14}, \overline{18}, \overline{3}, \overline{7}, \overline{11}, \overline{15}, \overline{0}, \overline{4}, \overline{8}, \overline{12}, \overline{16}, \overline{1}, \overline{5}, \overline{9} \}, \\
L_2: (p_i) &= \{ \overline{6}, \overline{10}, \overline{14}, \overline{18}, \overline{3}, \overline{7}, \overline{11}, \overline{15}, \overline{0}, \overline{4}, \overline{8}, \overline{12}, \overline{16}, \overline{1}, \overline{5} \}, \\
L_3: (p_i)& = \{ \overline{5}, \overline{1}, \overline{16}, \overline{12}, \overline{8}, \overline{4}, \overline{0}, \overline{15}, \overline{11}, \overline{7}, \overline{3}, \overline{18}, \overline{14}, \overline{10}, \overline{6} \}, \\
L_4: (p_i) &= \{ \overline{9}, \overline{5}, \overline{1}, \overline{16}, \overline{12}, \overline{8}, \overline{4}, \overline{0}, \overline{15}, \overline{11}, \overline{7}, \overline{3}, \overline{18}, \overline{14}, \overline{10} \}. 
\end{align*}
Notice that $L_1$ and $L_2$ are shifts of each other. Similarly, $L_3$ and $L_4$ are shifts of each other.
Equations in \eqref{pref of templates} indicate that the sixth, the seventh, and the ninth terms of $L_1$, namely $\overline{11}$, $\overline{15}$ and $\overline{4}$ are images of two distinct prefixes, respectively. Therefore, we have $8$ prefix sequences that are mapped to $ L_1$ under $f$. Similarly, we have $8$ prefix sequences that are mapped to each of $L_2$, $L_3$ and $L_4$ under $f$. Hence, we have $32$ cases to be analyzed in total, and these cases will give us $32$ parent templates of $T_{14}$ to understand the existence of abelian $14$-powers in our infinite fixed point.\\

\noindent\underline{\textit{Case $L_1$}}: 
Our goal is to find $ v_1, \ldots, v_{15} $ and $ v_1^\prime, \ldots, v_{15}^\prime $ corresponding to $L_1$. For each $i \in \{1,\dots,15\}$, the prefix $v_{i}$ and the suffix $v_i'$ are determined according to  $p_i=f(v_{i})$ scoring types. Recall that there are two different prefixes and suffixes in 3 different scoring types by \eqref{pref of templates}.
Thus, for the abelian $14$-power $B_1B_2\cdots B_{14}$,  there are $8$ choices for the prefixes. 
The first subcase is shown in Table  \ref{table:imp1} in detail:
\vspace{-0.2em}
$$    v_6 = babab, \ v_7=  bababab, \ v_9=  ab.
$$
\vspace{-0.5em}
\noindent \underline{\textit{Case 1.1}}:
We give the full list of $ v_1,\ldots,v_{15} $ and $v_1',\ldots,v_{15}'$ below:
\vspace{-0.2em}
\begingroup
\small
\begin{align*}
 &v_1=abaaba, v_2=abaabaab, v_3=abaabaabab, v_4= b, v_5=bab, v_6=babab, v_7=bababab, v_8=\epsilon,    \\
    &v_9=ab, v_{10}=baba, v_{11}=bababa, v_{12}=babababa, v_{13}=a, v_{14}=aba, v_{15}=abaab,\\
&v_1'=ababa, v_2'=aba, v_3'= a, v_4'=abababab, v_5'=ababab, v_6'=abab, v_7'=ab,  v_8'=\epsilon, \\
    & v_9'=aabaababa, v_{10}'=baba, v_{11}'=bab, v_{12}'=b, v_{13}'=baabaababa, v_{14}'=abaababa, v_{15}' = aababa.
\end{align*}
\endgroup
Recall that  $|w|_a$ and $|w|_b$ are the first and the second coordinates of the Parikh vector of the word $w$, respectively.  We use the fact that the number of $a$'s and $b$'s are the same in each $B_i$:
\vspace{-0.8em}
\begin{align*}|B_1|_a=|B_2|_a=\cdots=|B_{14}|_a,\\ |B_1|_b=|B_2|_b=\cdots=|B_{14}|_b.\end{align*} 
We let $  \alpha $ and $  \beta $ be the number of $h(a)$'s  and $h(b)$'s in the middle block $\mu_1$ of $B_1$, respectively. Then, using the equations above, the structures of the blocks $ B_1, B_2, \dots, B_{14} $ (in fact, the structures of the middle blocks $\mu_1,\ldots, \mu_{14}$) are determined as in Table \ref{table:imp1}.
\begin{table}[H] 
\centering
\caption{The structure of the blocks $ B_1, B_2, \dots, B_{14} $ where  $  \alpha $ and $  \beta $ denote the number of $h(a)$'s  and $h(b)$'s in the middle block $\mu_1$ \\}
\label{table:imp1}
\renewcommand{\arraystretch}{1}
\begin{tabular}{|c|c|c|c|}
\hline
Block \rule{0pt}{5ex}&\shortstack{ \footnotesize The prefix of the block \\  \footnotesize which is a suffix of $h(a)$ or $h(b)$} 
& \shortstack{\footnotesize The number of $h(a)$'s and $h(b)$'s \\ \footnotesize in $\mu_i$}
& \shortstack{ \footnotesize The suffix of the block \\ \footnotesize which is a prefix of $h(a)$ or $h(b)$}  \\
\hline\hline
\rowcolor{gray!30} $B_1$ &$ababa$     & $\alpha, \quad \beta $      & $abaabaab$  \\
 $B_2$ &$aba$       & $\alpha, \quad \beta$      & $abaabaabab$   \\
\rowcolor{gray!30} $B_3$ &$a$         & $(\alpha +1), \quad \beta $   & $b$  \\
 $B_4$ &$abababab$  & $(\alpha+1), \quad (\beta-1 )$ & $bab$  \\
\rowcolor{gray!30} $B_5$ &$ababab$   & $(\alpha+1), \quad (\beta-1) $ & $babab$  \\
 $B_6$ &$abab$   & $(\alpha +1), \quad (\beta -1) $ & $bababab$  \\
\rowcolor{gray!30} $B_7$ &$ab$   & $(\alpha+1), \quad \beta $ & $\epsilon$  \\
 $B_8$ &$\epsilon$& $(\alpha+1), \quad \beta $ & $ab$  \\
\rowcolor{gray!30}$B_{9}$ &$aabaababa$ & $\alpha, \quad \beta $ & $baba$  \\
 $B_{10}$ &$babab$ & $(\alpha+1), \quad (\beta-1) $ & $bababa$  \\
\rowcolor{gray!30}$B_{11}$ &$bab$ & $(\alpha+1), \quad (\beta-1) $ & $babababa$  \\
 $B_{12}$ &$b$ & $(\alpha+1) , \quad \beta $ & $a$  \\
\rowcolor{gray!30}$B_{13}$ &$baabaababa$ & $\alpha, \quad \beta $ & $aba$  \\
$B_{14}$ &$abaababa$   & $\alpha , \quad \beta$      & $abaab$   \\
\hline
\end{tabular}
\end{table}

The prefix given for $B_1$, which is $ababa$, and  the suffix given for $B_{14}$, which is $abaab$, are both factors of $h(a)$. 
We observe that the larger factor 
$$ v_1 B_1\ldots B_{14} v_{15}'=abaaba B_1B_2 \cdots B_{14}aababa$$ 
containing the abelian $14$-power $B_1\ldots, B_{14}$ is the image of $I$ under $h$, in other words,
$$h(I)= abaaba B_1B_2 \cdots B_{14}aababa$$
where
$$I:= aX_1aX_2aX_3bX_4bX_5bX_6bX_7\epsilon X_8aX_9bX_{10}bX_{11}bX_{12}aX_{13}aX_{14}a$$
$$\psi(X_i)=( \text{the number of} \ h(a)\text{'s} \ \text{in block} \ \mu_i \ , \ \text{the number of} \ h(b)\text{'s} \ \text{in block} \ \mu_i )$$ for each $i \in \{1,\dots,14 \}.$  
Notice that $I$ realizes the following template:
\begin{align*}
    T(1.1) := [a, a, a, b, b,&\ b, b, \epsilon, a, b, b, b, a, a, a,  (0,0),  (1,0), (0,-1),  (0,0),  (0,0), \\  
    &(0,1), (0,0), (-1,0),   (1,-1),   (0,0),  (0,1),  (-1,0), (0,0)].
\end{align*}
We examine the remaining seven cases corresponding to different choices of $ v_6$, $v_7$ and $v_9$ using the same technique as in Case 1.1. 
The templates arising from these cases (the remaining 7 cases of this case) are listed below:
\vspace{-0.5em}
\begin{align*}
    T(1.2) := [a, a, a, b, b,&\ b, b, \epsilon, b, b, b, b, a, a, a,  (0,0), (1,0),  (0,-1),  (0,0),  (0,0), \\  
    &(0,1),  (0,0),  (0,-1),   (0,0),  (0,0),  (0,1),  (-1,0),   (0,0)],
\end{align*}
\begin{align*}
    T(1.3) := [a, a, a, b, b,&\ b, a, \epsilon, b, b, b, b, a, a, a,  (0,0),  (1,0),  (0,-1),  (0,0),  (-1,1), \\  
    &(1,0),   (0,0),  (0,-1),   (0,0),   (0,0), (0,1),  (-1,0),   (0,0)],
\end{align*}
\begin{align*}
    T(1.4) := [a, a, a, b, b,&\ b, a, \epsilon,  a, b, b, b, a, a, a,  (0,0),  (1,0),  (0,-1),  (0,0),  (-1,1), \\  
    &(1,0),   (0,0), (-1,0),   (1,-1),   (0,0),  (0,1),  (-1,0),   (0,0)],
\end{align*}
\begin{align*}
    T(1.5) := [a, a, a, b, b,&\ a, a, \epsilon, a, b, b, b, a, a, a,  (0,0),  (1,0),  (0,-1),  (-1,1),  (0,0), \\  
    &(1,0),   (0,0),  (-1,0),   (1,-1),   (0,0),  (0,1), (-1,0),   (0,0)],
\end{align*}
\begin{align*}
    T(1.6) := [a, a, a, b, b,&\ a, a, \epsilon, b, b, b, b, a, a, a, (0,0),  (1,0),  (0,-1),  (-1,1),  (0,0), \\  
    &(1,0),   (0,0),  (0,-1),   (0,0),   (0,0),  (0,1),  (-1,0),   (0,0)],
\end{align*}
\begin{align*}
    T(1.7) := [a, a, a, b, b,&\ a, b, \epsilon, b, b, b, b, a, a, a,  (0,0),  (1,0),  (0,-1),  (-1,1), (1,-1), \\  
    &(0,1),  (0,0),  (0,-1),   (0,0),   (0,0),  (0,1),  (-1,0),   (0,0)],
\end{align*}
\begin{align*}
    T(1.8) := [a, a, a, b, b,&\ a, b, \epsilon, a, b, b, b, a, a, a,  (0,0), (1,0),  (0,-1),  (-1,1), (1,-1), \\  
    &(0,1),   (0,0),  (-1,0),   (1,-1),   (0,0),  (0,1),  (-1,0),   (0,0)].
\end{align*}

Observe that the templates emerging from $L_2$ can be computed from the templates obtained by $L_1$ as these two arithmetic progressions differ by only one term, and they are given by
\begin{align*}
    T(2.1) := [a, a, a, a, b,&\ b,  b, b, \epsilon, a, b, b, b, a, a,  (0,0), (0,0),  (1,0),  (0,-1),  (0,0), \\  
    &  (0,0),  (0,1),   (0,0), (-1,0),   (1,-1),  (0,0),  (0,1),  (-1,0)],
\end{align*}
\begin{align*}
    T(2.2) := [a, a, a, a, b,&\ b, b, b, \epsilon, b, b, b, b, a, a, (0,0),  (0,0),  (1,0),  (0,-1),  (0,0),  \\  
    &(0,0),   (0,1),   (0,0), (0,-1),  (0,0),   (0,0), (0,1),  (-1,0)],
\end{align*}
\begin{align*}
    T(2.3) := [a, a, a, a, b,&\ b, b, a, \epsilon, b, b, b, b, a, a,  (0,0),  (0,0),  (1,0),  (0,-1),  (0,0),  \\  
    & (-1,1),  (1,0),   (0,0),  (0,-1),  (0,0),  (0,0),  (0,1), (-1,0)],
\end{align*}
\begin{align*}
    T(2.4) := [a, a, a, a, b,&\ b, b, a, \epsilon,  a, b, b, b, a, a,  (0,0),  (0,0),  (1,0),  (0,-1),  (0,0),  \\  
    & (-1,1),  (1,0),  (0,0), (-1,0),   (1,-1),  (0,0),  (0,1),  (-1,0)],
\end{align*}
\begin{align*}
    T(2.5) := [a, a, a, a, b,&\ b, a, a, \epsilon, a, b, b, b, a, a,  (0,0),  (0,0),  (1,0),  (0,-1), (-1,1),  \\  
    & (0,0),  (1,0),  (0,0),  (-1,0),   (1,-1),  (0,0),  (0,1),  (-1,0)],
\end{align*}
\begin{align*}
    T(2.6) := [a, a, a, a, b,&\ b, a, a, \epsilon, b, b, b, b, a, a,  (0,0),  (0,0),  (1,0),  (0,-1),  (-1,1), \\  
    & (0,0),  (1,0),   (0,0),  (0,-1),   (0,0),  (0,0),  (0,1),  (-1,0)],
\end{align*}
\begin{align*}
    T(2.7) := [a, a, a, a, b,&\ b, a, b, \epsilon, b, b, b, b, a, a, (0,0),  (0,0),  (1,0),  (0,-1),  (-1,1),  \\  
    & (1,-1),  (0,1),  (0,0),  (0,-1),  (0,0),  (0,0),  (0,1),  (-1,0)],
\end{align*}
\begin{align*}
    T(2.8) := [a, a, a, a, b,&\ b, a, b, \epsilon, a, b, b, b, a, a,  (0,0),  (0,0),  (1,0),  (0,-1),  (-1,1),  \\  
    &(1,-1),   (0,1),  (0,0),  (-1,0), (1,-1),  (0,0),  (0,1),  (-1,0)].
\end{align*}
The templates obtained from $L_3$ are 
\begin{align*}
    T(3.1) := [a, a, b,  b, b,&\ b,  \epsilon, b, a,  b, b,  a, a, a, a,  (0,-1),  (1,0), (0,0), (0,0),  (0,0),\\  
    &  (0,0),  (-1,1),  (1,-1),  (0,0),  (-1,0), (0,1),   (0,0), (0,0)],
\end{align*}
\begin{align*}
    T(3.2) := [a, a, b,  b, b,&\ b,  \epsilon, b, b,  b, b,  a, a, a, a,   (0,-1), (1,0), (0,0), (0,0), (0,0),\\  
    & (0,0), (0,0),   (0,0),  (0,0),  (-1,0), (0,1),   (0,0), (0,0)],
\end{align*}
\begin{align*}
    T(3.3) := [a, a, b,  b, b,&\ b,  \epsilon, a, b, b, b, a, a, a, a,  (0,-1),  (1,0), (0,0), (0,0), (0,0),\\  
    &(-1,1),   (1,-1),  (0,0), (0,0),   (-1,0),  (0,1),  (0,0), (0,0)],
\end{align*}
\begin{align*}
    T(3.4) := [a, a, b,  b, b,&\ b,  \epsilon, a, a,  b, b, a,  a, a, a,   (0,-1),  (1,0),  (0,0),  (0,0),  (0,0),\\  
    & (-1,1),   (0,0),   (1,-1),  (0,0),  (-1,0),  (0,1),   (0,0),  (0,0)],
\end{align*}
\begin{align*}
    T(3.5) := [a, a, b,  b, b,&\ a,  \epsilon, a, a,  b, b,  a, a, a, a,   (0,-1), (1,0), (0,0), (0,0),  (-1,1),\\  
    &   (0,0),  (0,0),   (1,-1),  (0,0),  (-1,0),  (0,1),   (0,0),  (0,0)],
\end{align*}
\begin{align*}
    T(3.6) := [a, a, b, b, b,&\ a, \epsilon, a, b,  b, b, a, a, a, a,  (0,-1),  (1,0),  (0,0),  (0,0), (-1,1), \\  
    &  (0,0), (1,-1),  (0,0),  (0,0),  (-1,0),  (0,1),   (0,0),  (0,0)],
\end{align*}
\begin{align*}
    T(3.7) := [a, a, b, b, b,&\ a,  \epsilon, b, b,  b, b,  a,  a, a, a,   (0,-1),  (1,0),  (0,0),  (0,0), (-1,1),\\  
    & (1,-1),  (0,0),  (0,0),  (0,0),  (-1,0), (0,1),  (0,0),  (0,0)],
\end{align*}
\begin{align*}
    T(3.8) := [a, a, b, b, b,&\ a,  \epsilon, b, a,  b, b,  a, a, a, a,  (0,-1),  (1,0), (0,0), (0,0),  (-1,1),  \\  & (1,-1),  (-1,1),   (1,-1),   (0,0),  (-1,0),  (0,1),   (0,0), (0,0)].
\end{align*}
Finally, the templates arising from $L_4$ can be calculated with the aid of the templates obtained by $L_3$, and they are given by 
\begin{align*}
    T(4.1) := [a, a, a, b, b,&\ b, b, \epsilon, b, a, b, b,  a, a, a, (0,0), (0,-1), (1,0), (0,0), (0,0), \\  
    &(0,0), (0,0), (-1,1), (1,-1), (0,0), (-1,0), (0,1),  (0,0)],
\end{align*}
\begin{align*}
    T(4.2) := [a, a, a, b, b,&\ b, b, \epsilon, b, b, b, b,  a, a, a,  (0,0),  (0,-1),  (1,0),  (0,0),  (0,0), \\  
    &(0,0),   (0,0),  (0,0),   (0,0),   (0,0), (-1,0),  (0,1),   (0,0)],
\end{align*}
\begin{align*}
    T(4.3) := [a, a, a, b, b,&\ b, b, \epsilon, a, b, b, b,  a, a, a, (0,0), (0,-1),  (1,0),  (0,0),  (0,0), \\  
    &(0,0),   (-1,1),  (1,-1),   (0,0),  (0,0),  (-1,0),  (0,1),   (0,0)],
\end{align*}
\begin{align*}
    T(4.4) := [a, a, a, b, b,&\ b, b, \epsilon, a, a, b, b,  a, a, a,  (0,0), (0,-1), (1,0), (0,0), (0,0), \\  
    &(0,0), (-1,1), (0,0), (1,-1), (0,0), (-1,0), (0,1),  (0,0)],
\end{align*}
\begin{align*}
    T(4.5) := [a, a, a, b, b,&\ b, a, \epsilon, a, a, b, b,  a, a, a,  (0,0),  (0,-1),  (1,0),  (0,0), (0,0), \\  
    &(-1,1), (0,0), (0,0), (1,-1), (0,0), (-1,0), (0,1), (0,0)],
\end{align*}
\begin{align*}
    T(4.6) := [a, a, a, b, b,&\ b, a, \epsilon, a, b, b, b,  a, a, a,  (0,0), (0,-1), (1,0), (0,0), (0,0), \\  
    &(-1,1), (0,0), (1,-1), (0,0), (0,0), (-1,0), (0,1),  (0,0)],
\end{align*}
\begin{align*}
    T(4.7) := [a, a, a, b, b,&\ b, a, \epsilon, b, b, b, b,  a, a, a,  (0,0), (0,-1), (1,0), (0,0), (0,0), \\  
    &(-1,1), (1,-1), (0,0), (0,0), (0,0), (-1,0), (0,1),  (0,0)],
\end{align*}
\begin{align*}
    T(4.8) := [a, a, a, b, b,&\ b, a, \epsilon, b, a, b, b,  a, a, a,  (0,0), (0,-1), (1,0), (0,0), (0,0), \\  
    &(-1,1), (1,-1), (-1,1), (1,-1), (0,0), (-1,0), (0,1),  (0,0)].
\end{align*}
Define $$\tau =\{ T(i.j) \mid 1\leq i \leq 4 ,\, 1\leq j \leq 8\}$$
to be the set of $32$ parent templates of $T_{14}$ selected by our sieve technique. Applying the template method for $\tau$, we found that the set $\tau$  has $268$ parents and $1244$ ancestors appearing in $14$ generations. We computed that
    $$\max_{i< j \le 14}\big|\sum_{r=i}^{j-1} d_r \cdot (1, \ldots, 1)\big|=1$$ is the maximum value over all ancestors of $\tau$. Thus, the search bound for the ancestors of $\tau$  using Lemma \ref{new search bound} is
$$ N+k-1 + (k-1)(N-2 + 1)= 11+14-1 +(14-1) (11-2+1)=154,$$
where $N=11$ and $k=14$.
(The search bound coming from the $\Delta$-calculation is 869.)
\\

Now, we emphasize our key observation. We have seen that the arithmetic progressions of length $15$ in $P$ determine the existence of abelian $14$-powers in $h^\omega(a)$. As we eliminated the trivial arithmetic progressions, and the non-trivial ones $L_1, L_2, L_3, L_4$ yielded a set $\tau$ of $32$ selected templates, we conclude that $h^\omega (a)$ contains an abelian $14$-power if and only if one of the ancestors $\mathcal{T}$ of $\tau$ is realized by a factor of length at most 154.
Therefore, we concluded that there are no abelian 14-powers in our fixed point which arise from Case 1, 2, 3 or 4 (namely from $L_1, L_2, L_3, L_4$). Hence, there is no abelian 14-power in the aforementioned fixed point of $h$. 
Consequently, the morphism $h$ is abelian $16$-power free but not abelian 15-power free, with an abelian $14$-power free fixed point $h^\omega (a)$.
See
\url{https://github.com/nihantanisali/Abelian-Power-Free-Morphisms/blob/main/Section3_2}
for the accompanying code. 
\end{proof}
The results presented in this section demonstrate the strength of the sieve technique as a versatile tool for determining whether the fixed point of a morphism contains abelian powers. We applied this technique to show that the fixed point of a certain morphism avoids abelian $14$-powers, thereby establishing Theorem~\ref{T14-template}. On the other hand, we used the same technique to verify that the fixed point of a different morphism contains an abelian $8$-power. These complementary outcomes highlight the technique’s effectiveness in both ruling out and detecting abelian powers in the fixed point of a morphism.

\subsection{\texorpdfstring{An Overview of Morphisms and Abelian $k$-Power Freeness}{An Overview of Morphisms and Abelian k-Power Freeness}}
Richomme and Séébold \cite{Ric} gave an overview of morphisms and $k$-power freeness. Then, they presented several conjectures, and they partially solved one of those by giving
binary morphisms which are not $ k $-power free, but which are $ (k + 1) $-power free and have $k$-power free fixed points for any integer $ k \ge 3 $.
 Keränen  \cite{keranen 1986} addressed this problem by providing the following binary morphism 
	$$ g =   \begin{dcases}
	& a \mapsto aabaabbaabbabbaabaababbabbaabaabb \\
	& b \mapsto aabbabbaabaababbabbaabaabbaabbabb. \\
	\end{dcases}$$
While this morphism itself is not 3-power free since $ g(baba) $ contains a 3-power, its infinite fixed point $h^\omega(a)$ is 3-power free.
	However, they \cite{keranen 1986,Ric} studied power-freeness, not  abelian power-freeness.  As noted by Currie and Rampersad in \cite{C-R2012}, it is also of interest to examine morphisms that have abelian $k$-power free fixed points without themselves being $k$-power free morphisms. \\

In the literature, there are examples of morphisms such that their fixed points are  abelian $k$-power free while the morphisms themselves are not abelian $k$-power free. Two notable examples include the following morphisms:

\begin{itemize}
	\item The morphism in \cite{CassCurr} 
$$
\sigma_4 =
\begin{dcases}
\begin{aligned}
0 &\mapsto 03 &\qquad 3 &\mapsto 1 \\
1 &\mapsto 43 &\qquad 4 &\mapsto 01
\end{aligned}
\end{dcases}
$$
	is not abelian 3-power free since $ \sigma_4 (0011014) = 0 3034 3430 3430 1 $ contains an abelian 3-power. However, Cassaigne et al. showed that $ \sigma_4^{\omega}(0) $ contains no abelian cubes.
	\item The morphism in \cite{Rao-R} $$
\sigma_6 =
\begin{dcases}
\begin{aligned}
	& a \mapsto ace & b  \mapsto adf && c \mapsto bdf
	\\ & d  \mapsto bdc & e \mapsto aef && f  \mapsto bce \\
\end{aligned}
\end{dcases}
$$
	is not abelian square-free since $ \sigma_6 (afd) = a ceb ceb dc $ contains an abelian square. However, Rao and Rosenfeld proved that  $ \sigma_6^{\omega}(a) $ is abelian square-free.
\end{itemize}
\vspace{3mm}

In this work, we extended this line of inquiry into the abelian context by providing an example of a morphism defined over a binary alphabet. In Theorem~\ref{T14-template}, we observed the same phenomenon: we constructed a morphism over a binary alphabet that is abelian 16-power free but not abelian 15-power free, yet with an abelian 14-power free fixed point. This example is significant because it provides the first known case in the abelian setting where there is a two-level gap between the power-freeness of the morphism and that of its fixed point.
\section{On the Uniform Bounds of Lengths of Words Avoiding Abelian Powers}\label{Uniform Section}
In this section, we consider words on the binary alphabet $\{a, b\}$.
In the study of infinite words, a fascinating problem is the construction of words that avoid certain factors while also being abelian power-free.
We are interested in the length of such words. 
Our key result in this direction is Theorem \ref{4 consecutive a 6AP}. It states that there exists an infinite word $\Omega_1$ with two properties: 
\begin{enumerate}
    \item  avoiding the factors $aaaaa$  and $bb$,
    \item  being abelian 6-power free. 
\end{enumerate}
In the spirit of this result, we start by asking the same question for smaller abelian powers. We refer the reader to Table \ref{table:imp} for the results of this section.\\

We discovered that any word of length 9 avoiding the factor $bb$ contains an abelian 3-power, see \autoref{fig: word 9}.
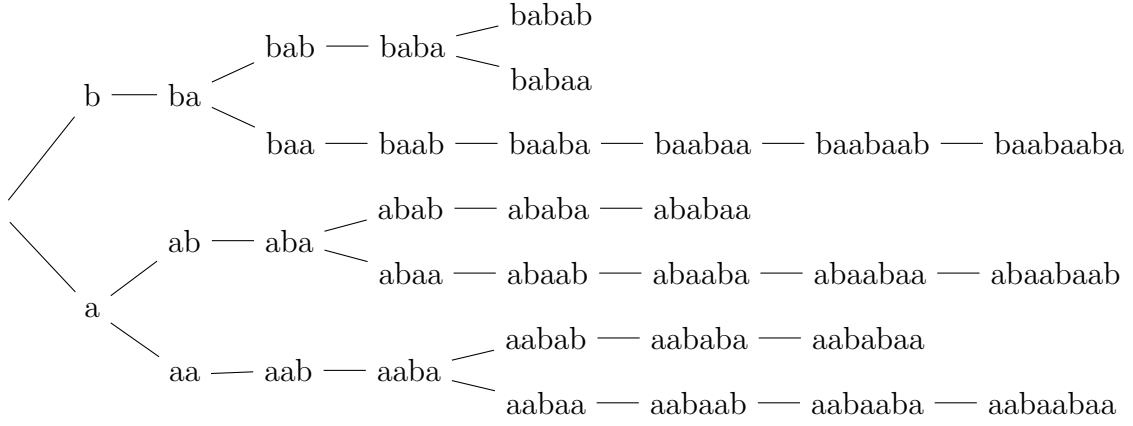
\begin{figure}[H]
    \centering
    \begin{forest}
    for tree={
    grow=east}
      [
      [a,  
      [aa, [aab, [aaba,[aabaa,[aabaab,[aabaaba,[aabaabaa]]  ]],[aabab,[aababa,[aababaa]] ] ]   ] ],
      [ab, [aba,[abaa, [abaab,[abaaba,[abaabaa,[abaabaab] ] ]]] ,[abab, [ababa,[ababaa ] ]] ]     ]
      ],
      [b,[ba,
      [baa,[baab,[baaba,
      [baabaa,[baabaab,[baabaaba]]]
      ]]],
      [bab,[baba,[babaa],[babab]]]
      ]]]
    \end{forest}
    \caption{Maximal-length words avoiding $bb$ and abelian $3$-power}
    \label{fig: word 9}
\end{figure}
Guided by this observation, we now detail the backtracking search used in our computations to prove the propositions in this section, see  \url{https://github.com/nihantanisali/Abelian-Power-Free-Morphisms/blob/main/Section4}.
\begin{enumerate}
    \item Start with the empty word.
    \item Extend the current word by appending (right-concatenating) $a$ or $b$.
    \item After each extension:
    \begin{itemize}
        \item Check whether any forbidden factor appears.\\
              If yes, discard this branch.
        \item Otherwise, check whether the word contains an abelian $k$-power.\\
              If yes, discard this branch.
        \item If both checks pass, recursively extend the  word again.
    \end{itemize}
    \item When the desired length $n$ is reached, record the word as valid.
\end{enumerate}

\begin{proposition} \label{p41}
Words of length 18 that avoid the factor $bb$ must contain an abelian 4-power. This bound is tight because there exists a word of length $17$ of this form that does not contain an abelian $4$-power: $babaaabaaabaaabab$. 
\end{proposition}

We now compare our results with the results in the literature. By \cite{fici}, we know that every binary word of length 10 contains an abelian cube.
If we restrict our attention to $bb$-free words, in Figure \ref{fig: word 9}, we reduce the length in \cite{fici} by 1. Dekking \cite{dekking} constructed an abelian $4$-power free infinite binary word. This contrast is sharper: once the factor $bb$ is forbidden, in Proposition \ref{p41}, this bound collapses from infinity to 18.\\

\begin{proposition}\label{aaa,bb,25,5power}
Any word of length 25 that avoids the factors $ aaa $ and $ bb $ must contain an abelian 5-power. This bound is tight because there exists a word of length 24 of this form that
does not contain an abelian $5$-power: $aababaababaababaababaaba $.
\end{proposition}
\begin{proposition}\label{aaaa,bb,78,5power}
Any word of length 78 that avoids the factors $aaaa$ and $ bb $  must contain an abelian 5-power. This bound is tight because there exists a word of length 77 of this form that
does not contain an abelian $5$-power:  $$aaababababaaababababaaabaaabaaabaaababababaaabaaabaaabaaababababaaababababaaa.$$
\end{proposition}
\section{A Binary Morphism Which is not Abelian Power-Free but its Fixed Point is Abelian 5-Power Free}\label{not_abelian_power_free}
In this section, we give a binary morphism which is not abelian power-free, yet has an abelian $5$-power free fixed point. Hence, we also give a resolution of the 2-ball-box distribution problem.\\

Initially, we attempted to prove Theorem \ref{5free} using an argument analogous to that of
Theorem \ref{4 consecutive a 6AP}. However, we could not find an abelian 5-power free morphism $h$ on the binary alphabet $\{a, b\}$ such that the fixed point 
$h^\omega(a)$ avoids the factor $bb$. We conjecture that no abelian 5-power free morphism on $\{a, b\}$ exists whose fixed point also avoids
the factor $bb$, and we pose this as an open problem in our last section. Furthermore, to the best of our
knowledge, no known morphism avoids all abelian powers while having a fixed point that
is abelian 5-power free. Our next result gives the first example of such a morphism, and the proof applies the template method. This reveals an infinite gap (in terms of the degree of abelian
power avoidance) between a morphism and its fixed point.

\begin{theorem}\label{5free}
    The fixed point $h^\omega(a)$ of the morphism 
    \begin{align*}
        h: a&\mapsto aaaab\\
           b&\mapsto ababab
    \end{align*} 
    is abelian $5$-power free and it avoids the factor $bb$. However, the morphism itself is not abelian power-free for any $k \ge 2$. 
\end{theorem}
\begin{proof}
    Let $k\geq 2$. First, we will show that the morphism is not abelian $k$-power free.
    Note that the image 
    $$h( \underbrace{b\cdots b}_{k-1 \text{ many}})= \underbrace{ab \cdots  ab }_{3k-3 \text{ many }} $$
    of the abelian $k$-power free word $b^{k-1}$ is not abelian $k$-power free as $3k-3 \geq k$.
    To check the condition on $bb$, just note that in $h(a)$ and $h(b)$ there is no $bb$ as a factor, and $h(aa)$, $h(ab) $, $h(ba)$ do not produce the factor $bb$ as well. 
    Hence, $h^\omega(a)$ avoids the factor $bb$.
    Note that the frequency matrix of the morphism is 
    $$M=\begin{pmatrix}
        4 &1\\
        3& 3
    \end{pmatrix}$$
    and the operator norm of $M^{-1}$ is approximately $0.63<1$. Hence, the template method is applicable to decide whether $h^\omega(a)$ is abelian $5$-power free. The template
    $$
T_5 = [ \epsilon, \epsilon, \epsilon, \epsilon, \epsilon, \epsilon, (0,0), (0,0), (0,0), (0,0)]
 $$
    has $16163$ parents and no other ancestors. 
    We computed that
    $$\max_{i< j \le 5}\big|\sum_{r=i}^{j-1} d_r \cdot (1, \ldots, 1)\big|=2$$ is the maximum value over all ancestors of $T_5$.
    By Lemma \ref{new search bound}, we see that $h^\omega (a)$ contains an abelian $5$-power if and only if one of the ancestors of $T_5$ is realized by a factor of length at most
$$ N+k-1 + (k-1)(N-2 + 2)= 6+5-1 +(5-1) (6-2+2)=34,$$
    where $ N=\max\left\lbrace |h(a)|,|h(b)|\right\rbrace=6$, $ k=5 $.
    All factors of length $\leq 34$ of $h^\omega(a)$ appear in $h^4(a)$.
    Hence, it is enough to check if the ancestors of $T_5$ are realized by a factor of $ h^4(a)$.  Since there is only one generation of ancestors of $T_5$, instead of checking if the ancestors are realized by a factor of $h^4(a)$, it is enough to check that if $h^5 (a)$ is abelian $5$-power free.	 Indeed, $h^5(a)$ is abelian $5$-power free, and this completes the proof.  The accompanying code of this proof can be found at \url{https://github.com/nihantanisali/Abelian-Power-Free-Morphisms/blob/main/Section5}.
\end{proof}

By Proposition \ref{bp-word} and Proposition \ref{p41}, we know that every 2-ball-box distribution admits a 5-BP. The next corrollary gives a complete resolution of the 2-ball-box distribution problem.

\begin{corollary}
	There is a 2-ball-box distribution that does not admit a 6-BP.
\end{corollary}
\begin{figure}[H]
	\centering
	\renewcommand{\arraystretch}{1.2} 
	\begin{tabular}{c}
		\hline
		 The first 24 terms of the infinite word produced by $h$ in Theorem \ref{5free}: $ h^{\omega}(a) $ \\
		\hline
        $ aaaabaaaabaaaabaaaababab $ \\ 
		\hline
	\end{tabular}
	\vspace{3mm} 
	
	\begin{center}
		\boxone{1}{1}
		\boxone{2}{2}
		\boxone{3}{3}
		\boxone{4}{4}
		\boxtwo{5}{5}
		\boxone{6}{7}
		\boxone{7}{8}
		\boxone{8}{9}
		\boxtwo{9}{10}
		\boxone{10}{12}
		\boxone{11}{13}
		\boxone{12}{14}
		\boxtwo{13}{15}
            \boxone{14}{17}
            \boxone{15}{18}
            \boxone{16}{19}
            \boxtwo{17}{20}
            \boxtwo{18}{22}
            \boxtwo{19}{24}
		\begin{tikzpicture}
		\draw[draw=white] (0.5,-0.5) rectangle ++ (1,1);
		\node[below] at (1,-0.5) {};
		\node[below] at (1,0.5) {\huge{$\cdots$}};
		\end{tikzpicture}
	\end{center}
	\caption{A 2-ball-box distribution with no 6-BP}
	\label{fig:enter-label5}
	
\end{figure}
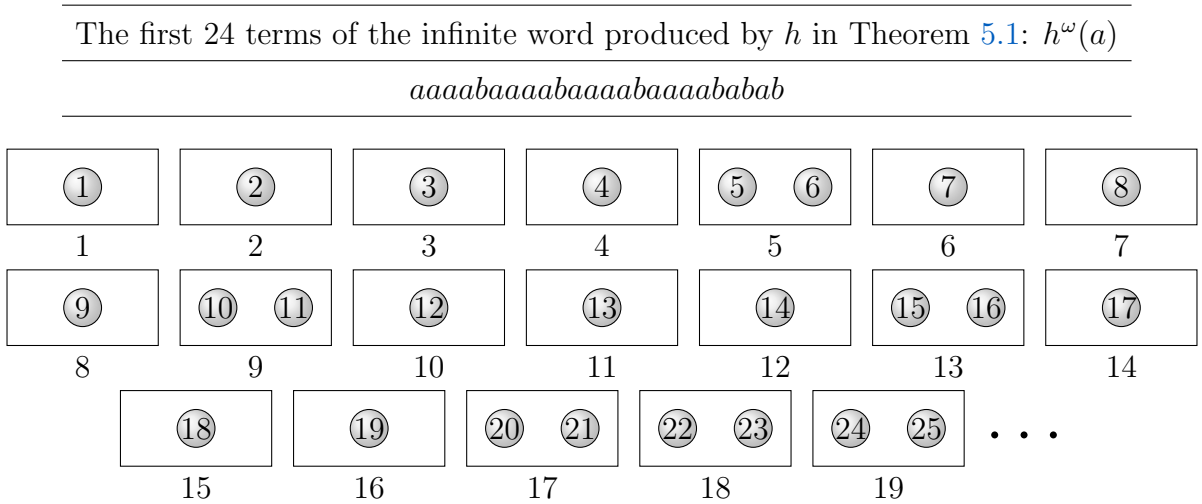	

\begin{remark}
    In Figure \ref{fig:enter-label5}, we permit placing exactly one ball in $ 3 $ consecutive boxes. This is because the corresponding infinite word has $aaaa$ as a factor, but not $ aaaaa $. By Proposition \ref{bp-word}, Proposition \ref{SievingT8} and Theorem \ref{T14-template}, we obtain the following consequences. There is a 2-ball-box distribution which admits a 9-BP but no 10-BP, and that does not have $3$ consecutive boxes in which there is exactly one ball. There is a 2-ball-box distribution which admits a 14-BP but no 15-BP placing exactly one ball into a maximum of $2$ consecutive boxes.
\end{remark}

\begin{remark}
Consider the morphism $h$ in Theorem \ref{5free}.
Let $G=\Z /  9 \Z$ and $f:\Sigma^*\to G$ be a monoid morphism such that
	$$ f(a)=\overline{2}, \ f(b)=\overline{1}. $$
Notice that $f$ vanishes on $h(\Sigma)$, and
it takes the following values when evaluated on $\operatorname{Pref}(h)$:

\begin{align*}
    P&: = f(\operatorname{Pref} (h) )=\{f(\epsilon), f(a),f(aa), f(aaa) ,f(aaaa), f(ab),f(aba),\\
    &f(abab),f(ababa)\}=\{\overline{0}, \overline{2},\overline{3},\overline{4},\overline{5}, \overline{6}, \overline{8} \}.
	\end{align*}
All the values in $P$ have a unique preimage in $\operatorname{Pref}(h)$ under $f$, except for $ \overline{6}$ and $\overline{8}$: 

\begin{align} 
    f(aaa)&=f(abab)=\overline{6}, \nonumber\\
    f(aaaa)&=f(ababa)=\overline{8}.\label{5rem}
\end{align}

We see that there are 12 non-trivial arithmetic progressions of length 6 in $P$, namely: \\

$(\overline{2},\overline{5}, \overline{8}, \overline{2},\overline{5}, \overline{8} ), (\overline{2},\overline{8}, \overline{5}, \overline{2},\overline{8}, \overline{5} ), (\overline{5},\overline{2}, \overline{8}, \overline{5},\overline{2}, \overline{8} ), (\overline{5},\overline{8}, \overline{2}, \overline{5},\overline{8}, \overline{2} ), (\overline{8},\overline{5}, \overline{2}, \overline{8},\overline{5}, \overline{2} ), (\overline{8},\overline{2}, \overline{5}, \overline{8},\overline{2}, \overline{5} ),$

$(\overline{0}, \overline{3}, \overline{6}, \overline{0}, \overline{3}, \overline{6}), (\overline{0}, \overline{6}, \overline{3}, \overline{0}, \overline{6}, \overline{3}), (\overline{3}, \overline{0}, \overline{6}, \overline{3}, \overline{0}, \overline{6}), (\overline{3}, \overline{6}, \overline{0}, \overline{3}, \overline{6}, \overline{0}), (\overline{6}, \overline{3}, \overline{0}, \overline{6}, \overline{3}, \overline{0}), (\overline{6}, \overline{0}, \overline{3}, \overline{6}, \overline{0}, \overline{3}).$
\\

By \eqref{5rem} and as in Proposition \ref{SievingT8} and Theorem \ref{T14-template}, we obtain 120 selected parent templates of $T_5$, and $h^\omega (a)$ contains an abelian $5$-power if and only if one of these selected parent templates is realized by a factor of length at most 34.
Therefore, this sieve argument reduces the search space from 16163 to 120 for parents.
However, we did not prefer this, and we applied the template method directly as our code terminates in several minutes.

\end{remark}
\section{Some Applications} \label{some_applications}
 \textbf{I.} There are some results known to be equivalent to van der Waerden's theorem. One of them is as follows (see \cite{Rabung}). If $  \left\lbrace a_1,a_2, \dots \right\rbrace  $ is an infinite sequence of integers satisfying $ 0 < a_{i+1} - a_i < r $ for all $i$ and for some $ r $, then the sequence contains arbitrarily long arithmetic progressions. So, we can ask the following question. 
 If $  \left\lbrace x_0,x_1,\dots \right\rbrace  $ is an   increasing sequence of real numbers such that the set
 $F= \left\lbrace x_{i+1}-x_i \ | \ i\ge 1\right\rbrace  $ of differences between consecutive terms is finite, then does the sequence $ \left\lbrace x_n \right\rbrace  $ contain arbitrarily long arithmetic progressions? Since there are infinite words over the binary alphabet $ \left\lbrace a,b\right\rbrace  $ which are abelian $ k $-power free for some $ k $, we quickly see that the answer is negative. If we symbolize the letter $ a $ as $ +1 $ increment and the letter $ b $ as $ +\sqrt{2} $ increment, then using Dekking's morphism \eqref{dekm}, we obtain a real sequence only containing arithmetic progressions (non-trivial) of length at most 4 with the desired property. It is important to choose the increments of the letters $ a $ and $ b $ to be linearly independent over $\mathbb Q$. In this case, the real sequence contains no non-trivial
$k$-AP if and only if the associated infinite word is abelian 
$(k-1)$-power free. Moreover, our result (Theorem \ref{5free}) indicates that if $F=\{1,\sqrt{2}\}$ and $x_{n+1}-x_n=\sqrt{2}$, then 
$x_{n+2}-x_{n+1}$ cannot be $\sqrt{2}$, then the answer is still not affirmative. There is such a real sequence which does not contain a non-trivial 6-AP by Theorem \ref{5free}. One may obtain similar results with more constraints by applying Theorem \ref{3 consecutive a 9AP} and Theorem \ref{T14-template}.\\
 
 \textbf{II.} Let a particle move on the grid
 $\{1,\ldots,N\} \times \{1,\ldots,N\}$ for a sufficiently large natural number $ N  $, starting at the point $ (1,1). $  We assume that the particle always moves forward (at least one of its coordinates increases). 
 In this case, is there any regularity in every walking choice of the particle? That is, to what extent does the walk on $\{1,\ldots,N\} \times \{1,\ldots,N\}$ 
 intersect with a line? 
 As in Application I, some walking choices of the particle lack regularity due to the existence of abelian power-free infinite words. As an example, we can again give the infinite word obtained from  Dekking's abelian 4-power free morphism $ h $ over the binary alphabet $\Sigma=\{a,b\}$:
 $ h(a)=abb$, $h(b)=aaab.$ If we represent the letter $ a $ as a step to the right and up (increment is the vector (1,1)) and the letter $ b  $ as a step to the right (increment is the vector (1,0)), then using this abelian $ 4 $-power free infinite word, we obtain a walking choice only containing arithmetic progressions of length at most $ 4 $. 
 An arithmetic progression of length $ k+1 $ on the grid $\{1,\ldots,N\} \times \{1,\ldots,N\}$ means that the inputs are simultaneously  arithmetic progressions of length $ k+1 $ on $ \left\lbrace 1,2,\dots,N\right\rbrace  $. In addition, if we impose new restrictions on the particle's trajectory, such as making it step in the $ (1,0) $ direction at most once, as in the box problem, then the regularity of its trajectory will increase relatively. If we consider the walking choice produced by the infinite word in Theorem \ref{5free}, then we obtain a walking choice which does not contain a non-trivial 6-AP. One can obtain similar results  by applying Theorem \ref{3 consecutive a 9AP} and Theorem \ref{T14-template} again. This is somewhat expected, as the density of such a sequence (the particle's trajectory) is 0, in fact it contains $O(N)$ elements in the grid $\{1,\ldots,N\} \times \{1,\ldots,N\}$. Thus, it is relatively small in this grid and it is a sparse subset. By the multi-dimensional Szemer\'edi theorem, if $A$ is a subset of $(\mathbb Z_{>0})^2$ with positive upper density, then it contains any finite configurations up to homothety, in particular it contains arbitrarily long arithmetic progressions. If our set is sparse, like the primes in the positive integers \cite{GreenTao}, then one can use the Relative Szemer\'edi theorem \cite{CFZ} for sparse sets if we have some randomness conditions. However, if a sparse set in $(\mathbb Z_{>0})^2$ has $O(N)$ elements in the grid $\{1,\ldots,N\} \times \{1,\ldots,N\}$, then it is generally far away from satisfying the conditions of the
 Relative Szemer\'edi theorem. Hence, once again such a sequence may not contain arbitrarily long arithmetic progressions.
 \begin{figure}[H]
\begin{tikzpicture}[xscale=0.6 , yscale=0.35,  every node/.style={font=\tiny}]
\draw[very thin, color=gray,opacity=0.2] (0,-1) grid (26,19);
\draw[->,color=black] (0,-1) -- (27,-1);

\foreach \x in {1,2,3,...,25}
\draw[shift={(\x,-1)},color=black] (0pt,2pt) -- (0pt,-2pt) node[below] 
{\footnotesize $\x$};
\draw[->,color=black] (0,-1) -- (0,20);

\foreach \y in {1,...,19}
\draw[shift={(0,\y-1)},color=black] (2pt,0pt) -- (-2pt,0pt) node[left] {\footnotesize $\y$};

\draw (26.4,-0.2) node {Ball};
\draw (+0.2,20) node[anchor=north west] {Box};
\draw[-,color=teal,very thick] (1,0) -- (5,4);
\draw[-,color=teal,very thick] (5,4) -- (6,4);
\draw[-,color=teal,very thick] (6,4) -- (10,8);
\draw[-,color=teal,very thick] (10,8) -- (11,8);
\draw[-,color=teal,very thick] (11,8) -- (15,12);
\draw[-,color=teal,very thick] (15,12) -- (16,12);
\draw[-,color=teal,very thick] (16,12) -- (20,16);
\draw[-,color=teal,very thick] (20,16) -- (21,16);
\draw[-,color=teal,very thick] (21,16) -- (22,17);
\draw[-,color=teal,very thick] (22,17) -- (23,17);
\draw[-,color=teal,very thick] (23,17) -- (24,18);
\draw[->,color=teal,very thick] (24,18) -- (25,18);

\draw[very thick,color=lightgray] (1,0) -- (21,16);
\foreach \y in {0,...,4}
\draw[shift={(5*\y+1, 4*\y)}]  node [circle,fill,gray, inner sep=1.5pt] { };
\end{tikzpicture}\\
 \caption{  Particle moves corresponding to the word   $  aaaab \ aaaab  \ aaaab  \ aaaab  \ abab   $ obtained in Theorem \ref{5free}  }
\end{figure}
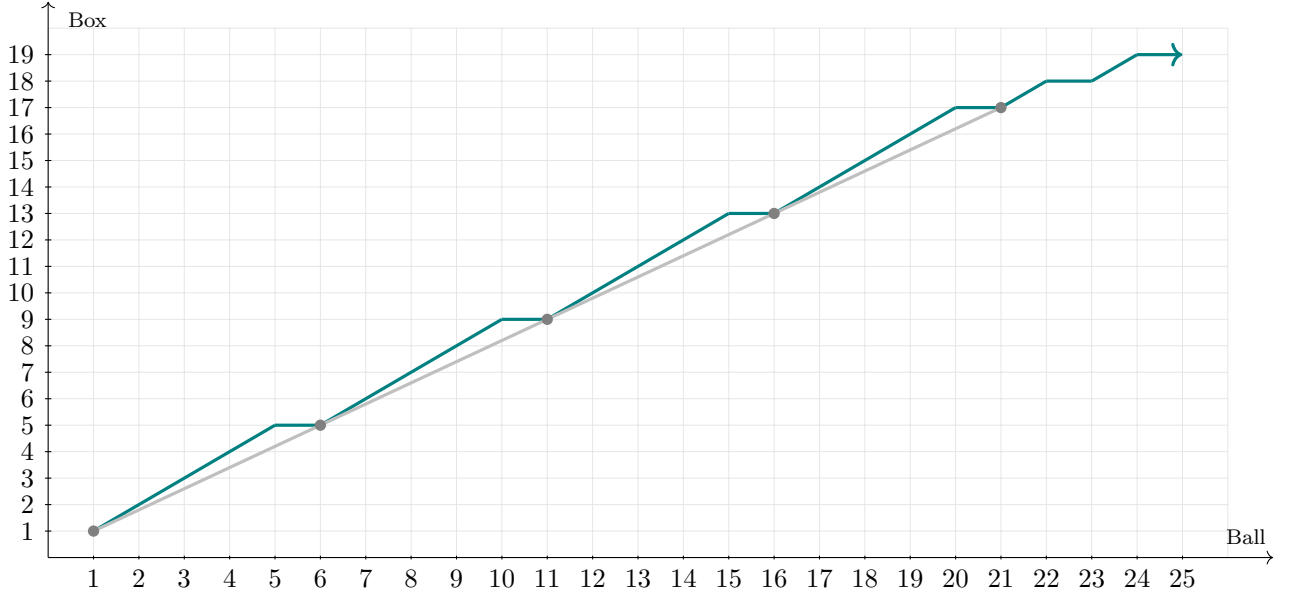 
\vspace{3mm}

 \textbf{III.} The concept of \emph{repetitive semigroups} originated from a well-known result in combinatorics, namely van der Waerden's theorem. This theorem laid the foundation for the study of repetitive structures in semigroups, where the concept of \emph{k-repetitiveness} was introduced.
 
 \begin{definition} \cite{prillo}
 	Let $ S $ be a semigroup, $ k \ge 2 $ be an integer, $ \Sigma $ be an alphabet, $ \phi : \Sigma^+ \to S$ be a morphism and $ w $ be a word of $ \Sigma^+ $, where the set of all non-empty words over $ \Sigma $ is denoted by $ \Sigma^+ $. If there exist $ x, w_1,\dots, w_k,y $ with $ w_i \in \Sigma^+, $ $ 1 \le i \le k , $ and $ x,y \in \Sigma^* $ such that 
 	$$ w =xw_1w_2 \dots w_ky $$
 	and 	$$ \phi (w_1)=\phi (w_2) = \dots =\phi (w_k), $$  	 	
 	then we say that $ w $ contains a $ k $-power modulo $ \phi. $
 	Moreover, if 
 	$$ |w_1|=|w_2| = \dots =|w_k|, $$
 	then  we say that $ w $ contains a uniform $ k $-power modulo $ \phi.$ Note that uniform $ k $-powers modulo $ \phi $ are usually called additive $ k $-powers.
 	The morphism $ \phi : \Sigma^+ \to S $ is said to be $ k $-repetitive (resp. uniform $ k $-repetitive) if there exists an integer $ R(\phi, k) $ such that any word $ w \in \Sigma^+ $ of length $ R(\phi, k) $ contains a $ k $-power modulo $ \phi $ (resp. a uniform $ k $-power modulo $ \phi $). If for each finite alphabet $ \Sigma $, each morphism $ \phi : \Sigma^+ \to S $ is $ k $-repetitive (resp. uniformly $ k $-repetitive), then the semigroup $ S $ is called $ k $-repetitive (resp. uniformly $ k $-repetitive). 	
 \end{definition}

 Recent results in the literature have determined whether some well-known semigroups are $ k $-repetitive for certain values of $  k $:
 
 \begin{itemize}
 	\item It was shown by Cassaigne et al. \cite{CassCurr} using the morphism $ \sigma_4 $ that $ \mathbb{Z} $ is not uniformly 3-repetitive.
 	\item Rao and Rosenfeld \cite{Rao-R} demonstrated that $ \mathbb{Z}^2 $ is not uniformly 2-repetitive using the morphism $ \sigma_6 $.
 \end{itemize}
 
  Observe that two words of the same length over $ \{0, 1\} $ have the same sum if and only if they are permutations of each other. Hence, using this observation, Dekking’s morphism over a binary alphabet shows that it is
  possible to avoid additive 4-powers. 
  Dekking's construction was extended in \cite{CMRS}, where they gave a decision algorithm as in the case of abelian powers (the template method). Andrade and Mol \cite{Andrade} implemented this decision algorithm and deduced the avoidability of abelian and additive powers in infinite rich words.
   In our work, we demonstrated that even in the case of morphisms belonging to a very specific class, such as the morphisms we defined in Theorems \ref{5free}, \ref{3 consecutive a 9AP} and \ref{T14-template}, the semigroup $ \mathbb{Z} $ cannot be $ k $-repetitive for certain values of $ k $ (5, 9 and 14, respectively) with more constraints. 
\section{Some Open Problems} \label{open problems}
\begin{enumerate}
\item Using Dekking's result (Theorem 
\ref{dekkinglemma}), can we find a morphism $h$ on the binary alphabet $\Sigma=\{a,b\}$ which is abelian 6-power free, and the fixed point $h^{\omega}(a)$ avoids $a^5$ and $b^2$?
\item Using Dekking's result (Theorem \ref{dekkinglemma}), can we find a morphism $h$ on the binary alphabet $\Sigma=\{a,b\}$ which is abelian 9-power free, and the fixed point $h^{\omega}(a)$ avoids $a^4$ and $b^2$?
\item Over a binary alphabet, if a morphism satisfies  $(C1)-(C3)$ from Theorem \ref{carpi}, does it also satisfy  $(O1)-(O3)$ from Theorem \ref{our result}? Under Carpi's conjecture, does the set of conditions in Theorem \ref{our result} give a complete characterization of abelian $ n $-power free morphisms defined on binary alphabets?
\item Can we find a morphism $h$ on the binary alphabet $\Sigma=\{a,b\}$ which is abelian 5-power free, and the fixed point $h^{\omega}(a)$ avoids $b^2$?
\item Can we find a morphism $h$ on the binary alphabet $\Sigma=\{a,b\}$ which is abelian $k$-power free for some $k$, and the fixed point $h^{\omega}(a)$ is $5$-power free and avoids $b^2$?
\end{enumerate}

\vspace{0.5cm}

\textbf{Acknowledgements:}
This work was partially supported by the Scientific and Technological Research Council of Turkey (TÜBİTAK) with the project number 122F027, and it was carried out by the second author. The first author was also partially supported by TÜBİTAK with the same project number. The second author's research is partially supported by the Science Academy's Young Scientist Award (BAGEP 2025). The third author is supported by the Horizon-Europe MSCA-DN project ENCODE.  
The authors would like to thank Narad Rampersad for sharing an initial Python implementation of the template method for abelian cubes, which made a significant contribution to the development of our work. We thank Hümeyra Akyüz for her valuable suggestions on the codes used in the article. We also thank Sarim Sarfraz for pointing out an issue in the statement of Theorem 2.9. Lastly, we are grateful to the  anonymous referees for their invaluable comments and suggestions, which immensely contributed to improving the quality of our manuscript.

\end{document}